\documentclass[12pt,thmsa]{article}
\usepackage{amssymb}
\usepackage{amsfonts}
\usepackage{cite}
\usepackage{amsmath}
\usepackage[top=3.8cm,bottom=3.8cm,textwidth=16cm,centering,a4paper]{geometry}

\numberwithin{equation}{section}
\newtheorem{theorem}{Theorem}[section]

\newtheorem{definition}[theorem]{Definition}

\newtheorem{lemma}[theorem]{Lemma}

\newtheorem{proposition}[theorem]{Proposition}
\newtheorem{remark}[theorem]{Remark}

\newenvironment{proof}[1][Proof]{\noindent\textbf{#1.} }{\hfill $\square$}

\begin{document}

\title{Existence and symmetry breaking of vectorial ground states for
Hartree-Fock type systems with potentials}
\date{}
\author{Juntao Sun$^{a}$, Tsung-fang Wu$^{b}$  \\
%EndAName
{\footnotesize $^a$\emph{School of Mathematics and Statistics, Shandong
University of Technology, Zibo, 255049, P.R. China }}\\
{\footnotesize $^{b}$\emph{Department of Applied Mathematics, National
University of Kaohsiung, Kaohsiung 811, Taiwan }}}
\maketitle

\begin{abstract}
In this paper we study the Hartree-Fock type system as follows:
\begin{equation*}
\left\{
\begin{array}{ll}
-\Delta u+V\left( x\right) u+\rho \left( x\right) \phi _{\rho ,\left(
u,v\right) }u=\left\vert u\right\vert ^{p-2}u+\beta \left\vert v\right\vert
^{\frac{p}{2}}\left\vert u\right\vert ^{\frac{p}{2}-2}u & \text{ in }\mathbb{%
R}^{3}, \\
-\Delta v+V\left( x\right) v+\rho \left( x\right) \phi _{\rho ,\left(
u,v\right) }v=\left\vert v\right\vert ^{p-2}v+\beta \left\vert u\right\vert
^{\frac{p}{2}}\left\vert v\right\vert ^{\frac{p}{2}-2}v & \text{ in }\mathbb{%
R}^{3},%
\end{array}%
\right.
\end{equation*}%
where $\phi _{\rho ,\left( u,v\right) }=\int_{\mathbb{R}^{3}}\frac{\rho
\left( y\right) \left( u^{2}(y)+v^{2}\left( y\right) \right) }{|x-y|}dy,$
the potentials $V(x),\rho (x)$ are positive continuous functions in $\mathbb{%
R}^{3},$ the parameter $\beta \in \mathbb{R}$ and $2<p<4$. Such system is
viewed as an approximation of the Coulomb system with two particles appeared
in quantum mechanics, whose main characteristic is the presence of the
double coupled terms. When $2<p<3,$ under suitable assumptions on
potentials, we shed some light on the behavior of the corresponding energy
functional on $H^{1}(\mathbb{R}^{3})\times H^{1}(\mathbb{R}^{3}),$ and prove
the existence of a global minimizer with negative energy. When $3\leq p<4,$
we find vectorial ground states by developing a new analytic method and
exploring the conditions on potentials. Finally, we study the phenomenon of
symmetry breaking of ground states when $2<p<3.$

\footnotetext{ ~\textit{E-mail addresses }: jtsun@sdut.edu.cn(J. Sun), tfwu@nuk.edu.tw (T.-F. Wu).}


\end{abstract}

\section{Introduction}

Our starting point is the following $(M+N)$-body Schr\"{o}dinger equation:%
\begin{equation*}
i\hbar \partial _{t}\Psi =-\frac{\hbar ^{2}}{2}\sum\limits_{j=1}^{M+N}\frac{1%
}{m_{j}}\Delta \Psi +\frac{e^{2}}{8\pi \varepsilon _{0}}\sum\limits_{j,k=1,%
\text{ }j\neq k}^{M+N}\frac{Z_{j}Z_{k}}{|x_{j}-x_{k}|}\Psi ,
\end{equation*}%
where $\Psi :\mathbb{R\times R}^{3(M+N)}\rightarrow \mathbb{C}$, the
constants $eZ_{j}$'s are the charges and in particular the charge numbers $%
Z_{j}$'s are positive for the nuclei and $-1$ for the electrons. Such system
originates from a molecular system made of $M$ nuclei interacting via the
Coulomb potential with $N$ electrons.

Its complexity led to use various approximations to describe the stationary
states with simpler models. Historically, the first effort made in this
direction began from Hartree \cite{H} by choosing some particular test
functions without considering the antisymmetry (i.e. the Pauli principle).
Subsequently, Fock \cite{F1} and Slater \cite{Sl1}, to take into account the
Pauli principle, proposed another class of test functions, i.e. the class of
Slater determinants, obtaining the following Hartree-Fock type system:%
\begin{equation}
\begin{array}{ll}
-\Delta \psi _{k}+V_{\text{ext}}\psi _{k}+\left( \int_{\mathbb{R}%
^{3}}|x-y|^{-1}\sum\limits_{j=1}^{N}|\psi _{j}(y)|^{2}dy\right) \psi
_{k}+(V_{\text{ex}}\psi )_{k}=E_{k}\psi _{k}, & \text{ }\forall k=1,2,...,N,%
\end{array}
\label{1-0}
\end{equation}%
where $\psi _{k}:\mathbb{R}^{3}\rightarrow \mathbb{C},$ $V_{\text{ext}}$ is
a given external potential, $(V_{\text{ex}}\psi )_{k}$ is the $k$'th
component of the crucial exchange potential defined by%
\begin{equation*}
\begin{array}{ll}
(V_{\text{ex}}\psi )_{k}=-\sum\limits_{j=1}^{N}\psi _{j}(y)\int_{\mathbb{R}%
^{3}}\frac{\psi _{k}(y)\bar{\psi}_{j}(y)}{|x-y|}dy, & \text{ }\forall
k=1,2,...,N,%
\end{array}%
\end{equation*}%
and $E_{k}$ is the $k$'th eigenvalue.

In \cite{Sl2}, a further relevant free-electron approximation for the $k$'th
component of the exchange potential $V_{\text{ex}}\psi $ is given by
\begin{equation}
\begin{array}{ll}
(V_{\text{ex}}\psi )_{k}=-C\left( \sum\limits_{j=1}^{N}|\psi
_{j}|^{2}\right) ^{1/3}\psi _{k}, & \text{ }\forall k=1,2,...,N,%
\end{array}
\label{1-1}
\end{equation}%
where $C$ is a positive constant. Such approximation is used more frequently
when the qualitative analysis of system (\ref{1-0}) is studied.

When $N=1,$ the exchange potential (\ref{1-1}) becomes $V_{\text{ex}}\psi
=-C|\psi _{1}|^{2/3}\psi _{1}$. If we consider $\psi _{1}$ as a real
function, renaming it as $u,$ and take, for simplicity, $C=1$, then system (%
\ref{1-0}) becomes Schr\"{o}dinger--Poisson--Slater equation as follows:

\begin{equation}
\begin{array}{ll}
-\Delta u+u+\mu \phi _{u}(x)u=|u|^{2/3}u & \text{ in }\mathbb{R}^{3},%
\end{array}
\label{1-3}
\end{equation}%
where $\mu >0$ is a parameter and%
\begin{equation*}
\phi _{u}(x)=\int_{\mathbb{R}^{3}}\frac{u^{2}(y)}{|x-y|}dy.
\end{equation*}%
S\'{a}nchez and Soler \cite{SS} used a minimization procedure in an
appropriate manifold to find a positive solution of equation (\ref{1-3}). If
the nonlinearity $|u|^{2/3}u$ is replaced with $|u|^{p-2}u$ $(2<p<6)$ (or,
more generally, $f(u))$, then equation (\ref{1-3}) becomes the Schr\"{o}%
dinger--Poisson equation (also called Schr\"{o}dinger--Maxwell equation),
where the existence and multiplicity of various solutions has been
extensively studied depending on the parameter $\mu $, see e.g. \cite%
{A,AP,R1,R2,SWF,SWF1,SWF2,ZZ}. More precisely, Ruiz \cite{R1} considered the
Schr\"{o}dinger--Poisson equation with the nonlinearity $|u|^{p-2}u$ $%
(2<p<6),$ and concluded that if $2<p\leq 3,$ one positive radial solutions
is found for $\mu >0$ sufficiently small and no nontrivial solution admits
for $\mu \geq 1/4$ while if $3<p<6,$ one positive radial solution admits for
all $\mu >0.$

Comparing with the case of $N=1,$ Hartree-Fock type system (\ref{1-0}) is
much more complicated when $N=2$. In order to further simplify it, we assume
the exchange potential%
\begin{equation}
V_{\text{ex}}\psi =-C\binom{(\left\vert \psi _{1}\right\vert ^{p-2}+\beta
\left\vert \psi _{1}\right\vert ^{\frac{p}{2}-2}\left\vert \psi
_{2}\right\vert ^{\frac{p}{2}})\psi _{1}}{(\left\vert \psi _{2}\right\vert
^{p-2}+\beta \left\vert \psi _{1}\right\vert ^{\frac{p}{2}}\left\vert \psi
_{2}\right\vert ^{\frac{p}{2}-2})\psi _{2}},  \label{1-4}
\end{equation}%
where $\beta \in \mathbb{R}$ and $2<p<6$. Clearly, if $p=\frac{8}{3},$ then (%
\ref{1-4}) is written as%
\begin{equation*}
V_{\text{ex}}\psi =-C\binom{(\left\vert \psi _{1}\right\vert ^{\frac{2}{3}%
}+\beta \left\vert \psi _{1}\right\vert ^{-\frac{2}{3}}\left\vert \psi
_{2}\right\vert ^{\frac{4}{3}})\psi _{1}}{(\left\vert \psi _{2}\right\vert ^{%
\frac{2}{3}}+\beta \left\vert \psi _{1}\right\vert ^{\frac{4}{3}}\left\vert
\psi _{2}\right\vert ^{-\frac{2}{3}})\psi _{2}},
\end{equation*}%
which can be viewed as an approximation of the exchange potential (\ref{1-1}%
) proposed by Slater. Considering $\psi _{1}$ and $\psi _{2}$ real
functions, renaming them as $u,v,$ and taking, for simplicity, $C=1$, system
(\ref{1-0}) becomes the following autonomous Hartree-Fock type system:%
\begin{equation}
\left\{
\begin{array}{ll}
-\Delta u+u+\mu \phi _{u,v}u=\left\vert u\right\vert ^{p-2}u+\beta
\left\vert v\right\vert ^{\frac{p}{2}}\left\vert u\right\vert ^{\frac{p}{2}%
-2}u & \text{ in }\mathbb{R}^{3}, \\
-\Delta v+v+\mu \phi _{u,v}v=\left\vert v\right\vert ^{p-2}v+\beta
\left\vert u\right\vert ^{\frac{p}{2}}\left\vert v\right\vert ^{\frac{p}{2}%
-2}v & \text{ in }\mathbb{R}^{3},%
\end{array}%
\right.  \label{1-2}
\end{equation}%
where $\mu \in \mathbb{R}$ and $\phi _{u,v}(x)=\int_{\mathbb{R}^{3}}\frac{%
u^{2}(y)+v^{2}\left( y\right) }{|x-y|}dy.$ For a solution $\left( u,v\right)
$ of system (\ref{1-2}), we introduce some concepts of its triviality and
positiveness.

\begin{definition}
\label{D1}A vector function $\left( u,v\right) $ is said to be\newline
$\left( i\right) $ nontrivial if either $u\neq 0$ or $v\neq 0;$\newline
$\left( ii\right) $ semitrivial if it is nontrivial but either $u=0$ or $%
v=0; $\newline
$\left( iii\right) $ vectorial if both of $u$ and $v$ are not zero;\newline
$\left( iv\right) $ nonnegative if $u\geq 0$ and $v\geq 0;$\newline
$\left( v\right) $ positive if $u>0$ and $v>0.$
\end{definition}

If $\mu =0,$ then system (\ref{1-2}) is deduced to the local weakly coupled
nonlinear Schr\"{o}dinger system%
\begin{equation}
\left\{
\begin{array}{ll}
-\Delta u+u=\left\vert u\right\vert ^{p-2}u+\beta \left\vert v\right\vert ^{%
\frac{p}{2}}\left\vert u\right\vert ^{\frac{p}{2}-2}u & \text{ in }\mathbb{R}%
^{3}, \\
-\Delta v+v=\left\vert v\right\vert ^{p-2}v+\beta \left\vert u\right\vert ^{%
\frac{p}{2}}\left\vert v\right\vert ^{\frac{p}{2}-2}v & \text{ in }\mathbb{R}%
^{3},%
\end{array}%
\right.  \label{1-5}
\end{equation}%
which arises in various physical phenomena, such as the occurrence of
phase-separation in Bose-Einstein condensates with multiple states, or the
propagation of mutually incoherent wave packets in nonlinear optics, see
e.g. \cite{AA,EGBB,F,M,T1}. The parameter $\beta $ is the interaction
between the two components. the positive sign of $\beta $ stays for
attractive interaction, while the negative sign stays for repulsive
interaction. In recent years, the existence and multiplicity of positive
vectorial solutions for system (\ref{1-5}) have been extensively studied,
depending on the range of $\beta $, see e.g. \cite{AC,BDW,CZ2,LW,LW1,MMP,S0}%
. Furthermore, for the study of nonlinear Schr\"{o}dinger systems with an
external potential $V(x),$ we refer the reader to \cite{LW2,LL,WWZ}.

If $\mu >0,$ then the typical characteristic of system (\ref{1-2}) lies on
the presence of the double coupled terms, including a Coulomb interacting
one and a cooperative pure power one. Very recently, d'Avenia, Maia and
Siciliano \cite{DMS} firstly studied the existence of radial vectorial
solutions to system (\ref{1-2}) when either $3<p<4$ and $\beta >0,$ or $%
4\leq p<6$ and $\beta \geq 2^{2q-1}-1.$ The proof is based on the method of
Nehari-Pohozaev manifold defined in $\mathbf{H}_{r}:=H_{rad}^{1}(\mathbb{R}%
^{3})\times H_{rad}^{1}(\mathbb{R}^{3})$ developed by Ruiz \cite{R1}. Later,
we \cite{SW} considered another interesting case, i.e. $2<p<3,$ and proved
the existence of radial vectorial solutions for $\left( HF_{\beta }\right) $
when either $\mu >0$ small enough or $\beta >\beta (\mu ).$ In addition, for
$\beta >\beta (\mu ),$ we also obtained vectorial ground states in $\mathbf{H%
}:=H^{1}(\mathbb{R}^{3})\times H^{1}(\mathbb{R}^{3})$ when either $3\leq p<4$
and $0<\mu <\mu _{0},$ or $\frac{1+\sqrt{73}}{3}\leq p<4$ and $\mu >0.$

In this paper we study the existence of vectorial solutions for the
Hartree-Fock type system with potentials

\begin{equation}
\left\{
\begin{array}{ll}
-\Delta u+V\left( x\right) u+\rho \left( x\right) \phi _{\rho ,\left(
u,v\right) }u=\left\vert u\right\vert ^{p-2}u+\beta \left\vert v\right\vert
^{\frac{p}{2}}\left\vert u\right\vert ^{\frac{p}{2}-2}u & \text{ in }\mathbb{%
R}^{3}, \\
-\Delta v+V\left( x\right) v+\rho \left( x\right) \phi _{\rho ,\left(
u,v\right) }v=\left\vert v\right\vert ^{p-2}v+\beta \left\vert u\right\vert
^{\frac{p}{2}}\left\vert v\right\vert ^{\frac{p}{2}-2}v & \text{ in }\mathbb{%
R}^{3},%
\end{array}%
\right.  \tag*{$\left( HF_{\beta }\right) $}
\end{equation}%
where $\beta \in \mathbb{R}$, $2<p<6$ and%
\begin{equation*}
\phi _{\rho ,\left( u,v\right) }(x)=\int_{\mathbb{R}^{3}}\frac{\rho \left(
x\right) (u^{2}(y)+v^{2}\left( y\right) )}{|x-y|}dy.
\end{equation*}%
In addition, $V(x)$ and $\rho (x)$ represents an external potential and a
charge potential, respectively, which satisfy the following assumptions:%
\newline
$(D1)$ $V\in C(\mathbb{R}^{3},\mathbb{R})$ with $\lambda :=\inf_{x\in
\mathbb{R}^{3}}V\left( x\right) >0$ and $\lim\limits_{|x|\rightarrow \infty
}V(x)=V_{\infty }>0;$\newline
$(D2)$ $\rho \in C(\mathbb{R}^{3},\mathbb{R})$ with $\rho _{\min
}:=\inf_{x\in \mathbb{R}^{3}}\rho \left( x\right) >0$ on $\mathbb{R}^{3}$
and $\lim\limits_{|x|\rightarrow \infty }\rho (x)=\rho _{\infty }>0.$

Our approach is variational, and we look for solutions of system $\left(
HF_{\beta }\right) $ as critical points of the associated energy functional $%
J_{\beta }:\mathbf{H}\rightarrow \mathbb{R}$ defined by%
\begin{equation*}
J_{\beta }(u,v)=\frac{1}{2}\left\Vert \left( u,v\right) \right\Vert ^{2}+%
\frac{1}{4}\int_{\mathbb{R}^{3}}\rho \left( x\right) \phi _{\rho ,\left(
u,v\right) }\left( u^{2}+v^{2}\right) dx-\frac{1}{p}\int_{\mathbb{R}%
^{3}}\left( \left\vert u\right\vert ^{p}+\left\vert v\right\vert ^{p}+2\beta
\left\vert u\right\vert ^{\frac{p}{2}}\left\vert v\right\vert ^{\frac{p}{2}%
}\right) dx,
\end{equation*}%
where $\left\Vert \left( u,v\right) \right\Vert =\left[ \int_{\mathbb{R}%
^{3}}(|\nabla u|^{2}+|\nabla v|^{2}+V(x)u^{2}+V(x)v^{2})dx\right] ^{1/2}$ is
the standard norm in $\mathbf{H}.$ Clearly, $J_{\beta }$ is a well-defined
and $C^{1}$ functional on $\mathbf{H}.$

Observe that the case of $4\leq p<6$ is trivial, and we only consider the
case of $2<p<4$ in the whole paper.

We point out that the geometric properties of $J_{\beta }$ have not been
studied in detail in existing papers \cite{DMS,SW}. One of the motivations
of this paper is to shed some light on the behavior of $J_{\beta }$ on $%
\mathbf{H}$ when $2<p<3$. As we shall see, under conditions $(D1)-(D2),$
system $\left( HF_{\beta }\right) $ does not admits any nontrivial solutions
for $\beta $ sufficiently small, even positive, while $J_{\beta }$ is not
bounded from below on $\mathbf{H}\ $for $\beta >0$ sufficiently large.
However, under additional assumptions on $V,\rho $ (condition $(D3)$ below),
$J_{\beta }$ is bounded from below on $\mathbf{H}$ when $\beta $ lies in a
finite interval\textbf{,} and we prove the existence of a global minimizer
with negative energy, which seems to be an interesting result. In addition,
we also find that the existence of vectorial solutions is subject to the
parameter $\beta $, regardless of size of the charge potential, which is
totally different from the Schr\"{o}dinger--Poisson equation, e.g. \cite{R1}.

For the autonomous system (\ref{1-2}) with $3\leq p<4$, we have concluded
that the existence of vectorial ground states depends on both $\mu >0$
sufficiently small and $\beta >0$ sufficiently large in \cite{SW}. For the
nonautonomous system $\left( HF_{\beta }\right) $ with $3\leq p<4$, in this
paper we are likewise interested in finding vectorial ground states by
developing a new analytic method and exploring the conditions on $V,\rho .$
However, unlike the autonomous system, we shall focus on relaxing the
restriction on the parameters, which is the second aim of our paper.

Note that nonradial ground states have been found in other frameworks
previously, such as in the Schr\"{o}dinger--Poisson equation, see \cite%
{R2,Wu}, where the Coulomb interacting term need to be controlled. However,
to the best of our knowledge, there seems no any related result for the
Hartree-Fock type system like $\left( HF_{\beta }\right) $. In view of this,
the last object of this paper is to study the existence of nonradial ground
states for a class of Hartree-Fock type systems with potentials. As we shall
see, symmetry breaking phenomenon could occur when the cooperative pure
power term is controlled, not the Coulomb interacting one.

\subsection{Main results}

First of all, we consider the following minimization problem:%
\begin{equation}
\Lambda \left( \theta ,k\right) :=\inf_{u\in \mathbf{H}_{r}\setminus \left\{
\left( 0,0\right) \right\} }\frac{\frac{1}{2}\left\Vert \left( u,v\right)
\right\Vert _{\theta }^{2}+\frac{k}{4}\int_{\mathbb{R}^{3}}\phi _{k,\left(
u,v\right) }\left( u^{2}+v^{2}\right) dx-\frac{1}{p}\int_{\mathbb{R}%
^{3}}(\left\vert u\right\vert ^{p}+\left\vert v\right\vert ^{p})dx}{\frac{2}{%
p}\int_{\mathbb{R}^{3}}\left\vert u\right\vert ^{\frac{^{p}}{2}}\left\vert
v\right\vert ^{\frac{p}{2}}dx}.  \label{2-17}
\end{equation}%
where $\theta ,k>0$ are constants and $\left\Vert \left( u,v\right)
\right\Vert _{\theta }^{2}=\int_{\mathbb{R}^{3}}\left( |\nabla u|^{2}+\theta
u^{2}+|\nabla v|^{2}+\theta v^{2}\right) dx.$ Then we have the following
proposition.

\begin{proposition}
\label{p1}Assume that $2<p<3$ and $\theta ,k$ are positive constants. Then
we have\newline
$(i)$ $\Lambda \left( \theta ,k\right) \geq \frac{p}{2}\left( \frac{\theta }{%
3-p}\right) ^{3-p}\left( \frac{k}{p-2}\right) ^{p-2}-1;$\newline
$(ii)$ $\Lambda \left( \theta ,k\right) $ is achieved, i.e. there exists $%
\left( u_{0},v_{0}\right) \in \mathbf{H}_{r}\setminus \left\{ \left(
0,0\right) \right\} $ such that
\begin{equation*}
\Lambda \left( \theta ,k\right) =\frac{\frac{1}{2}\left\Vert \left(
u_{0},v_{0}\right) \right\Vert _{\theta }^{2}+\frac{k}{4}\int_{\mathbb{R}%
^{3}}\phi _{k,\left( u_{0},v_{0}\right) }\left( u_{0}^{2}+v_{0}^{2}\right)
dx-\frac{1}{p}\int_{\mathbb{R}^{3}}(\left\vert u_{0}\right\vert
^{p}+\left\vert v_{0}\right\vert ^{p})dx}{\frac{2}{p}\int_{\mathbb{R}%
^{3}}\left\vert u_{0}\right\vert ^{\frac{^{p}}{2}}\left\vert
v_{0}\right\vert ^{\frac{p}{2}}dx}>0.
\end{equation*}
\end{proposition}

Then we have the following results.

\begin{theorem}
\label{t1-2}Assume that $2<p<3$ and conditions $\left( D1\right) -\left(
D2\right) $ hold. Then we have the following statements.\newline
$\left( i\right) $ there exists $\overline{\Lambda }_{0}\geq 2^{\frac{p-2}{2}%
}\left( \frac{\lambda }{3-p}\right) ^{3-p}\left( \frac{\rho _{\min }}{p-2}%
\right) ^{p-2}-1$ such that for every $\beta \in \left( -\infty ,\overline{%
\Lambda }_{0}\right) ,$ system $(HF_{\beta })$ does not admits any
nontrivial solutions;\newline
$\left( ii\right) $ for every $\beta >\Lambda \left( V_{\max },\rho _{\max
}\right) ,$ there holds $\inf_{\left( u,v\right) \in \mathbf{H}}J_{\beta
}(u,v)=-\infty ,$ where $V_{\max }:=\sup_{x\in \mathbb{R}^{3}}V\left(
x\right) $ and $\rho _{\max }:=\sup_{x\in \mathbb{R}^{3}}\rho \left(
x\right) .$
\end{theorem}

\begin{theorem}
\label{t1}Assume that $2<p<3$ and conditions $\left( D1\right) -\left(
D2\right) $ hold. In addition, the following assumption also holds:\newline
$\left( D3\right) $ there exist two positive constants $\lambda _{0},k_{0}$
such that%
\begin{equation*}
V_{\infty }^{3-p}\rho _{\infty }^{p-2}>\frac{2^{\frac{6-p}{2}}}{p}\left(
p-2\right) ^{p-2}\left( 3-p\right) ^{3-p}\left( 1+\Lambda \left( \lambda
_{0},k_{0}\right) \right)
\end{equation*}%
and%
\begin{eqnarray*}
&&\frac{1}{2}\int_{\mathbb{R}^{3}}V\left( x\right) \left(
u_{0}^{2}+v_{0}^{2}\right) dx+\frac{1}{4}\int_{\mathbb{R}^{3}}\rho \left(
x\right) \phi _{\rho ,\left( u_{0},v_{0}\right) }\left(
u_{0}^{2}+v_{0}^{2}\right) dx \\
&<&\frac{1}{2}\int_{\mathbb{R}^{3}}\lambda _{0}\left(
u_{0}^{2}+v_{0}^{2}\right) dx+\frac{k_{0}}{4}\int_{\mathbb{R}^{3}}\phi
_{k_{0},\left( u,v\right) }\left( u_{0}^{2}+v_{0}^{2}\right) dx,
\end{eqnarray*}%
where $\left( u_{0},v_{0}\right) \in \mathbf{H}_{r}\setminus \left\{ \left(
0,0\right) \right\} $ is a minimizer of the minimization problem (\ref{2-17}%
) in Proposition \ref{p1} with $\theta =\lambda _{0}$ and $k=k_{0}.$\newline
Then for every%
\begin{equation*}
\Lambda \left( \lambda _{0},k_{0}\right) <\beta <\frac{p}{2^{\frac{6-p}{2}}}%
\left( \frac{V_{\infty }}{3-p}\right) ^{3-p}\left( \frac{\rho _{\infty }}{p-2%
}\right) ^{p-2}-1,
\end{equation*}%
system $(HF_{\beta })$ admits a vectorial ground state solution $\left(
u_{\beta }^{\left( 1\right) },v_{\beta }^{\left( 1\right) }\right) \in
\mathbf{H}$ satisfying $J_{\beta }\left( u_{\beta }^{\left( 1\right)
},v_{\beta }^{\left( 1\right) }\right) =\inf_{\left( u,v\right) \in \mathbf{H%
}}J_{\beta }(u,v)<0.$
\end{theorem}

\begin{remark}
\label{R1}By conditions $\left( D1\right) -\left( D3\right) ,$ for every%
\begin{equation*}
\Lambda \left( \lambda _{0},k_{0}\right) <\beta <\frac{p}{2^{\frac{6-p}{2}}}%
\left( \frac{V_{\infty }}{3-p}\right) ^{3-p}\left( \frac{\rho _{\infty }}{p-2%
}\right) ^{p-2}-1,
\end{equation*}%
the sublevel set
\begin{equation*}
\mathcal{B}:=\left\{ x\in \mathbb{R}^{3}:V^{3-p}\left( x\right) \rho
^{p-2}\left( x\right) <\frac{2^{\frac{6-p}{2}}}{p}\left( p-2\right)
^{p-2}\left( 3-p\right) ^{3-p}\left( 1+\beta \right) \right\}
\end{equation*}%
is nonempty and $0<\left\vert \mathcal{B}\right\vert <\infty .$
\end{remark}

\begin{theorem}
\label{t2}Assume that $2<p<4$ and conditions $\left( D1\right) -\left(
D2\right) $ hold. In addition, we assume that\newline
$\left( D4\right) $ $V_{\infty }=V_{\max }\geq \lambda $ and $\rho _{\infty
}=\rho _{\max }\geq \rho _{\min }$ with strictly inequalities hold on a
positive measure set.\newline
Then there exists $\beta _{0}>\frac{p-2}{2}$ such that for every $\beta
>\beta _{0},$ system $(HF_{\beta })$ admits a vectorial positive solution $%
\left( u_{\beta }^{\left( 1\right) },v_{\beta }^{\left( 1\right) }\right)
\in \mathbf{H}$ satisfying $J_{\beta }\left( u_{\beta }^{\left( 1\right)
},v_{\beta }^{\left( 1\right) }\right) >0.$
\end{theorem}

Next, we establish the existence of vectorial ground state solution of
system $(HF_{\beta }).$\newline
$\left( D5\right) $ $V\in C^{1}(\mathbb{R}^{3})\cap L^{\infty }(\mathbb{R}%
^{3})$ and $2V\left( x\right) +\langle \nabla V(x),x\rangle \geq d_{0}$ for
all $x\in \mathbb{R}^{3}$ and for some $d_{0}>0.$\newline
$\left( D6\right) $ $\rho \in C^{1}(\mathbb{R}^{3})\cap L^{\infty }(\mathbb{R%
}^{3})$ and $\langle \nabla \rho (x),x\rangle +\frac{2\left( p-3\right) }{p-2%
}\rho \left( x\right) \geq 0$ for all $x\in \mathbb{R}^{3}.$

\begin{theorem}
\label{t3}Assume that $3\leq p<4$ and conditions $\left( D1\right) -\left(
D2\right) $ and $\left( D4\right) -\left( D6\right) $ hold. Then the
following statements are true. \newline
$\left( i\right) $ If $3\leq p<\frac{1+\sqrt{73}}{3}$, then there exists $%
\widehat{\beta }_{0}\geq \beta _{0}$ such that for every $\beta >\widehat{%
\beta }_{0},$ system $(HF_{\beta })$ admits a vectorial ground state
solution $\left( u_{\beta },v_{\beta }\right) \in \mathbf{H}$ satisfying $%
J_{\beta }\left( u_{\beta },v_{\beta }\right) >0.$\newline
$\left( ii\right) $ If $\frac{1+\sqrt{73}}{3}\leq p<4,$ then for every $%
\beta >\beta _{0},$ system $(HF_{\beta })$ admits a vectorial ground state
solution $\left( u_{\beta },v_{\beta }\right) \in \mathbf{H}$ satisfying $%
J_{\beta }\left( u_{\beta },v_{\beta }\right) >0.$
\end{theorem}

To study the existence of nonradial ground state solutions, we consider the
following system:%
\begin{equation}
\left\{
\begin{array}{ll}
-\Delta u+V_{\varepsilon }\left( x\right) u+\rho _{\varepsilon }\left(
x\right) \phi _{\rho ,\left( u,v\right) }u=\left\vert u\right\vert
^{p-2}u+\beta \left\vert v\right\vert ^{\frac{p}{2}}\left\vert u\right\vert
^{\frac{p}{2}-2}u & \text{ in }\mathbb{R}^{3}, \\
-\Delta v+V_{\varepsilon }\left( x\right) v+\rho _{\varepsilon }\left(
x\right) \phi _{\rho ,\left( u,v\right) }v=\left\vert v\right\vert
^{p-2}v+\beta \left\vert u\right\vert ^{\frac{p}{2}}\left\vert v\right\vert
^{\frac{p}{2}-2}v & \text{ in }\mathbb{R}^{3},%
\end{array}%
\right.  \tag*{$\left( HF_{\varepsilon ,\beta }\right) $}
\end{equation}%
where $V_{\varepsilon }(x)=V(\varepsilon x)$ and $\rho _{\varepsilon
}(x)=\rho (\varepsilon x)$ for any $\varepsilon >0$. Then we have the
following result.

\begin{theorem}
\label{th1.4}Assume that $2<p<3$ and conditions $\left( D1\right) -\left(
D2\right) $ hold. In addition, we assume that\newline
$(D7)$ $V(x)=V(|x|)$ and $\rho (x)=\rho (|x|);$\newline
$(D8)$ there exist a point $x_{0}\in \mathbb{R}^{3}$ and two positive
constants $\lambda _{0},k_{0}$ such that%
\begin{equation*}
x_{0}\in \mathcal{D}:=\left\{ x\in \mathbb{R}^{3}:V_{\min }<V(x)<\lambda _{0}%
\text{ and }\rho _{\min }<\rho (x)<k_{0}\right\}
\end{equation*}%
and
\begin{equation*}
V_{\infty }^{3-p}\rho _{\infty }^{p-2}>\frac{2^{\frac{6-p}{2}}}{p}\left(
p-2\right) ^{p-2}\left( 3-p\right) ^{3-p}\left( 1+\Lambda _{0}\left( \lambda
_{0},k_{0}\right) \right) .
\end{equation*}%
Then for every%
\begin{equation*}
\Lambda \left( \lambda _{0},k_{0}\right) <\beta <\frac{p}{2^{\frac{6-p}{2}}}%
\left( \frac{V_{\infty }}{3-p}\right) ^{3-p}\left( \frac{\rho _{\infty }}{p-2%
}\right) ^{p-2}-1,
\end{equation*}%
system $\left( HF_{\varepsilon ,\beta }\right) $ admits a vectorial positive
solution $\left( u_{\varepsilon },v_{\varepsilon }\right) \in \mathbf{H}$
such that
\begin{equation*}
J_{\varepsilon ,\beta }\left( u_{\varepsilon },v_{\varepsilon }\right)
<-K\leq \inf\limits_{u\in \mathbf{H}_{r}}J_{\varepsilon ,\beta }(u,v)\text{
for some constant }K>0\text{ and }\varepsilon >0\text{ sufficiently small,}
\end{equation*}%
where $J_{\varepsilon ,\beta }=J_{\beta }$ for $V(x)=V_{\varepsilon }(x)$
and $\rho (x)=\rho _{\varepsilon }(x).$ Furthermore, $\left( u_{\varepsilon
},v_{\varepsilon }\right) $ is a nonradial vectorial ground state solution
of system $\left( HF_{\varepsilon ,\beta }\right) .$
\end{theorem}

\begin{remark}
\label{r1.2}Assume that conditions $(D1)-(D2)$ and $(D8)$ hold. Let $\left(
u_{0},v_{0}\right) \in \mathbf{H}_{r}\setminus \left\{ \left( 0,0\right)
\right\} $ be a minimizer of the minimization problem (\ref{2-17}) obtained
in Proposition \ref{p1} with $\theta =\lambda _{0}$ and $k=k_{0}.$ Define $%
u_{\varepsilon }(x)=u_{0}(x-\frac{x_{0}}{\varepsilon })$ and $v_{\varepsilon
}(x)=v_{0}(x-\frac{x_{0}}{\varepsilon })$, where $x_{0}\in \mathbb{R}^{3}$
as in condition $\left( D8\right) .$ Then it follows from conditions $%
(D1)-\left( D2\right) $ that%
\begin{eqnarray*}
\int_{\mathbb{R}^{3}}V_{\varepsilon }(u_{\varepsilon }^{2}+v_{\varepsilon
}^{2})dx &=&\int_{\mathbb{R}^{3}}V(\varepsilon
x+x_{0})(u_{0}^{2}+v_{0}^{2})dx \\
&=&\int_{\mathbb{R}^{3}}V(x_{0})(u_{0}^{2}+v_{0}^{2})dx+o(\varepsilon ) \\
&<&\lambda _{0}\int_{\mathbb{R}^{3}}(u_{0}^{2}+v_{0}^{2})dx\text{ for }%
\varepsilon >0\text{ sufficiently small,}
\end{eqnarray*}%
and
\begin{eqnarray*}
\int_{\mathbb{R}^{3}}\rho _{\varepsilon }(x)\phi _{\rho _{\varepsilon
},\left( u_{\varepsilon },v_{\varepsilon }\right) }\left( u_{\varepsilon
}^{2}+v_{\varepsilon }^{2}\right) dx &=&\int_{\mathbb{R}^{3}}\rho
(\varepsilon x+x_{0})\phi _{\rho (\varepsilon x+x_{0}),\left(
u_{0},v_{0}\right) }\left( u_{0}^{2}+v_{0}^{2}\right) dx \\
&=&\int_{\mathbb{R}^{2}}\rho (x_{0})\phi _{\rho (x_{0}),\left(
u_{0},v_{0}\right) }\left( u_{0}^{2}+v_{0}^{2}\right) dx+o(\varepsilon ) \\
&<&\int_{\mathbb{R}^{2}}k_{0}\phi _{k_{0},\left( u_{0},v_{0}\right) }\left(
u_{0}^{2}+v_{0}^{2}\right) dx\text{ for }\varepsilon >0\text{ sufficiently
small.}
\end{eqnarray*}%
These imply that when $V(x)$ and $\rho \left( x\right) $ are replaced by $%
V(\varepsilon x)$ and $\rho \left( \varepsilon x\right) ,$ respectively,
condition $(D3)$ still holds for $\varepsilon >0$ sufficiently small.
Therefore, by Theorem \ref{t1}, system $(HF_{\varepsilon ,\beta })$ admits a
vectorial ground state solution $\left( u_{\varepsilon ,\beta }^{\left(
1\right) },v_{\varepsilon ,\beta }^{\left( 1\right) }\right) \in \mathbf{H}$
satisfying $J_{\varepsilon ,\beta }\left( u_{\varepsilon ,\beta }^{\left(
1\right) },v_{\varepsilon ,\beta }^{\left( 1\right) }\right) <J_{\beta
}^{\infty }(u_{0},v_{0})<0.$
\end{remark}

The rest of this paper is organized as follows. After introducing some
preliminary results in Section 2, we establish two natural constraints in
Section 3. In Sections 4 and 5, we prove Theorems \ref{t1-2}, \ref{t1}, \ref{t2} and \ref{t3}. Finally, we
give the proof of Theorem \ref{th1.4} in Section 6.

\section{Preliminary results}

For sake of convenience, we set%
\begin{equation*}
F_{\beta }\left( u,v\right) :=|u|^{p}+\left\vert v\right\vert ^{p}+2\beta
|u|^{\frac{^{p}}{2}}\left\vert v\right\vert ^{\frac{p}{2}}.
\end{equation*}%
Then the energy functional $J_{\beta }$ is rewritten as%
\begin{equation*}
J_{\beta }(u,v)=\frac{1}{2}\left\Vert \left( u,v\right) \right\Vert ^{2}+%
\frac{1}{4}\int_{\mathbb{R}^{3}}\rho \left( x\right) \phi _{\rho ,\left(
u,v\right) }\left( u^{2}+v^{2}\right) dx-\frac{1}{p}\int_{\mathbb{R}%
^{3}}F_{\beta }\left( u,v\right) dx.
\end{equation*}%
Now we establish some estimates.

\begin{lemma}
\label{g3}Let $2<p<3$ and $\beta ,c,d>0.$ Let $f_{c,d}\left( s\right) =\frac{%
d}{4}+\sqrt{\frac{1}{8}}cs-\frac{1+\beta }{p}s^{p-2}$ for $s\geq 0.$ Then
there exist
\begin{equation*}
d_{c}:=\left( 3-p\right) \left( \frac{4\left( 1+\beta \right) }{p}\right)
^{1/\left( 3-p\right) }\left( \frac{p-2}{\sqrt{2}c}\right) ^{\left(
p-2\right) /\left( 3-p\right) }>0
\end{equation*}%
and
\begin{equation*}
s_{c}:=\left( \frac{\sqrt{8}\left( 1+\beta \right) \left( p-2\right) }{pc}%
\right) ^{1/\left( 3-p\right) }>0
\end{equation*}%
such that\newline
$\left( i\right) $ $f_{c,d}^{\prime }\left( s_{c}\right) =0$ and $%
f_{c,d_{c}}\left( s_{c}\right) =0;$\newline
$\left( ii\right) $ for each $d<d_{c},$ there exist $\eta _{d},\xi _{d}>0$
such that $\eta _{d}<s_{c}<\xi _{d}$ and $f_{c,d}\left( s\right) <0$ for all
$s\in \left( \eta _{d},\xi _{d}\right) ;$\newline
$\left( iii\right) $ for each $d>d_{c},$ $f_{c,d}\left( s\right) >0$ for all
$s>0.$
\end{lemma}

\begin{proof}
The proof is straightforward, and we omit it here.
\end{proof}

\begin{lemma}
\label{L2-1}Let $2<p<4$ and $\beta >0.$ Let $g_{\beta }\left( s\right) =s^{%
\frac{p}{2}}+\left( 1-s\right) ^{\frac{p}{2}}+2\beta s^{\frac{p}{4}}\left(
1-s\right) ^{\frac{p}{4}}$ for $s\in \left[ 0,1\right] .$ Then there exists $%
s_{\beta }\in \left( 0,1\right) 0<s_{\beta }<1$ such that $g_{\beta }\left(
s_{\beta }\right) =\max_{s\in \left[ 0,1\right] }g_{\beta }\left( s\right)
>1.$ In particular, if $\beta \geq \frac{p-2}{2},$ then $s_{\beta }=\frac{1}{%
2}.$
\end{lemma}

\begin{proof}
The proof is similar to that in \cite[Lemma 2.4]{DMS}, and we omit it here.
\end{proof}

The function $\phi _{\rho ,\left( u,v\right) }$ defined as (\ref{1-2})
possesses certain properties \cite{AP,R1}.

\begin{lemma}
\label{L2-3}Assume that conditions $(D1)-(D2)$ holds. Then for each $\left(
u,v\right) \in \mathbf{H},$ the following two inequalities are true.\newline
$\left( i\right) $ $\phi _{\rho ,\left( u,v\right) }\geq 0;$\newline
$\left( ii\right) $
\begin{eqnarray*}
\int_{\mathbb{R}^{3}}\rho \left( x\right) \phi _{\rho ,\left( u,v\right)
}\left( u^{2}+v^{2}\right) dx &\leq &\frac{16\sqrt[3]{2}\rho _{\max }^{2}}{3%
\sqrt{3}\pi }\left[ \int_{\mathbb{R}^{3}}(u^{2}+v^{2})dx\right] ^{3/2}\left[
\int_{\mathbb{R}^{3}}(|\nabla u|^{2}+|\nabla v|^{2})dx\right] ^{1/2} \\
&\leq &\frac{16\sqrt[3]{2}\rho _{\max }^{2}}{3\sqrt{3}\pi \lambda ^{3/2}}%
\left\Vert \left( u,v\right) \right\Vert ^{4},
\end{eqnarray*}%
where $\rho _{\max }=\sup_{x\in \mathbb{R}^{3}}\rho \left( x\right) .$
\end{lemma}

\begin{proof}
$\left( i\right) $ It is obvious.\newline
$\left( ii\right) $ By Hardy-Littlewood-Sobolev and Gagliardo-Nirenberg
inequalities, we have%
\begin{eqnarray*}
&&\int_{\mathbb{R}^{3}}\rho \left( x\right) \phi _{\rho ,\left( u,v\right)
}(u^{2}+v^{2})dx \\
&\leq &\frac{8\sqrt[3]{2}\rho _{\max }^{2}}{3\sqrt[3]{\pi }}\left[ \int_{%
\mathbb{R}^{3}}\left( u^{2}+v^{2}\right) ^{6/5}dx\right] ^{5/3} \\
&\leq &\frac{8\sqrt[3]{4}\rho _{\max }^{2}S}{3\sqrt[3]{\pi }}\left[ \int_{%
\mathbb{R}^{3}}(u^{2}+v^{2})dx\right] ^{3/2}\left[ \left( \int_{\mathbb{R}%
^{3}}|\nabla u|^{2}dx\right) ^{3}+\left( \int_{\mathbb{R}^{3}}|\nabla
v|^{2}dx\right) ^{3}\right] ^{1/6} \\
&\leq &\frac{16\sqrt[3]{2}\rho _{\max }^{2}}{3\sqrt{3}\pi \lambda ^{3/2}}%
\left[ \int_{\mathbb{R}^{3}}V(x)(u^{2}+v^{2})dx\right] ^{3/2}\left[ \int_{%
\mathbb{R}^{3}}(|\nabla u|^{2}+|\nabla v|^{2})dx\right] ^{1/2} \\
&\leq &\frac{16\sqrt[3]{2}\rho _{\max }^{2}}{3\sqrt{3}\pi \lambda ^{3/2}}%
\left\Vert \left( u,v\right) \right\Vert ^{4}.
\end{eqnarray*}%
This completes the proof.
\end{proof}

\begin{lemma}
\label{L2-2}Assume that $2<p<4,\beta >0$ and conditions $\left( D1\right)
-\left( D2\right) $ hold. Then for each $z\in H^{1}(\mathbb{R}%
^{3})\backslash \left\{ 0\right\} $, there exists $s_{z}\in \left(
0,1\right) $ such that
\begin{equation*}
J_{\beta }\left( \sqrt{s_{z}}z,\sqrt{1-s_{z}}z\right) <J_{\beta }\left(
z,0\right) =J_{\beta }\left( 0,z\right) =I_{0}\left( z\right) ,
\end{equation*}%
where
\begin{equation*}
I_{0}(z):=\frac{1}{2}\int_{\mathbb{R}^{3}}\left( |\nabla
z|^{2}+V(x)z^{2}\right) dx+\frac{1}{4}\int_{\mathbb{R}^{3}}\rho \left(
x\right) \phi _{\rho ,z}z^{2}dx-\frac{1}{p}\int_{\mathbb{R}^{3}}\left\vert
z\right\vert ^{p}dx.
\end{equation*}
\end{lemma}

\begin{proof}
Let $\left( u,v\right) =\left( \sqrt{s}z,\sqrt{1-s}z\right) $ for $z\in
H^{1}(\mathbb{R}^{3})\backslash \left\{ 0\right\} $ and $s\in \lbrack 0,1].$
A direct calculation shows that%
\begin{equation*}
\left\Vert \left( u,v\right) \right\Vert ^{2}=s\left\Vert z\right\Vert
_{V}^{2}+\left( 1-s\right) \left\Vert z\right\Vert _{V}^{2}=\left\Vert
z\right\Vert _{V}^{2}
\end{equation*}%
and%
\begin{equation*}
\int_{\mathbb{R}^{3}}\rho \left( x\right) \phi _{\rho ,\left( u,v\right)
}\left( u^{2}+v^{2}\right) dx=\int_{\mathbb{R}^{3}}\rho \left( x\right) \phi
_{\rho ,\left( sz^{2},\left( 1-s\right) z^{2}\right) }\left( sz^{2}+\left(
1-s\right) z^{2}\right) dx=\int_{\mathbb{R}^{3}}\rho \left( x\right) \phi
_{\rho ,z}z^{2}dx,
\end{equation*}%
where%
\begin{equation*}
\left\Vert u\right\Vert _{V}:=\left[ \int_{\mathbb{R}^{3}}(|\nabla
u|^{2}+V(x)u^{2})dx\right] ^{1/2}\text{ and }\phi _{\rho ,z}(x):=\int_{%
\mathbb{R}^{3}}\frac{\rho \left( y\right) z^{2}(y)}{|x-y|}dy.
\end{equation*}%
Moreover, by Lemma \ref{L2-1}, there exists $s_{z}\in \left( 0,1\right) $
such that
\begin{equation*}
\int_{\mathbb{R}^{3}}\left( \left\vert u\right\vert ^{p}+\left\vert
v\right\vert ^{p}+2\beta \left\vert u\right\vert ^{\frac{p}{2}}\left\vert
v\right\vert ^{\frac{p}{2}}\right) dx=\left[ s_{z}^{\frac{p}{2}}+\left(
1-s_{z}\right) ^{\frac{p}{2}}+2\beta s_{z}^{\frac{p}{4}}\left(
1-s_{z}\right) ^{\frac{p}{4}}\right] \int_{\mathbb{R}^{3}}\left\vert
z\right\vert ^{p}dx>\int_{\mathbb{R}^{3}}\left\vert z\right\vert ^{p}dx.
\end{equation*}%
Thus, we have
\begin{eqnarray*}
J_{\beta }\left( \sqrt{s_{z}}z,\sqrt{1-s_{z}}z\right) &=&\frac{1}{2}%
\left\Vert z\right\Vert _{V}^{2}+\frac{1}{4}\int_{\mathbb{R}^{3}}\rho \left(
x\right) \phi _{\rho ,z}z^{2}dx \\
&&-\frac{1}{p}\left[ s_{z}^{\frac{p}{2}}+\left( 1-s_{z}\right) ^{\frac{p}{2}%
}+2\beta s_{z}^{\frac{p}{4}}\left( 1-s_{z}\right) ^{\frac{p}{4}}\right]
\int_{\mathbb{R}^{3}}\left\vert z\right\vert ^{p}dx \\
&<&\frac{1}{2}\left\Vert z\right\Vert _{V}^{2}+\frac{1}{4}\int_{\mathbb{R}%
^{3}}\rho \left( x\right) \phi _{\rho ,z}z^{2}dx-\frac{1}{p}\int_{\mathbb{R}%
^{3}}\left\vert z\right\vert ^{p}dx \\
&=&J_{\beta }\left( z,0\right) =J_{\beta }\left( 0,z\right) =I_{0}\left(
z\right) .
\end{eqnarray*}%
The proof is complete.
\end{proof}

\begin{lemma}
\label{L2-4}Assume that $2<p<4,\beta >0$ and conditions $\left( D1\right)
-\left( D2\right) $ hold. Let $\left( u_{0},v_{0}\right) \in \mathbf{H}$
with
\begin{equation*}
J_{\beta }\left( u_{0},v_{0}\right) =\inf_{\left( u,v\right) \in \mathbf{H}%
}J_{\beta }\left( u,v\right) <\inf_{u\in H^{1}(\mathbb{R}^{3})}I_{0}\left(
u\right) <0.
\end{equation*}%
Then $u_{0}\neq 0$ and $v_{0}\neq 0.$
\end{lemma}

\begin{proof}
Suppose that $v_{0}=0.$ Let $\left( u,v\right) =\left( \sqrt{s}u_{0},\sqrt{%
1-s}u_{0}\right) $ for $s\in \lbrack 0,1].$ A direct calculation shows that%
\begin{equation*}
\left\Vert \left( u,v\right) \right\Vert ^{2}=s\left\Vert u_{0}\right\Vert
_{V}^{2}+\left( 1-s\right) \left\Vert u_{0}\right\Vert _{V}^{2}=\left\Vert
u_{0}\right\Vert _{V}^{2}
\end{equation*}%
and%
\begin{equation*}
\int_{\mathbb{R}^{3}}\rho \left( x\right) \phi _{\rho ,\left(
u_{0},v_{0}\right) }\left( u_{0}^{2}+v_{0}^{2}\right) dx=\int_{\mathbb{R}%
^{3}}\rho \left( x\right) \phi _{\rho ,\left( u_{0},v_{0}\right) }\left(
su_{0}^{2}+\left( 1-s\right) u_{0}^{2}\right) dx=\int_{\mathbb{R}^{3}}\phi
_{\rho ,u_{0}}u_{0}^{2}dx.
\end{equation*}%
Moreover, by Lemma \ref{L2-1}, there exists $s_{0}\in \left( 0,1\right) $
such that
\begin{equation*}
\int_{\mathbb{R}^{3}}\left( \left\vert u\right\vert ^{p}+\left\vert
v\right\vert ^{p}+2\beta \left\vert u\right\vert ^{\frac{p}{2}}\left\vert
v\right\vert ^{\frac{p}{2}}\right) dx=\left[ s_{0}^{\frac{p}{2}}+\left(
1-s_{0}\right) ^{\frac{p}{2}}+2\beta s_{0}^{\frac{p}{4}}\left(
1-s_{0}\right) ^{\frac{p}{4}}\right] \int_{\mathbb{R}^{3}}\left\vert
u_{0}\right\vert ^{p}dx>\int_{\mathbb{R}^{3}}\left\vert u_{0}\right\vert
^{p}dx.
\end{equation*}%
Thus, we have
\begin{eqnarray*}
J_{\beta }\left( \sqrt{s_{0}}u_{0},\sqrt{1-s_{0}}u_{0}\right) &\leq &\frac{1%
}{2}\left\Vert u_{0}\right\Vert _{V}^{2}+\frac{1}{4}\int_{\mathbb{R}%
^{3}}\phi _{\rho ,u_{0}}u_{0}^{2}dx \\
&&-\frac{1}{p}\left[ s_{0}^{\frac{p}{2}}+\left( 1-s_{0}\right) ^{\frac{p}{2}%
}+2\beta s_{0}^{\frac{p}{4}}\left( 1-s_{0}\right) ^{\frac{p}{4}}\right]
\int_{\mathbb{R}^{3}}\left\vert u_{0}\right\vert ^{p}dx \\
&<&\frac{1}{2}\left\Vert u_{0}\right\Vert _{V}^{2}+\frac{1}{4}\int_{\mathbb{R%
}^{3}}\phi _{\rho ,u_{0}}u_{0}^{2}dx-\frac{1}{p}\int_{\mathbb{R}%
^{3}}\left\vert u_{0}\right\vert ^{p}dx \\
&=&J_{\beta }\left( u_{0},0\right) =\inf_{\left( u,v\right) \in \mathbf{H}%
}J_{\beta }\left( u,v\right) ,
\end{eqnarray*}%
a contradiction. The proof is complete.
\end{proof}

\section{Construction of natural constraints}

Define the Nehari manifold
\begin{equation*}
\mathcal{N}_{\beta }:=\{\left( u,v\right) \in \mathbf{H}\backslash \{\left(
0,0\right) \}:\left\langle J_{\beta }^{\prime }\left( u,v\right) ,\left(
u,v\right) \right\rangle =0\}.
\end{equation*}%
Then $u\in \mathcal{N}_{\beta }$ if and only if
\begin{equation*}
\left\Vert \left( u,v\right) \right\Vert ^{2}+\int_{\mathbb{R}^{3}}\rho
\left( x\right) \phi _{\rho ,\left( u,v\right) }\left( u^{2}+v^{2}\right)
dx-\int_{\mathbb{R}^{3}}F_{\beta }\left( u,v\right) dx=0.
\end{equation*}%
It follows the Sobolev and Young inequalities that there exists $C_{\beta
}>0 $ such that
\begin{eqnarray*}
\left\Vert \left( u,v\right) \right\Vert ^{2} &\leq &\left\Vert \left(
u,v\right) \right\Vert ^{2}+\int_{\mathbb{R}^{3}}\rho \left( x\right) \phi
_{\rho ,\left( u,v\right) }\left( u^{2}+v^{2}\right) dx \\
&=&\int_{\mathbb{R}^{3}}F_{\beta }\left( u,v\right) dx \\
&\leq &C_{\beta }\left\Vert \left( u,v\right) \right\Vert ^{p}\text{ for all
}u\in \mathcal{N}_{\beta }.
\end{eqnarray*}%
So it leads to
\begin{equation}
\left\Vert \left( u,v\right) \right\Vert \geq C_{\beta }^{-1/\left(
p-2\right) }\text{ for all }u\in \mathcal{N}_{\beta }.  \label{2-2}
\end{equation}

The Nehari manifold $\mathcal{N}_{\beta }$ is closely linked to the behavior
of the function of the form $\phi _{\beta ,\left( u,v\right) }:t\rightarrow
J_{\beta }\left( tu,tv\right) $ for $t>0.$ Such maps are known as fibering
maps introduced by Dr\'{a}bek-Pohozaev \cite{DP}. For $\left( u,v\right) \in
\mathbf{H},$ we find that%
\begin{eqnarray*}
\phi _{\beta ,\left( u,v\right) }\left( t\right) &=&\frac{t^{2}}{2}%
\left\Vert \left( u,v\right) \right\Vert ^{2}+\frac{t^{4}}{4}\int_{\mathbb{R}%
^{3}}\rho \left( x\right) \phi _{\rho ,\left( u,v\right) }\left(
u^{2}+v^{2}\right) dx-\frac{t^{p}}{p}\int_{\mathbb{R}^{3}}F_{\beta }\left(
u,v\right) dx, \\
\phi _{\beta ,\left( u,v\right) }^{\prime }\left( t\right) &=&t\left\Vert
\left( u,v\right) \right\Vert ^{2}+t^{3}\int_{\mathbb{R}^{3}}\rho \left(
x\right) \phi _{\rho ,\left( u,v\right) }\left( u^{2}+v^{2}\right)
dx-t^{p-1}\int_{\mathbb{R}^{3}}F_{\beta }\left( u,v\right) dx, \\
\phi _{\beta ,\left( u,v\right) }^{\prime \prime }\left( t\right)
&=&\left\Vert \left( u,v\right) \right\Vert ^{2}+3t^{2}\int_{\mathbb{R}%
^{3}}\rho \left( x\right) \phi _{\rho ,\left( u,v\right) }\left(
u^{2}+v^{2}\right) dx-\left( p-1\right) t^{p-2}\int_{\mathbb{R}^{3}}F_{\beta
}\left( u,v\right) dx.
\end{eqnarray*}%
A direct calculation shows that
\begin{equation*}
t\phi _{\beta ,\left( u,v\right) }^{\prime }\left( t\right) =\left\Vert
\left( tu,tv\right) \right\Vert ^{2}+\int_{\mathbb{R}^{3}}\rho \left(
x\right) \phi _{\rho ,\left( tu,tv\right) }\left(
t^{2}u^{2}+t^{2}v^{2}\right) dx-\int_{\mathbb{R}^{3}}F_{\beta }\left(
tu,tv\right) dx
\end{equation*}%
and so, for $\left( u,v\right) \in \mathbf{H}\backslash \left\{ \left(
0,0\right) \right\} $ and $t>0,$ $\phi _{\beta ,\left( u,v\right) }^{\prime
}\left( t\right) =0$ holds if and only if $\left( tu,tv\right) \in \mathcal{N%
}_{\beta }$. In particular, $\phi _{\beta ,\left( u,v\right) }^{\prime
}\left( 1\right) =0$ holds if and only if $\left( u,v\right) \in \mathcal{N}%
_{\beta }.$ It becomes natural to split $\mathcal{N}_{\beta }$ into three
parts corresponding to the local minima, local maxima and points of
inflection. We define
\begin{eqnarray*}
\mathcal{N}_{\beta }^{+} &:&=\{u\in \mathcal{N}_{\beta }:\phi _{\beta
,\left( u,v\right) }^{\prime \prime }\left( 1\right) >0\}, \\
\mathcal{N}_{\beta }^{0} &:&=\{u\in \mathcal{N}_{\beta }:\phi _{\beta
,\left( u,v\right) }^{\prime \prime }\left( 1\right) =0\}, \\
\mathcal{N}_{\beta }^{-} &:&=\{u\in \mathcal{N}_{\beta }:\phi _{\beta
,\left( u,v\right) }^{\prime \prime }\left( 1\right) <0\}.
\end{eqnarray*}

\begin{lemma}
\label{g2}Assume that $\left( u_{0},v_{0}\right) $ is a local minimizer for $%
J_{\beta }$ on $\mathcal{N}_{\beta }$ and $\left( u_{0},v_{0}\right) \notin
\mathcal{N}_{\beta }^{0}.$ Then $J_{\beta }^{\prime }\left(
u_{0},v_{0}\right) =0$ in $\mathbf{H}^{-1}.$
\end{lemma}

\begin{proof}
The proof is essentially same as that in Brown-Zhang \cite[Theorem 2.3]{BZ},
so we omit it here.
\end{proof}

For each $\left( u,v\right) \in \mathcal{N}_{\beta },$ we find that
\begin{eqnarray}
\phi _{\beta ,\left( u,v\right) }^{\prime \prime }\left( 1\right) &=&-\left(
p-2\right) \left\Vert \left( u,v\right) \right\Vert ^{2}+\left( 4-p\right)
\int_{\mathbb{R}^{3}}\rho \left( x\right) \phi _{\rho ,\left( u,v\right)
}\left( u^{2}+v^{2}\right) dx  \label{2-6-1} \\
&=&-2\left\Vert \left( u,v\right) \right\Vert ^{2}+\left( 4-p\right) \int_{%
\mathbb{R}^{3}}F_{\beta }\left( u,v\right) dx.  \label{2-6-2}
\end{eqnarray}%
For each $\left( u,v\right) \in \mathcal{N}_{\beta }^{-}$, using (\ref{2-2})
and (\ref{2-6-2}) gives
\begin{eqnarray*}
J_{\lambda ,\beta }(u,v) &=&\frac{1}{4}\left\Vert \left( u,v\right)
\right\Vert ^{2}-\frac{4-p}{4p}\int_{\mathbb{R}^{3}}F_{\beta }\left(
u,v\right) dx \\
&>&\frac{p-2}{4p}\left\Vert \left( u,v\right) \right\Vert ^{2} \\
&\geq &\frac{p-2}{4p}C_{\beta }^{-1/\left( p-2\right) }>0.
\end{eqnarray*}%
For each $\left( u,v\right) \in \mathcal{N}_{\beta }^{+},$ by (\ref{2-6-1})
one has%
\begin{eqnarray*}
J_{\beta }(u,v) &=&\frac{p-2}{2p}\left\Vert \left( u,v\right) \right\Vert
^{2}-\frac{4-p}{4p}\int_{\mathbb{R}^{3}}\rho \left( x\right) \phi _{\rho
,\left( u,v\right) }\left( u^{2}+v^{2}\right) dx \\
&<&\frac{p-2}{4p}\left\Vert \left( u,v\right) \right\Vert ^{2}.
\end{eqnarray*}%
Hence, we have the following result.

\begin{lemma}
\label{g5}The energy functional $J_{\beta }$ is coercive and bounded from
below on $\mathcal{N}_{\beta }^{-}.$ Furthermore, for all $u\in \mathcal{N}%
_{\beta }^{-},$ there holds
\begin{equation*}
J_{\beta }(u,v)>\frac{p-2}{4p}C_{\beta }^{-1/\left( p-2\right) }>0.
\end{equation*}
\end{lemma}

Following the idea of \cite{SWF1}. Let $\left( u,v\right) \in \mathcal{N}%
_{\beta }$ with $J_{\beta }\left( u,v\right) <\frac{3\sqrt{3}\pi \lambda
^{3/2}\left( p-2\right) ^{2}}{64\sqrt[3]{2}p\rho _{\max }^{2}(4-p)},$ by
Lemma \ref{L2-3}, we can deduce that%
\begin{eqnarray*}
\frac{3\sqrt{3}\pi \lambda ^{3/2}\left( p-2\right) ^{2}}{64\sqrt[3]{2}p\rho
_{\max }^{2}(4-p)} &>&J_{\beta }(u,v)=\frac{p-2}{2p}\left\Vert \left(
u,v\right) \right\Vert ^{2}-\frac{4-p}{4p}\int_{\mathbb{R}^{3}}\rho \left(
x\right) \phi _{\rho ,\left( u,v\right) }\left( u^{2}+v^{2}\right) dx \\
&\geq &\frac{p-2}{2p}\left\Vert \left( u,v\right) \right\Vert ^{2}-\frac{%
2^{7/3}\rho _{\max }^{2}(4-p)}{3\sqrt{3}\pi p\lambda ^{3/2}}\left\Vert
\left( u,v\right) \right\Vert ^{4}.
\end{eqnarray*}%
Set
\begin{equation*}
f\left( x\right) :=\frac{p-2}{2p}x-\frac{4\sqrt[3]{2}(4-p)\rho _{\max }^{2}}{%
3\sqrt{3}\pi p\lambda ^{3/2}}x^{2}\text{ for }x>0.
\end{equation*}%
A direct calculation shows that%
\begin{equation*}
\max_{x\geq 0}f\left( x\right) =f\left( \bar{x}\right) =\frac{3\sqrt{3}\pi
\left( p-2\right) ^{2}\lambda ^{3/2}}{64\sqrt[3]{2}p(4-p)\rho _{\max }^{2}},
\end{equation*}%
where%
\begin{equation*}
\bar{x}:=\frac{3\sqrt{3}\pi \left( p-2\right) \lambda ^{3/2}}{16\sqrt[3]{2}%
(4-p)\rho _{\max }^{2}}.
\end{equation*}%
Thus, we have the decomposition of the filtration of $\mathcal{N}_{\beta }$
as follows:%
\begin{equation*}
\mathcal{N}_{\beta }\left[ \frac{3\sqrt{3}\pi \left( p-2\right) ^{2}\lambda
^{3/2}}{64\sqrt[3]{2}p(4-p)\rho _{\max }^{2}}\right] =\mathcal{N}_{\beta
}^{(1)}\left[ \frac{3\sqrt{3}\pi \left( p-2\right) ^{2}\lambda ^{3/2}}{64%
\sqrt[3]{2}p(4-p)\rho _{\max }^{2}}\right] \cup \mathcal{N}_{\beta }^{(2)}%
\left[ \frac{3\sqrt{3}\pi \left( p-2\right) ^{2}\lambda ^{3/2}}{64\sqrt[3]{2}%
p(4-p)\rho _{\max }^{2}}\right] ,
\end{equation*}%
where%
\begin{equation*}
\mathcal{N}_{\beta }[D]:=\left\{ u\in \mathcal{N}_{\beta }:J_{\beta }\left(
u,v\right) <D\right\}
\end{equation*}%
and%
\begin{equation*}
\mathcal{N}_{\beta }^{(1)}[D]:=\left\{ u\in \mathcal{N}_{\beta
}[D]:\left\Vert \left( u,v\right) \right\Vert <\left( \frac{3\sqrt{3}\pi
\left( p-2\right) \lambda ^{3/2}}{16\sqrt[3]{2}(4-p)\rho _{\max }^{2}}%
\right) ^{1/2}\right\}
\end{equation*}%
and
\begin{equation*}
\mathcal{N}_{\beta }^{(2)}[D]:=\left\{ u\in \mathcal{N}_{\beta
}[D]:\left\Vert \left( u,v\right) \right\Vert >\left( \frac{3\sqrt{3}\pi
\left( p-2\right) \lambda ^{3/2}}{16\sqrt[3]{2}(4-p)\rho _{\max }^{2}}%
\right) ^{1/2}\right\}
\end{equation*}%
for $D>0.$ For convenience, we always set%
\begin{equation*}
\mathcal{N}_{\beta }^{(1)}:=\mathcal{N}_{\beta }^{(1)}\left[ \frac{3\sqrt{3}%
\pi \left( p-2\right) ^{2}\lambda ^{3/2}}{64\sqrt[3]{2}p(4-p)\rho _{\max
}^{2}}\right] \text{ and }\mathcal{N}_{\beta }^{(2)}:=\mathcal{N}_{\beta
}^{(2)}\left[ \frac{3\sqrt{3}\pi \left( p-2\right) ^{2}\lambda ^{3/2}}{64%
\sqrt[3]{2}p(4-p)\rho _{\max }^{2}}\right] .
\end{equation*}%
From (\ref{2-6-1}), the Sobolev inequality and Lemma \ref{L2-3}, it follows
that
\begin{eqnarray}
\phi _{\beta ,\left( u,v\right) }^{\prime \prime }\left( 1\right) &=&-\left(
p-2\right) \left\Vert \left( u,v\right) \right\Vert ^{2}+\left( 4-p\right)
\int_{\mathbb{R}^{3}}\rho \left( x\right) \phi _{\rho ,\left( u,v\right)
}\left( u^{2}+v^{2}\right) dx  \notag \\
&\leq &\left\Vert \left( u,v\right) \right\Vert ^{2}\left[ \frac{16\sqrt[3]{2%
}(4-p)\rho _{\max }^{2}}{3\sqrt{3}\pi \lambda ^{3/2}}\left\Vert \left(
u,v\right) \right\Vert ^{2}-\left( p-2\right) \right]  \notag \\
&<&0\text{ for all }u\in \mathcal{N}_{\beta }^{(1)}.  \label{3-6}
\end{eqnarray}%
By (\ref{2-6-2}) we derive that
\begin{eqnarray*}
\frac{1}{4}\left\Vert \left( u,v\right) \right\Vert ^{2}-\frac{4-p}{4p}\int_{%
\mathbb{R}^{3}}F_{\beta }\left( u,v\right) dx &=&J_{\beta }\left( u,v\right)
<\frac{3\sqrt{3}\pi \left( p-2\right) ^{2}\lambda ^{3/2}}{64\sqrt[3]{2}%
p(4-p)\rho _{\max }^{2}} \\
&<&\frac{p-2}{4p}\left\Vert \left( u,v\right) \right\Vert ^{2}\text{ for all
}u\in \mathcal{N}_{\beta }^{(2)},
\end{eqnarray*}%
which implies that%
\begin{equation*}
\phi _{\beta ,\left( u,v\right) }^{\prime \prime }\left( 1\right)
=-2\left\Vert \left( u,v\right) \right\Vert ^{2}+\left( 4-p\right) \int_{%
\mathbb{R}^{3}}F_{\beta }\left( u,v\right) dx>0\text{ for all }u\in \mathcal{%
N}_{\beta }^{(2)}.
\end{equation*}%
Thus, the following lemma is true.

\begin{lemma}
\label{g7}Assume that $2<p<4,\beta >0$ and conditions $\left( D1\right)
-\left( D2\right) $ hold. Then $\mathcal{N}_{\beta }^{(1)}\subset \mathcal{N}%
_{\beta }^{-}$ and $\mathcal{N}_{\beta }^{(2)}\subset \mathcal{N}_{\beta
}^{+}$ are $C^{1}$ submanifolds. Furthermore, each local minimizer of the
functional $J_{\beta }$ in the sub-manifolds $\mathcal{N}_{\beta }^{(1)}$
and $\mathcal{N}_{\beta }^{(2)}$ is a nontrivial critical point of $J_{\beta
}$ in $\mathbf{H}.$
\end{lemma}

Furthermore, we also have the following result.

\begin{lemma}
\label{l5}Assume that $2<p<4,\beta >0$ and conditions $\left( D1\right)
-\left( D2\right) $ hold. Let $\left( u_{0},v_{0}\right) $ be a critical
point of $J_{\beta }$ on $\mathcal{N}_{\beta }^{-}.$ Then we have $J_{\beta
}\left( u_{0},v_{0}\right) >\frac{p-2}{2p}S_{\lambda ,p}^{2p/\left(
p-2\right) }$ if either $u_{0}=0$ or $v_{0}=0.$
\end{lemma}

\begin{proof}
Without loss of generality, we may assume that $v_{0}=0.$ Then we have%
\begin{equation*}
J_{\beta }\left( u_{0},0\right) =\frac{1}{2}\left\Vert u_{0}\right\Vert
_{V}^{2}+\frac{1}{4}\int_{\mathbb{R}^{3}}\rho \left( x\right) \phi _{\rho
,u_{0}}u_{0}^{2}dx-\frac{1}{p}\int_{\mathbb{R}^{3}}\left\vert
u_{0}\right\vert ^{p}dx
\end{equation*}%
and%
\begin{equation*}
-2\left\Vert u_{0}\right\Vert _{V}^{2}+\left( 4-p\right) \int_{\mathbb{R}%
^{3}}\left\vert u_{0}\right\vert ^{p}dx<0.
\end{equation*}%
Note that
\begin{equation*}
\int_{\mathbb{R}^{3}}\left[ |\nabla t_{0}\left( u_{0}\right)
u_{0}|^{2}+\lambda \left( t_{0}\left( u_{0}\right) u_{0}\right) ^{2}\right]
dx-\int_{\mathbb{R}^{3}}\left\vert t_{0}\left( u_{0}\right) u_{0}\right\vert
^{p}dx=0,
\end{equation*}%
where
\begin{equation}
0<t_{0}\left( u_{0}\right) :=\left[ \frac{\int_{\mathbb{R}^{3}}(|\nabla
u_{0}|^{2}+\lambda u_{0}^{2})dx}{\int_{\mathbb{R}^{3}}\left\vert
u_{0}\right\vert ^{p}dx}\right] ^{1/\left( p-2\right) }<1.  \label{2-6-3}
\end{equation}%
Adopting the similar argument in \cite[Lemma 2.6]{SWF1}, one has
\begin{equation*}
J_{\beta }\left( u_{0},0\right) =\sup_{0\leq t\leq t_{0}^{+}}J_{\beta
}(tu_{0},0),
\end{equation*}%
where $t_{0}^{+}>\left( \frac{2}{4-p}\right) ^{1/\left( p-2\right)
}t_{0}\left( u_{0}\right) >1$ by (\ref{2-6-3}). Using this, together with (%
\ref{2-6-3}) again, leads to
\begin{equation*}
J_{\beta }\left( u_{0},0\right) >J_{\beta }(t_{0}\left( u_{0}\right)
u_{0},0).
\end{equation*}%
Thus, by \cite{Wi}, we have
\begin{eqnarray*}
J_{\beta }\left( u_{0},0\right) &\geq &J_{\beta }(t_{0}\left( u_{0}\right)
u_{0},0) \\
&\geq &\frac{1}{2}\left\Vert t_{0}\left( u_{0}\right) u_{0}\right\Vert
_{V}^{2}-\frac{1}{p}\int_{\mathbb{R}^{3}}\left\vert t_{0}\left( u_{0}\right)
u_{0}\right\vert ^{p}dx+\frac{\lambda \left[ t_{0}\left( u_{0}\right) \right]
^{4}}{4}\int_{\mathbb{R}^{3}}\rho \left( x\right) \phi _{\rho
,u_{0}}u_{0}^{2}dx \\
&>&\frac{1}{2}\int_{\mathbb{R}^{3}}\left[ |\nabla t_{0}\left( u_{0}\right)
u_{0}|^{2}+\lambda \left( t_{0}\left( u_{0}\right) u_{0}\right) ^{2}\right]
dx-\frac{1}{p}\int_{\mathbb{R}^{3}}\left\vert t_{0}\left( u_{0}\right)
u_{0}\right\vert ^{p}dx \\
&\geq &\frac{p-2}{2p}S_{\lambda ,p}^{2p/\left( p-2\right) }.
\end{eqnarray*}%
The proof is complete.
\end{proof}

Let $w_{\beta }$ be the unique positive radial solution of the Schr\"{o}%
dinger equation%
\begin{equation}
\begin{array}{ll}
-\Delta u+\lambda u=\frac{\lambda g_{\beta }\left( s_{\beta }\right) }{%
V_{\max }}\left\vert u\right\vert ^{p-2}u & \text{ in }\mathbb{R}^{3},%
\end{array}
\tag*{$\left( E_{\beta }^{\infty }\right) $}
\end{equation}%
where $g_{\beta }\left( s_{\beta }\right) =\max_{s\in \left[ 0,1\right]
}g_{\beta }\left( s\right) >1$ as in Lemma \ref{L2-1}. Note that $s_{\beta }=%
\frac{1}{2}$ and $g_{\beta }\left( \frac{1}{2}\right) =2^{-\frac{p-2}{2}%
}\left( 1+\beta \right) $ for all $\beta \geq \frac{p-2}{2}.$ From \cite{K},
we see that
\begin{equation*}
w_{\beta }\left( 0\right) =\max_{x\in \mathbb{R}^{3}}w_{\beta }(x)\text{ and
}\left\Vert w_{\beta }\right\Vert _{\lambda }^{2}=\int_{\mathbb{R}^{3}}\frac{%
\lambda g_{\beta }\left( s_{\beta }\right) }{V_{\max }}\left\vert w_{\beta
}\right\vert ^{p}dx=\left( \frac{V_{\max }S_{\lambda ,p}^{p}}{\lambda
g_{\beta }\left( s_{\beta }\right) }\right) ^{2/\left( p-2\right) }
\end{equation*}%
and%
\begin{equation}
\alpha _{\beta }^{\infty }:=\inf_{u\in \mathcal{M}_{\beta }^{\infty
}}I_{\beta }^{\infty }(u)=I_{\beta }^{\infty }(w_{\beta })=\frac{p-2}{2p}%
\left( \frac{V_{\max }S_{\lambda ,p}^{p}}{\lambda g_{\beta }\left( s_{\beta
}\right) }\right) ^{2/\left( p-2\right) },  \label{4-1}
\end{equation}%
where $I_{\beta }^{\infty }$ is the energy functional of equation $\left(
E_{\beta }^{\infty }\right) $ in $H^{1}(\mathbb{R}^{3})$ given by%
\begin{equation*}
I_{\beta }^{\infty }(u)=\frac{1}{2}\left\Vert u\right\Vert _{\lambda }^{2}-%
\frac{\lambda g_{\beta }\left( s_{\beta }\right) }{pV_{\max }}\int_{\mathbb{R%
}^{3}}\left\vert u\right\vert ^{p}dx,
\end{equation*}%
and
\begin{equation*}
\mathcal{M}_{\beta }^{\infty }:=\{u\in H^{1}(\mathbb{R}^{3})\backslash
\{0\}:\left\langle (I_{\beta }^{\infty })^{\prime }\left( u\right)
,u\right\rangle =0\},
\end{equation*}%
here $\left\Vert u\right\Vert _{\lambda }:=\left( \int_{\mathbb{R}%
^{3}}(|\nabla u|^{2}+\lambda u^{2})dx\right) ^{1/2}.$ We have the following
result.

\begin{lemma}
\label{g4}Assume that $2<p<4$ and conditions $\left( D1\right) -\left(
D2\right) $ hold. Then there exists $\bar{\beta}_{0}=\bar{\beta}_{0}\left(
\lambda ,V_{\max },\rho _{\max }\right) >\frac{p-2}{2}$ such that for every $%
\beta >\bar{\beta}_{0},$ there exists two constants $t_{\beta }^{+}$ and $%
t_{\beta }^{-}$ satisfying
\begin{equation*}
1<t_{\beta }^{-}<\left( \frac{2}{4-p}\right) ^{\frac{1}{p-2}}<t_{\beta }^{+}
\end{equation*}%
such that $\left( t_{\beta }^{+}\sqrt{s_{\beta }}w_{\beta },t_{\beta }^{+}%
\sqrt{1-s_{\beta }}w_{\beta }\right) \in \mathcal{N}_{\beta }^{\left(
2\right) }\cap \mathbf{H}_{r}$ with%
\begin{equation*}
J_{\beta }\left( t_{\beta }^{+}\sqrt{s_{\beta }}w_{\beta },t_{\beta }^{+}%
\sqrt{1-s_{\beta }}w_{\beta }\right) =\inf_{t\geq 0}J_{\beta }\left( t\sqrt{%
s_{\beta }}w_{\beta },t\sqrt{1-s_{\beta }}w_{\beta }\right) <0,
\end{equation*}%
and $\left( t_{\beta }^{-}\sqrt{s_{\beta }}w_{\beta },t_{\beta }^{-}\sqrt{%
1-s_{\beta }}w_{\beta }\right) \in \mathcal{N}_{\beta }^{\left( 1\right)
}\cap \mathbf{H}_{r}$ with%
\begin{equation*}
J_{\beta }\left( t_{\beta }^{-}\sqrt{s_{\beta }}w_{\beta },t_{\beta }^{-}%
\sqrt{1-s_{\beta }}w_{\beta }\right) <\gamma \left( \beta \right) <\min
\left\{ \frac{3\sqrt{3}\pi \left( p-2\right) ^{2}\lambda ^{3/2}}{64\sqrt[3]{2%
}p(4-p)\rho _{\max }^{2}},\frac{p-2}{2p}S_{\lambda ,p}^{2p/\left( p-2\right)
}\right\} ,
\end{equation*}%
where
\begin{equation}
\gamma \left( \beta \right) :=\frac{2V_{\max }\left( p-2\right) }{p\lambda }%
\left( \frac{V_{\max }S_{\lambda ,p}^{p}}{\lambda \left( 1+\beta \right) }%
\right) ^{2/\left( p-2\right) }+\frac{16\sqrt[3]{2}\rho _{\max }^{2}V_{\max }%
}{3\sqrt{3}\pi \lambda ^{5/2}}\left( \frac{2}{4-p}\right) ^{\frac{4}{p-2}%
}\left( \frac{V_{\max }S_{\lambda ,p}^{p}}{\lambda \left( 1+\beta \right) }%
\right) ^{4/\left( p-2\right) }.  \label{3-3}
\end{equation}
\end{lemma}

\begin{proof}
For $t>0$ and $\beta \geq \frac{p-2}{2},$ we define
\begin{equation*}
\eta \left( t\right) :=t^{-2}\left\Vert \left( \sqrt{s_{\beta }}w_{\beta },%
\sqrt{1-s_{\beta }}w_{\beta }\right) \right\Vert ^{2}-t^{p-4}\int_{\mathbb{R}%
^{3}}F_{\beta }\left( \sqrt{s_{\beta }}w_{\beta },\sqrt{1-s_{\beta }}%
w_{\beta }\right) dx.
\end{equation*}%
Then we have%
\begin{eqnarray}
\eta \left( t\right) &=&t^{-2}\left\Vert w_{\beta }\right\Vert
_{V}^{2}-t^{p-4}\int_{\mathbb{R}^{3}}g_{\beta }\left( s_{\beta }\right)
\left\vert w_{\beta }\right\vert ^{p}dx  \notag \\
&\leq &\frac{V_{\max }}{\lambda }\left( t^{-2}\left\Vert w_{\beta
}\right\Vert _{\lambda }^{2}-t^{p-4}\int_{\mathbb{R}^{3}}\frac{\lambda
g_{\beta }\left( s_{\beta }\right) }{V_{\max }}\left\vert w_{\beta
}\right\vert ^{p}dx\right) .  \label{3-7}
\end{eqnarray}%
Clearly, $tu\in \mathcal{N}_{\beta }$ if and only if
\begin{equation*}
\eta \left( t\right) =-\int_{\mathbb{R}^{3}}\rho \left( x\right) \phi _{\rho
,w_{\beta }}w_{\beta }^{2}dx.
\end{equation*}%
Let%
\begin{equation*}
\eta _{0}\left( t\right) =t^{-2}\left\Vert w_{\beta }\right\Vert _{\lambda
}^{2}-t^{p-4}\int_{\mathbb{R}^{3}}\frac{\lambda g_{\beta }\left( s_{\beta
}\right) }{V_{\max }}\left\vert w_{\beta }\right\vert ^{p}dx.
\end{equation*}%
A straightforward calculation shows that
\begin{equation*}
\eta _{0}\left( 1\right) =0,\ \lim_{t\rightarrow 0^{+}}\eta _{0}(t)=\infty
\text{ and }\lim_{t\rightarrow \infty }\eta _{0}(t)=0.
\end{equation*}%
Since $2<p<4$ and
\begin{equation*}
\eta _{0}^{\prime }\left( t\right) =t^{-3}\left\Vert w_{\beta }\right\Vert
_{\lambda }^{2}\left[ -2+\left( 4-p\right) t^{p-2}\right] ,
\end{equation*}%
we find that $\eta _{0}\left( t\right) $ is decreasing when $0<t<\left(
\frac{2}{4-p}\right) ^{1/\left( p-2\right) }$ and is increasing when $%
t>\left( \frac{2}{4-p}\right) ^{1/\left( p-2\right) }.$ This implies that
\begin{equation}
\inf_{t>0}\eta _{0}\left( t\right) =\eta _{0}\left( \left( \frac{2}{4-p}%
\right) ^{1/\left( p-2\right) }\right) .  \label{3-8}
\end{equation}%
Set%
\begin{equation*}
B_{0}:=\frac{2V_{\max }S_{\lambda ,p}^{p}}{\lambda \left( 4-p\right)
^{\left( 4-p\right) /2}}\left( \frac{32\sqrt[3]{2}\rho _{\max }^{2}}{3\sqrt{3%
}\pi \lambda ^{3/2}}\right) ^{\left( p-2\right) /2}\left( \frac{\lambda }{%
\left( p-2\right) V_{\max }}\right) ^{\left( p-2\right) /2}-1.
\end{equation*}%
Then for each $\beta >\beta _{1}:=\max \left\{ \frac{p-2}{2},B_{0}\right\} ,$
by (\ref{3-7})--(\ref{3-8}) one has%
\begin{eqnarray*}
\eta \left( t\right) &\leq &\frac{V_{\max }}{\lambda }\eta _{0}\left( \left(
\frac{2}{4-p}\right) ^{1/(p-2)}\right) \\
&=&-\frac{V_{\max }}{\lambda }\left( \frac{p-2}{2}\right) \left( \frac{4-p}{2%
}\right) ^{\left( 4-p\right) /(p-2)}\left\Vert w_{\beta }\right\Vert
_{\lambda }^{2} \\
&<&-\frac{16\sqrt[3]{2}\rho _{\max }^{2}}{3\sqrt{3}\pi \lambda ^{3/2}}%
\left\Vert w_{\beta }\right\Vert _{\lambda }^{4} \\
&\leq &-\int_{\mathbb{R}^{3}}\rho \left( x\right) \phi _{\rho ,w_{\beta
}}w_{\beta }^{2}dx \\
&=&-\int_{\mathbb{R}^{3}}\rho \left( x\right) \phi _{\rho ,\left( \sqrt{%
s_{\beta }}w_{\beta },\sqrt{1-s_{\beta }}w_{\beta }\right) }\left[ \left(
\sqrt{s_{\beta }}w_{\beta }\right) ^{2}+\left( \sqrt{1-s_{\beta }}w_{\beta
}\right) ^{2}\right] dx,
\end{eqnarray*}%
which implies that there exist two positive constants $t_{\beta }^{+}$ and $%
t_{\beta }^{-}$ satisfying
\begin{equation*}
1<t_{\beta }^{-}<\left( \frac{2}{4-p}\right) ^{1/\left( p-2\right)
}<t_{\beta }^{+}
\end{equation*}%
such that
\begin{equation*}
\eta \left( t_{\beta }^{\pm }\right) +\int_{\mathbb{R}^{3}}\rho \left(
x\right) \phi _{\rho ,\left( \sqrt{s_{\beta }}w_{\beta },\sqrt{1-s_{\beta }}%
w_{\beta }\right) }\left[ \left( \sqrt{s_{\beta }}w_{\beta }\right)
^{2}+\left( \sqrt{1-s_{\beta }}w_{\beta }\right) ^{2}\right] dx=0.
\end{equation*}%
That is,
\begin{equation*}
\left( t_{\beta }^{\pm }\sqrt{s_{\beta }}w_{\beta },t_{\beta }^{\pm }\sqrt{%
1-s_{\beta }}w_{\beta }\right) \in \mathcal{N}_{\beta }\cap \mathbf{H}_{r}.
\end{equation*}%
By a calculation on the second order derivatives, we find
\begin{eqnarray*}
\phi _{\beta ,\left( t_{\beta }^{-}\sqrt{s_{\beta }}w_{\beta },t_{\beta }^{-}%
\sqrt{1-s_{\beta }}w_{\beta }\right) }^{\prime \prime }\left( 1\right)
&=&-2\left\Vert t_{\beta }^{-}w_{\beta }\right\Vert _{H^{1}}^{2}+\left(
4-p\right) \int_{\mathbb{R}^{3}}g_{\beta }\left( s_{\beta }\right)
\left\vert t_{\beta }^{-}w_{\beta }\right\vert ^{p}dx \\
&=&\left( t_{\beta }^{-}\right) ^{5}\eta ^{\prime }\left( t_{\beta
}^{-}\right) <0
\end{eqnarray*}%
and
\begin{eqnarray*}
\phi _{\beta ,\left( t_{\beta }^{+}\sqrt{s_{\beta }}w_{\beta },t_{\beta }^{+}%
\sqrt{1-s_{\beta }}w_{\beta }\right) }^{\prime \prime }\left( 1\right)
&=&-2\left\Vert t_{\beta }^{+}w_{\beta }\right\Vert _{H^{1}}^{2}+\left(
4-p\right) \int_{\mathbb{R}^{3}}g_{\beta }\left( s_{\beta }\right)
\left\vert t_{\beta }^{+}w_{\beta }\right\vert ^{p}dx \\
&=&\left( t_{\beta }^{+}\right) ^{5}\eta ^{\prime }\left( t_{\beta
}^{+}\right) >0,
\end{eqnarray*}%
leading to
\begin{equation*}
\left( t_{\beta }^{\pm }\sqrt{s_{\beta }}w_{\beta },t_{\beta }^{\pm }\sqrt{%
1-s_{\beta }}w_{\beta }\right) \in \mathcal{N}_{\beta }^{\pm }\cap \mathbf{H}%
_{r}.
\end{equation*}%
Since
\begin{equation*}
\phi _{\beta ,\left( \sqrt{s_{\beta }}w_{\beta },\sqrt{1-s_{\beta }}w_{\beta
}\right) }^{\prime }\left( t\right) =t^{3}\left[ \eta (t)+\int_{\mathbb{R}%
^{3}}\rho \left( x\right) \phi _{\rho ,\left( \sqrt{s_{\beta }}w_{\beta },%
\sqrt{1-s_{\beta }}w_{\beta }\right) }\left( \left( \sqrt{s_{\beta }}%
w_{\beta }\right) ^{2}+\left( \sqrt{1-s_{\beta }}w_{\beta }\right)
^{2}\right) dx\right] ,
\end{equation*}%
we have
\begin{equation*}
\phi _{\beta ,\left( \sqrt{s_{\beta }}w_{\beta },\sqrt{1-s_{\beta }}w_{\beta
}\right) }^{\prime }\left( t\right) >0\text{ for all }t\in \left( 0,t_{\beta
}^{-}\right) \cup \left( t_{\beta }^{+},\infty \right)
\end{equation*}%
and
\begin{equation*}
\phi _{\beta ,\left( \sqrt{s_{\beta }}w_{\beta },\sqrt{1-s_{\beta }}w_{\beta
}\right) }^{\prime }\left( t\right) <0\text{ for all }t\in (t_{\beta
}^{-},t_{\beta }^{+}),
\end{equation*}%
implying that
\begin{equation*}
J_{\beta }\left( t_{\beta }^{-}\sqrt{s_{\beta }}w_{\beta },t_{\beta }^{-}%
\sqrt{1-s_{\beta }}w_{\beta }\right) =\sup_{0\leq t\leq t_{\beta
}^{+}}J_{\beta }\left( t\sqrt{s_{\beta }}w_{\beta },t\sqrt{1-s_{\beta }}%
w_{\beta }\right)
\end{equation*}%
and%
\begin{equation*}
J_{\beta }\left( t_{\beta }^{+}\sqrt{s_{\beta }}w_{\beta },t_{\beta }^{+}%
\sqrt{1-s_{\beta }}w_{\beta }\right) =\inf_{t\geq t_{\beta }^{-}}J_{\beta
}\left( t\sqrt{s_{\beta }}w_{\beta },t\sqrt{1-s_{\beta }}w_{\beta }\right) ,
\end{equation*}%
and so
\begin{equation*}
J_{\beta }\left( t_{\beta }^{+}\sqrt{s_{\beta }}w_{\beta },t_{\beta }^{+}%
\sqrt{1-s_{\beta }}w_{\beta }\right) <J_{\beta }\left( t_{\beta }^{-}\sqrt{%
s_{\beta }}w_{\beta },t_{\beta }^{-}\sqrt{1-s_{\beta }}w_{\beta }\right) .
\end{equation*}%
Note that
\begin{eqnarray*}
J_{\beta }\left( t\sqrt{s_{\beta }}w_{\beta },t\sqrt{1-s_{\beta }}w_{\beta
}\right) &=&\frac{t^{2}}{2}\left\Vert w_{\beta }\right\Vert _{V}^{2}+\frac{%
t^{4}}{4}\int_{\mathbb{R}^{3}}\rho \left( x\right) \phi _{\rho ,w_{\beta
}}w_{\beta }^{2}dx-\frac{t^{p}}{p}\int_{\mathbb{R}^{3}}g_{\beta }\left(
s_{\beta }\right) \left\vert w_{\beta }\right\vert ^{p}dx \\
&=&t^{4}\left( \xi \left( t\right) +\frac{1}{4}\int_{\mathbb{R}^{3}}\rho
\left( x\right) \phi _{\rho ,w_{\beta }}w_{\beta }^{2}dx\right) ,
\end{eqnarray*}%
where
\begin{eqnarray}
\xi \left( t\right) := &&\frac{t^{-2}}{2}\left\Vert w_{\beta }\right\Vert
_{V}^{2}-\frac{t^{p-4}}{p}\int_{\mathbb{R}^{3}}g_{\beta }\left( s_{\beta
}\right) \left\vert w_{\beta }\right\vert ^{p}dx  \notag \\
&\leq &t^{-2}\frac{V_{\max }}{2\lambda }\left\Vert w_{\beta }\right\Vert
_{\lambda }^{2}-\frac{t^{p-4}}{p}\int_{\mathbb{R}^{3}}g_{\beta }\left(
s_{\beta }\right) \left\vert w_{\beta }\right\vert ^{p}dx  \notag \\
&=&\frac{V_{\max }}{\lambda }\left( \frac{t^{-2}}{2}\left\Vert w_{\beta
}\right\Vert _{\lambda }^{2}-\frac{t^{p-4}}{p}\int_{\mathbb{R}^{3}}\frac{%
\lambda g_{\beta }\left( s_{\beta }\right) }{V_{\max }}\left\vert w_{\beta
}\right\vert ^{p}dx\right) .  \label{3-9}
\end{eqnarray}%
Clearly, $J_{\beta }\left( t\sqrt{s_{\beta }}w_{\beta },t\sqrt{1-s_{\beta }}%
w_{\beta }\right) =0$ if and only if
\begin{equation*}
\xi \left( t\right) +\frac{1}{4}\int_{\mathbb{R}^{3}}\rho \left( x\right)
\phi _{\rho ,w_{\beta }}w_{\beta }^{2}dx=0.
\end{equation*}%
Let%
\begin{equation}
\xi _{0}\left( t\right) =\frac{t^{-2}}{2}\left\Vert w_{\beta }\right\Vert
_{\lambda }^{2}-\frac{t^{p-4}}{p}\int_{\mathbb{R}^{3}}\frac{\lambda g_{\beta
}\left( s_{\beta }\right) }{V_{\max }}\left\vert w_{\beta }\right\vert
^{p}dx.  \label{3-10}
\end{equation}%
It is not difficult to verify that
\begin{equation*}
\xi _{0}\left( \hat{t}_{0}\right) =0,\ \lim_{t\rightarrow 0^{+}}\xi
_{0}(t)=\infty \ \text{and}\ \lim_{t\rightarrow \infty }\xi _{0}(t)=0,
\end{equation*}%
where $\hat{t}_{0}=\left( \frac{p}{2}\right) ^{1/\left( p-2\right) }.$ By
calculating the derivative of $\xi _{0}(t)$, we find that%
\begin{eqnarray*}
\xi _{0}^{\prime }\left( t\right) &=&-t^{-3}\left\Vert w_{\beta }\right\Vert
_{\lambda }^{2}+\frac{4-p}{p}t^{p-5}\int_{\mathbb{R}^{3}}\frac{\lambda
g_{\beta }\left( s_{\beta }\right) }{V_{\max }}\left\vert w_{\beta
}\right\vert ^{p}dx \\
&=&t^{-3}\left\Vert w_{\beta }\right\Vert _{\lambda }^{2}\left[ \frac{\left(
4-p\right) t^{p-2}}{p}-1\right] ,
\end{eqnarray*}%
which implies that $\xi _{0}\left( t\right) $ is decreasing when $0<t<\left(
\frac{p}{4-p}\right) ^{1/\left( p-2\right) }$ and is increasing when $%
t>\left( \frac{p}{4-p}\right) ^{1/\left( p-2\right) }.$ Then for each $\beta
>\beta _{1},$ it follows from (\ref{3-9}) and (\ref{3-10}) that%
\begin{eqnarray*}
\inf_{t>0}\xi \left( t\right) &\leq &\xi _{0}\left[ \left( \frac{p}{4-p}%
\right) ^{1/\left( p-2\right) }\right] \\
&=&\frac{V_{\max }}{\lambda p}\left( \frac{4-p}{p}\right) ^{\left(
4-p\right) /\left( p-2\right) }\left\Vert w_{\beta }\right\Vert
_{H^{1}}^{2}\left( \frac{4-p}{2}-1\right) \\
&=&-\frac{V_{\max }}{\lambda }\left( \frac{p-2}{2}\right) \left( \frac{4-p}{2%
}\right) ^{\left( 4-p\right) /(p-2)}\left\Vert w_{\beta }\right\Vert
_{H^{1}}^{2} \\
&<&-\frac{4\sqrt[3]{2}\rho _{\max }^{2}}{3\sqrt{3}\pi \lambda ^{3/2}}%
\left\Vert w_{\beta }\right\Vert _{H^{1}}^{4} \\
&<&-\frac{1}{4}\int_{\mathbb{R}^{3}}\rho \left( x\right) \phi _{\rho ,\left(
\sqrt{s_{\beta }}w_{\beta },\sqrt{1-s_{\beta }}w_{\beta }\right) }\left[
\left( \sqrt{s_{\beta }}w_{\beta }\right) ^{2}+\left( \sqrt{1-s_{\beta }}%
w_{\beta }\right) ^{2}\right] dx,
\end{eqnarray*}%
which yields that
\begin{equation*}
J_{\beta }\left( t_{\beta }^{+}\sqrt{s_{\beta }}w_{\beta },t_{\beta }^{+}%
\sqrt{1-s_{\beta }}w_{\beta }\right) =\inf_{t\geq 0}J_{\beta }\left( t\sqrt{%
s_{\beta }}w_{\beta },t\sqrt{1-s_{\beta }}w_{\beta }\right) <0.
\end{equation*}%
This implies that $\left( t_{\beta }^{+}\sqrt{s_{\beta }}w_{\beta },t_{\beta
}^{+}\sqrt{1-s_{\beta }}w_{\beta }\right) \in \mathcal{N}_{\beta }^{\left(
2\right) }\cap \mathbf{H}_{r}.$

By (\ref{4-1}), there exists $\bar{\beta}_{0}=\bar{\beta}_{0}\left( \lambda
,V_{\max },\rho _{\max }\right) \geq \beta _{1}$ such that for every $\beta >%
\bar{\beta}_{0},$%
\begin{eqnarray*}
&&J_{\beta }\left( t_{\beta }^{-}\sqrt{s_{\beta }}w_{\beta },t_{\beta }^{-}%
\sqrt{1-s_{\beta }}w_{\beta }\right) \\
&=&\frac{\left( t_{\beta }^{-}\right) ^{2}}{2}\left\Vert w_{\beta
}\right\Vert _{V}^{2}+\frac{\left( t_{\beta }^{-}\right) ^{4}}{4}\int_{%
\mathbb{R}^{3}}\rho \left( x\right) \phi _{\rho ,w_{\beta }}w_{\beta }^{2}dx-%
\frac{\left( t_{\beta }^{-}\right) ^{p}}{p}\int_{\mathbb{R}^{3}}g_{\beta
}\left( s_{\beta }\right) w_{\beta }^{p}dx \\
&\leq &\left( \frac{\left( t_{\beta }^{-}\right) ^{2}}{2}\left\Vert w_{\beta
}\right\Vert _{\lambda }^{2}+\frac{\left( t_{\beta }^{-}\right) ^{4}\lambda
}{4V_{\max }}\int_{\mathbb{R}^{3}}\rho \left( x\right) \phi _{\rho ,w_{\beta
}}w_{\beta }^{2}dx-\frac{\left( t_{\beta }^{-}\right) ^{p}}{p}\int_{\mathbb{R%
}^{3}}\frac{\lambda g_{\beta }\left( s_{\beta }\right) }{V_{\max }}w_{\beta
}^{p}dx\right) \\
&<&\frac{V_{\max }}{\lambda }\left( \alpha _{\beta }^{\infty }+\frac{4\sqrt[3%
]{2}\rho _{\max }^{2}}{3\sqrt{3}\pi \lambda ^{3/2}}\left( \frac{2}{4-p}%
\right) ^{\frac{4}{p-2}}\left\Vert w_{\beta }\right\Vert _{H^{1}}^{4}\right)
\\
&=&\frac{2\left( p-2\right) V_{\max }}{p\lambda }\left( \frac{V_{\max
}S_{\lambda ,p}^{p}}{\lambda \left( 1+\beta \right) }\right) ^{2/\left(
p-2\right) }+\frac{16\sqrt[3]{2}\rho _{\max }^{2}V_{\max }}{3\sqrt{3}\pi
\lambda ^{5/2}}\left( \frac{2}{4-p}\right) ^{\frac{4}{p-2}}\left( \frac{%
V_{\max }S_{\lambda ,p}^{p}}{\lambda \left( 1+\beta \right) }\right)
^{4/\left( p-2\right) } \\
&<&\min \left\{ \frac{3\sqrt{3}\pi \left( p-2\right) ^{2}\lambda ^{3/2}}{64%
\sqrt[3]{2}p(4-p)\rho _{\max }^{2}},\frac{p-2}{2p}S_{\lambda ,p}^{2p/\left(
p-2\right) }\right\} ,
\end{eqnarray*}%
which implies that $\left( t_{\beta }^{-}\sqrt{s_{\beta }}w_{\beta
},t_{\beta }^{-}\sqrt{1-s_{\beta }}w_{\beta }\right) \in \mathcal{N}_{\beta
}^{\left( 1\right) }\cap \mathbf{H}_{r}.$ The proof is complete.
\end{proof}

\section{The behavior of the energy functional $J_{\beta }$}

\textbf{We are now ready to prove Theorem \ref{t1-2}: }$\left( i\right) $
Define%
\begin{equation*}
\overline{\Lambda }_{0}:=\inf_{u\in \mathbf{H}\setminus \left\{ \left(
0,0\right) \right\} }\frac{\left\Vert \left( u,v\right) \right\Vert
^{2}+\int_{\mathbb{R}^{3}}\rho \left( x\right) \phi _{\rho ,\left(
u,v\right) }\left( u^{2}+v^{2}\right) dx-\int_{\mathbb{R}^{3}}(\left\vert
u\right\vert ^{p}+\left\vert v\right\vert ^{p})dx}{\int_{\mathbb{R}%
^{3}}\left\vert u\right\vert ^{\frac{^{p}}{2}}\left\vert v\right\vert ^{%
\frac{p}{2}}dx}.
\end{equation*}%
By conditions $\left( D1\right) -\left( D2\right) ,$ we have
\begin{equation*}
\left\Vert \left( u,v\right) \right\Vert ^{2}\geq \int_{\mathbb{R}%
^{3}}(|\nabla u|^{2}+\lambda u^{2})dx+\int_{\mathbb{R}^{3}}\left( |\nabla
v|^{2}+\lambda v^{2}\right) dx
\end{equation*}%
and%
\begin{equation*}
\int_{\mathbb{R}^{3}}\rho \left( x\right) \phi _{\rho ,\left( u,v\right)
}\left( u^{2}+v^{2}\right) dx\geq \rho _{\min }\int_{\mathbb{R}^{3}}\phi
_{\rho _{\min },\left( u,v\right) }\left( u^{2}+v^{2}\right) dx.
\end{equation*}%
Following the idea of Lions \cite{Lions}, for all $\left( u,v\right) \in
\mathbf{H}\setminus \left\{ \left( 0,0\right) \right\} $ we have%
\begin{eqnarray}
\sqrt{2}\rho _{\min }\int_{\mathbb{R}^{3}}(\left\vert u\right\vert
^{3}+v^{2}\left\vert u\right\vert )dx &=&\sqrt{2}\int_{\mathbb{R}^{3}}\left(
-\Delta \phi _{\rho _{\min },\left( u,v\right) }\right) \left\vert
u\right\vert dx  \notag \\
&=&\sqrt{2}\int_{\mathbb{R}^{3}}\left\langle \nabla \phi _{\rho _{\min
},\left( u,v\right) },\nabla \left\vert u\right\vert \right\rangle dx  \notag
\\
&\leq &\int_{\mathbb{R}^{3}}\left\vert \nabla u\right\vert ^{2}dx+\frac{\rho
_{\min }}{2}\int_{\mathbb{R}^{3}}\phi _{\rho _{\min },\left( u,v\right)
}\left( u^{2}+v^{2}\right) dx  \label{3-4}
\end{eqnarray}%
and%
\begin{eqnarray}
\sqrt{2}\rho _{\min }\int_{\mathbb{R}^{3}}(u^{2}\left\vert v\right\vert
+\left\vert v\right\vert ^{3})dx &=&\sqrt{2}\int_{\mathbb{R}^{3}}\left(
-\Delta \phi _{\rho _{\min },\left( u,v\right) }\right) \left\vert
v\right\vert dx  \notag \\
&=&\sqrt{2}\int_{\mathbb{R}^{3}}\left\langle \nabla \phi _{\rho _{\min
},\left( u,v\right) },\nabla \left\vert v\right\vert \right\rangle dx  \notag
\\
&\leq &\int_{\mathbb{R}^{3}}\left\vert \nabla v\right\vert ^{2}dx+\frac{\rho
_{\min }}{2}\int_{\mathbb{R}^{3}}\phi _{\rho _{\min },\left( u,v\right)
}\left( u^{2}+v^{2}\right) dx.  \label{3-5}
\end{eqnarray}%
Then it follows from conditions $\left( D1\right) -\left( D2\right) $ and (%
\ref{3-4})--(\ref{3-5}) that%
\begin{eqnarray*}
&&\frac{\left\Vert \left( u,v\right) \right\Vert ^{2}+\int_{\mathbb{R}%
^{3}}\rho \left( x\right) \phi _{\rho ,\left( u,v\right) }\left(
u^{2}+v^{2}\right) dx-\int_{\mathbb{R}^{3}}(\left\vert u\right\vert
^{p}+\left\vert v\right\vert ^{p})dx}{2\int_{\mathbb{R}^{3}}\left\vert
u\right\vert ^{\frac{^{p}}{2}}\left\vert v\right\vert ^{\frac{p}{2}}dx} \\
&\geq &\frac{\int_{\mathbb{R}^{3}}\lambda \left( u^{2}+v^{2}\right) dx+\sqrt{%
2}\rho _{\min }\int_{\mathbb{R}^{3}}(\left\vert u\right\vert ^{3}+\left\vert
v\right\vert ^{3})dx-\int_{\mathbb{R}^{3}}(\left\vert u\right\vert
^{p}+\left\vert v\right\vert ^{p})dx}{2\int_{\mathbb{R}^{3}}\left\vert
u\right\vert ^{\frac{^{p}}{2}}\left\vert v\right\vert ^{\frac{p}{2}}dx} \\
&\geq &\frac{\left( d_{\lambda ,\rho _{\min }}-1\right) \int_{\mathbb{R}%
^{3}}(\left\vert u\right\vert ^{p}+\left\vert v\right\vert ^{p})dx}{\int_{%
\mathbb{R}^{3}}(\left\vert u\right\vert ^{p}+\left\vert v\right\vert ^{p})dx}
\\
&=&d_{\lambda ,\rho _{\min }}-1,
\end{eqnarray*}%
where $d_{\lambda ,\rho _{\min }}:=2^{\frac{p-2}{2}}\left( \frac{\lambda }{%
3-p}\right) ^{3-p}\left( \frac{\rho _{\min }}{p-2}\right) ^{p-2}.$ This
shows that%
\begin{equation*}
\overline{\Lambda }_{0}:=\inf_{u\in \mathbf{H}\setminus \left\{ \left(
0,0\right) \right\} }\frac{\left\Vert \left( u,v\right) \right\Vert
^{2}+\int_{\mathbb{R}^{3}}\rho \left( x\right) \phi _{\rho ,\left(
u,v\right) }\left( u^{2}+v^{2}\right) dx-\int_{\mathbb{R}^{3}}(\left\vert
u\right\vert ^{p}+\left\vert v\right\vert ^{p})dx}{2\int_{\mathbb{R}%
^{3}}\left\vert u\right\vert ^{\frac{^{p}}{2}}\left\vert v\right\vert ^{%
\frac{p}{2}}dx}\geq d_{\lambda ,\rho _{\min }}-1.
\end{equation*}%
Let $\left( u_{0},v_{0}\right) $ be a nontrivial solution of system $%
(HF_{\beta }).$ Thus, according to the definition of $\overline{\Lambda }%
_{0},$ for $\beta <\overline{\Lambda }_{0}$ one has%
\begin{equation*}
0=\left\Vert \left( u_{0},v_{0}\right) \right\Vert ^{2}+\int_{\mathbb{R}%
^{3}}\rho \left( x\right) \phi _{\rho ,\left( u_{0},v_{0}\right) }\left(
u_{0}^{2}+v_{0}^{2}\right) dx-\int_{\mathbb{R}^{3}}F_{\beta }\left(
u_{0},v_{0}\right) dx>0,
\end{equation*}%
which is a contradiction. This shows that system $(HF_{\beta })$ does not
admits any nontrivial solutions for $\beta <\overline{\Lambda }_{0}$.\newline
$\left( ii\right) $ It follows from conditions $\left( D1\right) -\left(
D2\right) $ that%
\begin{equation}
J_{\beta }(u,v)\leq J_{\beta ,V_{\max },\rho _{\max }}(u,v),  \label{1-7}
\end{equation}%
where%
\begin{eqnarray*}
J_{\beta ,V_{\max },\rho _{\max }}(u,v) &:&=\frac{1}{2}\left\Vert \left(
u,v\right) \right\Vert _{V_{\max }}^{2}+\frac{\rho _{\max }}{4}\int_{\mathbb{%
R}^{3}}\phi _{\rho _{\max },\left( u,v\right) }\left( u^{2}+v^{2}\right) dx
\\
&&-\frac{1}{p}\int_{\mathbb{R}^{3}}F_{\beta }\left( u,v\right) dx.
\end{eqnarray*}%
Then by Proposition \ref{p1}, $\Lambda \left( V_{\max },\rho _{\max }\right)
$ is achieved, i.e. there exists $\left( \widehat{u}_{0},\widehat{v}%
_{0}\right) \in \mathbf{H}_{r}\setminus \left\{ \left( 0,0\right) \right\} $
such that
\begin{equation*}
\Lambda \left( V_{\max },\rho _{\max }\right) =\frac{\frac{1}{2}\left\Vert
\left( \widehat{u}_{0},\widehat{v}_{0}\right) \right\Vert _{V_{\max }}^{2}+%
\frac{\rho _{\max }}{4}\int_{\mathbb{R}^{3}}\phi _{\rho _{\max },\left(
\widehat{u}_{0},\widehat{v}_{0}\right) }\left( \widehat{u}_{0}^{2}+\widehat{v%
}_{0}^{2}\right) dx-\frac{1}{p}\int_{\mathbb{R}^{3}}(\left\vert \widehat{u}%
_{0}\right\vert ^{p}+\left\vert \widehat{v}_{0}\right\vert ^{p})dx}{\frac{2}{%
p}\int_{\mathbb{R}^{3}}\left\vert \widehat{u}_{0}\right\vert ^{\frac{^{p}}{2}%
}\left\vert \widehat{v}_{0}\right\vert ^{\frac{p}{2}}dx}>0.
\end{equation*}%
This indicates that $J_{\beta ,V_{\max },\rho _{\max }}(\widehat{u}_{0},%
\widehat{v}_{0})<0$ for $\beta >\Lambda \left( V_{\max },\rho _{\max
}\right) .$ Thus, by adopting the multibump technique developed by Ruiz \cite%
{R2}, one has%
\begin{equation*}
\inf_{\left( u,v\right) \in \mathbf{H}}J_{\beta ,V_{\max },\rho _{\max
}}(u,v)=-\infty ,
\end{equation*}%
which implies that $\inf_{\left( u,v\right) \in \mathbf{H}}J_{\beta
}(u,v)=-\infty $ for $\beta >\Lambda \left( V_{\max },\rho _{\max }\right) ,$
where we have used (\ref{1-7}). The proof is complete.

\section{Existence of vectorial solutions}

\subsection{The case of $2<p<3$}

Similar to (\ref{3-4})--(\ref{3-5}), we have%
\begin{equation}
\frac{1}{\sqrt{8}}\int_{\mathbb{R}^{3}}\rho \left( x\right) (\left\vert
u\right\vert ^{3}+v^{2}\left\vert u\right\vert )dx\leq \frac{1}{4}\int_{%
\mathbb{R}^{3}}\left\vert \nabla u\right\vert ^{2}dx+\frac{1}{8}\int_{%
\mathbb{R}^{3}}\rho \left( x\right) \phi _{\rho ,\left( u,v\right) }\left(
u^{2}+v^{2}\right) dx  \label{2-0}
\end{equation}%
and%
\begin{equation}
\frac{1}{\sqrt{8}}\int_{\mathbb{R}^{3}}\rho \left( x\right) (u^{2}\left\vert
v\right\vert +\left\vert v\right\vert ^{3})dx\leq \frac{1}{4}\int_{\mathbb{R}%
^{3}}\left\vert \nabla v\right\vert ^{2}dx+\frac{1}{8}\int_{\mathbb{R}%
^{3}}\rho \left( x\right) \phi _{\rho ,\left( u,v\right) }\left(
u^{2}+v^{2}\right) dx  \label{2-00}
\end{equation}%
for all $\left( u,v\right) \in \mathbf{H}\setminus \left\{ \left( 0,0\right)
\right\} .$ Then it follows from condition $\left( D3\right) $ and (\ref{2-0}%
)--(\ref{2-00}) that%
\begin{eqnarray}
J_{\beta }(u,v) &\geq &\frac{1}{4}\left\Vert \left( u,v\right) \right\Vert
^{2}+\int_{\mathbb{R}^{3}}\left( \frac{1}{4}V\left( x\right) u^{2}+\frac{1}{%
\sqrt{8}}\rho \left( x\right) \left\vert u\right\vert ^{3}-\frac{1}{p}%
\left\vert u\right\vert ^{p}\right) dx  \notag \\
&&+\int_{\mathbb{R}^{3}}\left( \frac{1}{4}V\left( x\right) v^{2}+\frac{1}{%
\sqrt{8}}\rho \left( x\right) \left\vert v\right\vert ^{3}-\frac{1}{p}%
\left\vert v\right\vert ^{p}\right) dx  \notag \\
&&+\int_{\mathbb{R}^{3}}\frac{1}{\sqrt{8}}\rho \left( x\right)
u^{2}\left\vert v\right\vert +\frac{1}{\sqrt{8}}\rho \left( x\right)
v^{2}\left\vert u\right\vert -\frac{2\beta }{p}\left\vert u\right\vert ^{%
\frac{p}{2}}\left\vert v\right\vert ^{\frac{p}{2}}dx  \notag \\
&\geq &\frac{1}{4}\left\Vert \left( u,v\right) \right\Vert ^{2}+\int_{%
\mathbb{R}^{3}}\left( \frac{1}{4}V\left( x\right) u^{2}+\frac{1}{\sqrt{8}}%
\rho \left( x\right) \left\vert u\right\vert ^{3}-\frac{1+\beta }{p}%
\left\vert u\right\vert ^{p}\right) dx  \notag \\
&&+\int_{\mathbb{R}^{3}}\left( \frac{1}{4}V\left( x\right) v^{2}+\frac{1}{%
\sqrt{8}}\rho \left( x\right) \left\vert v\right\vert ^{3}-\frac{1+\beta }{p}%
\left\vert v\right\vert ^{p}\right) dx.  \label{3-0}
\end{eqnarray}%
Then we have the following results.

\begin{lemma}
\label{m5}Assume that $2<p<3$ and conditions $(D1)-\left( D3\right) $ hold.
Then for every%
\begin{equation*}
\Lambda \left( \lambda _{0},k_{0}\right) <\beta <\frac{p}{2^{\frac{6-p}{2}}}%
\left( \frac{V_{\infty }}{3-p}\right) ^{3-p}\left( \frac{\rho _{\infty }}{p-2%
}\right) ^{p-2}-1,
\end{equation*}%
$J_{\beta }$ is coercive and bounded from below on $\mathbf{H}.$
Furthermore, there holds
\begin{equation*}
-\infty <\alpha _{\beta }:=\inf_{u\in \mathbf{H}}J_{\beta }(u,v)<0.
\end{equation*}
\end{lemma}

\begin{proof}
Define
\begin{equation*}
m_{\beta }\left( x\right) :=\inf_{s\geq 0}\left( \frac{1}{4}V\left( x\right)
s^{2}+\sqrt{\frac{1}{8}}\rho \left( x\right) s^{3}-\frac{1+\beta }{p}%
s^{p}\right)
\end{equation*}%
and%
\begin{equation*}
\mathcal{B}:=\left\{ x\in \mathbb{R}^{3}:V^{3-p}\left( x\right) \rho
^{p-2}\left( x\right) <\frac{2^{\frac{6-p}{2}}}{p}\left( p-2\right)
^{p-2}\left( 3-p\right) ^{3-p}\left( 1+\beta \right) \right\} .
\end{equation*}%
Then for each%
\begin{equation*}
\Lambda \left( \lambda _{0},k_{0}\right) <\beta <\frac{p}{2^{\frac{6-p}{2}}}%
\left( \frac{V_{\infty }}{3-p}\right) ^{3-p}\left( \frac{\rho _{\infty }}{p-2%
}\right) ^{p-2}-1,
\end{equation*}%
according to Remark \ref{R1} and Lemma \ref{g3}, we have%
\begin{equation*}
\inf_{x\in \mathcal{B}}m_{\beta }\left( x\right) <0,
\end{equation*}%
and further get%
\begin{equation}
-\infty <\left\vert \mathcal{B}\right\vert \inf_{x\in \mathcal{B}}m_{\beta
}\left( x\right) \leq \int_{\mathcal{B}}m_{\beta }\left( x\right) dx<0.
\label{3-2}
\end{equation}%
It follows from (\ref{3-0})--(\ref{3-2}) that%
\begin{equation*}
J_{\beta }(u,v)\geq \frac{1}{4}\left\Vert \left( u,v\right) \right\Vert
^{2}+2\int_{\mathcal{B}}m_{\beta }\left( x\right) dx\geq 2\left\vert
\mathcal{B}\right\vert \inf_{x\in \mathcal{B}}m_{\beta }\left( x\right)
>-\infty ,
\end{equation*}%
which shows that $J_{\beta }$ is coercive and bounded from below on $\mathbf{%
H}.$

Since $\left( u_{0},v_{0}\right) \in \mathbf{H}_{r}\setminus \left\{ \left(
0,0\right) \right\} $ is a minimizer of the minimization problem (\ref{2-17}%
) obtained in Proposition \ref{p1} for $\theta =\lambda _{0}$ and $k=k_{0},$
by condition $\left( D3\right) $ one has%
\begin{eqnarray*}
J_{\beta }(u_{0},v_{0}) &=&\frac{1}{2}\left\Vert \left( u_{0},v_{0}\right)
\right\Vert ^{2}+\frac{1}{4}\int_{\mathbb{R}^{3}}\rho \left( x\right) \phi
_{\rho ,\left( u_{0},v_{0}\right) }\left( u_{0}^{2}+v_{0}^{2}\right) dx-%
\frac{1}{p}\int_{\mathbb{R}^{3}}F_{\beta }\left( u_{0},v_{0}\right) dx \\
&\leq &\frac{1}{2}\int_{\mathbb{R}^{3}}(|\nabla u_{0}|^{2}+\lambda
_{0}u_{0}^{2})dx+\frac{1}{2}\int_{\mathbb{R}^{3}}(|\nabla v_{0}|^{2}+\lambda
_{0}v_{0}^{2})dx+\frac{k_{0}}{4}\int_{\mathbb{R}^{3}}\phi _{k_{0},\left(
u_{0},v_{0}\right) }\left( u_{0}^{2}+v_{0}^{2}\right) dx \\
&&-\frac{1}{p}\int_{\mathbb{R}^{3}}F_{\beta }\left( u_{0},v_{0}\right) dx \\
&<&0\text{ for }\beta >\Lambda \left( \lambda _{0},k_{0}\right) ,
\end{eqnarray*}%
which implies that%
\begin{equation*}
-\infty <\alpha _{\beta }:=\inf_{u\in \mathbf{H}}J_{\beta }(u,v)<0.
\end{equation*}%
This completes the proof.
\end{proof}

\begin{lemma}
\label{L2-5}Let $2<p<3$ and $\beta >0.$ Then for every $a,b>0$ satisfying
\begin{equation*}
a^{3-p}b^{p-2}\geq \frac{2^{\frac{6-p}{2}}}{p}\left( 3-p\right) ^{3-p}\left(
p-2\right) ^{p-2}\left( 1+\beta \right) ,
\end{equation*}%
there holds $\inf_{u\in \mathbf{H}}J_{a,b,\beta }^{\infty }(u,v)\geq 0,$
where $J_{a,b,\beta }^{\infty }=J_{\beta }$ with $V(x)\equiv a$ and $\rho
(x)\equiv b.$ Furthermore, if $\left( u_{0},v_{0}\right) $ is a nontrivial
solution of the following system:%
\begin{equation}
\left\{
\begin{array}{ll}
-\Delta u+au+b\phi _{b,\left( u,v\right) }u=\left\vert u\right\vert
^{p-2}u+\beta \left\vert v\right\vert ^{\frac{p}{2}}\left\vert u\right\vert
^{\frac{p}{2}-2}u & \text{ in }\mathbb{R}^{3}, \\
-\Delta v+av+b\phi _{b,\left( u,v\right) }v=\left\vert v\right\vert
^{p-2}v+\beta \left\vert u\right\vert ^{\frac{p}{2}}\left\vert v\right\vert
^{\frac{p}{2}-2}v & \text{ in }\mathbb{R}^{3},%
\end{array}%
\right.   \tag*{$\left( HF_{a,b,\beta }^{\infty }\right) $}
\end{equation}%
then we have $J_{a,b,\beta }^{\infty }(u_{0},v_{0})>0.$
\end{lemma}

\begin{proof}
For each $a,b>0$ satisfying%
\begin{equation*}
a^{3-p}b^{p-2}\geq \frac{2^{\frac{6-p}{2}}}{p}\left( 3-p\right) ^{3-p}\left(
p-2\right) ^{p-2}\left( 1+\beta \right) ,
\end{equation*}%
by Lemma \ref{g3} and (\ref{3-0}), one has%
\begin{eqnarray*}
J_{a,b,\beta }^{\infty }(u,v) &\geq &\frac{1}{4}\left\Vert \left( u,v\right)
\right\Vert _{a}^{2}+\int_{\mathbb{R}^{3}}\left[ \frac{1}{4}au^{2}+\frac{1}{%
\sqrt{8}}b\left\vert u\right\vert ^{3}-\frac{1+\beta }{p}\left\vert
u\right\vert ^{p}\right] dx \\
&&+\int_{\mathbb{R}^{3}}\left[ \frac{1}{4}av^{2}+\frac{1}{\sqrt{8}}%
b\left\vert v\right\vert ^{3}-\frac{1+\beta }{p}\left\vert v\right\vert ^{p}%
\right] dx \\
&\geq &\frac{1}{4}\left\Vert \left( u,v\right) \right\Vert _{a}^{2}\text{
for all }\left( u,v\right) \in \mathbf{H}\setminus \left\{ \left( 0,0\right)
\right\} .
\end{eqnarray*}%
Thus, $\inf_{u\in \mathbf{H}}J_{a,b,\beta }^{\infty }(u,v)\geq 0.$ Clearly,
if $\left( u_{0},v_{0}\right) $ is a nontrivial solution of system $\left(
HF_{a,b,\beta }^{\infty }\right) ,$ then $J_{a,b,\beta }^{\infty
}(u_{0},v_{0})>0.$ This completes the proof.
\end{proof}

\textbf{We are now ready to prove Theorem \ref{t1}: }It follows from Lemma %
\ref{m5} that $J_{\beta }$ is coercive and bounded from below on $\mathbf{H}$
and
\begin{equation*}
-\infty <\alpha _{\beta }:=\inf_{(u,v)\in \mathbf{H}}J_{\beta }(u,v)<0.
\end{equation*}%
Then by the Ekeland variational principle \cite{E}, there exists a sequence $%
\{\left( u_{n},v_{n}\right) \}\subset \mathbf{H}$ such that
\begin{equation*}
J_{\beta }(u_{n},v_{n})=\alpha _{\beta }+o(1)\text{ and }J_{\beta }^{\prime
}(u_{n},v_{n})=o(1)\text{ in }\mathbf{H}^{-1}.
\end{equation*}%
Adopting the argument used in \cite[Theorem 4.3]{R1}, there exist a
subsequence $\{\left( u_{n},v_{n}\right) \}\subset \mathbf{H}$ and $\left(
u_{\beta }^{\left( 1\right) },v_{\beta }^{\left( 1\right) }\right) \in
\mathbf{H}$ such that as $n\rightarrow \infty ,$%
\begin{equation}
\begin{array}{l}
\left( u_{n},v_{n}\right) \rightharpoonup \left( u_{\beta }^{\left( 1\right)
},v_{\beta }^{\left( 1\right) }\right) \text{ weakly in }\mathbf{H}, \\
\left( u_{n},v_{n}\right) \rightarrow \left( u_{\beta }^{\left( 1\right)
},v_{\beta }^{\left( 1\right) }\right) \text{ strongly in }L_{loc}^{p}(%
\mathbb{R}^{3})\times L_{loc}^{p}(\mathbb{R}^{3}), \\
\left( u_{n},v_{n}\right) \rightarrow \left( u_{\beta }^{\left( 1\right)
},v_{\beta }^{\left( 1\right) }\right) \text{ a.e. in }\mathbb{R}^{3}.%
\end{array}
\label{2-1}
\end{equation}%
Moreover, $\left( u_{\beta }^{\left( 1\right) },v_{\beta }^{\left( 1\right)
}\right) $ is a solution of system $(HF_{\beta }).$ Next, we show that%
\begin{equation*}
\left( u_{n},v_{n}\right) \rightarrow \left( u_{\beta }^{\left( 1\right)
},v_{\beta }^{\left( 1\right) }\right) \text{ in }\mathbf{H}.
\end{equation*}%
To this end, we suppose the contrary. Then it has
\begin{equation}
\left\Vert \left( u_{\beta }^{\left( 1\right) },v_{\beta }^{\left( 1\right)
}\right) \right\Vert ^{2}<\liminf_{n\rightarrow \infty }\left\Vert \left(
u_{n},v_{n}\right) \right\Vert ^{2}.  \label{2-4}
\end{equation}%
Let $\left( \widetilde{u}_{n},\widetilde{v}_{n}\right) =\left(
u_{n}-u_{\beta }^{\left( 1\right) },v_{n}-v_{\beta }^{\left( 1\right)
}\right) .$ Then by (\ref{2-1})--(\ref{2-4}), we can assume that%
\begin{equation*}
\left( \widetilde{u}_{n},\widetilde{v}_{n}\right) \rightharpoonup \left(
0,0\right) \text{ weakly in }\mathbf{H}
\end{equation*}%
and%
\begin{equation}
\left\Vert \left( \widetilde{u}_{n},\widetilde{v}_{n}\right) \right\Vert
>c_{0}\text{ for some }c_{0}>0.  \label{2-9}
\end{equation}%
Moreover, similar to the argument in \cite[Lemma 4.1]{CV}, it follows from
the Brezis-Lieb Lemma \cite{BLi} that%
\begin{equation}
J_{\beta }(u_{n},v_{n})=J_{\beta }(u_{\beta }^{\left( 1\right) },v_{\beta
}^{\left( 1\right) })+J_{V_{\infty },\rho _{\infty },\beta }^{\infty }(%
\widetilde{u}_{n},\widetilde{v}_{n})+o\left( 1\right)  \label{2-21}
\end{equation}%
and
\begin{equation*}
\left( J_{V_{\infty },\rho _{\infty },\beta }^{\infty }\right) ^{\prime }(%
\widetilde{u}_{n},\widetilde{v}_{n})=o\left( 1\right) \text{ in }\mathbf{H}%
^{-1}.
\end{equation*}%
Thus, by (\ref{2-9}), Remark \ref{R1} and Lemma \ref{L2-5}, for $n$ large
enough one has
\begin{equation*}
J_{\beta }(u_{\beta }^{\left( 1\right) },v_{\beta }^{\left( 1\right)
})+J_{V_{\infty },\rho _{\infty },\beta }^{\infty }(\widetilde{u}_{n},%
\widetilde{v}_{n})>\alpha _{\beta }+\delta _{0}\text{ for some }\delta
_{0}>0,
\end{equation*}%
which contradicts to (\ref{2-21}). This implies that
\begin{equation*}
\left( u_{n},v_{n}\right) \rightarrow \left( u_{\beta }^{\left( 1\right)
},v_{\beta }^{\left( 1\right) }\right) \text{ in }\mathbf{H}
\end{equation*}%
and%
\begin{equation*}
J_{\beta }\left( u_{\beta }^{\left( 1\right) },v_{\beta }^{\left( 1\right)
}\right) =\alpha _{\beta }=\inf_{u\in \mathbf{H}}J_{\beta }(u,v)<0.
\end{equation*}%
Moreover, by Lemma \ref{L2-4} it follows that $u_{\beta }^{\left( 1\right)
}\neq 0,v_{\beta }^{\left( 1\right) }\neq 0,$ and so is $\left( \left\vert
u_{\beta }^{\left( 1\right) }\right\vert ,\left\vert v_{\beta }^{\left(
1\right) }\right\vert \right) .$ Thus, we may assume that $\left( u_{\beta
}^{\left( 1\right) },v_{\beta }^{\left( 1\right) }\right) $ is a vectorial
positive ground state solution of system $\left( HF_{\beta }\right) .$ This
completes the proof.

\subsection{The case of $3\leq p<4$}

First of all, we define the Palais--Smale (simply by (PS)) sequences and
(PS)--conditions in $\mathbf{H}$ for $J_{\beta }$ as follows.

\begin{definition}
$(i)$ For $\alpha \in \mathbb{R}\mathbf{,}$ a sequence $\left\{ \left(
u_{n},v_{n}\right) \right\} $ is a $(PS)_{\alpha }$--sequence in $\mathbf{H}$
for $J_{\beta }$ if $J_{\beta }(u_{n},v_{n})=\alpha +o(1)\;$and$\;J_{\beta
}^{\prime }(u_{n},v_{n})=o(1)\;$strongly in $\mathbf{H}^{-1}$ as $%
n\rightarrow \infty .$\newline
$(ii)$ We say that $J_{\beta }$ satisfies the $(PS)_{\alpha }$--condition in
$\mathbf{H}$ if every (PS)$_{\alpha }$--sequence in $\mathbf{H}$ for $%
J_{\beta }$ contains a convergent subsequence.
\end{definition}

\begin{proposition}
\label{P3}Assume that $3\leq p<4,\beta >0$ and conditions $\left( D5\right)
-\left( D6\right) $ hold. Let $\left( u_{0},v_{0}\right) $ be a nontrivial
solution of system $(HF_{\beta }).$ Then we have
\begin{equation*}
J_{\beta }\left( u_{0},v_{0}\right) \geq \frac{d_{0}\left( p-2\right) }{%
2\left( 6-p\right) }\int_{\mathbb{R}^{3}}(u_{0}^{2}+v_{0}^{2})dx>0.
\end{equation*}
\end{proposition}

\begin{proof}
Let $\left( u_{0},v_{0}\right) $ be a nontrivial solution of system $%
(HF_{\beta }).$ Then $\left( u_{0},v_{0}\right) $ satisfies the following
identity:%
\begin{equation}
\left\Vert \left( u_{0},v_{0}\right) \right\Vert ^{2}+\int_{\mathbb{R}%
^{3}}\rho \left( x\right) \phi _{\rho ,\left( u_{0},v_{0}\right) }\left(
u_{0}^{2}+v_{0}^{2}\right) dx-\int_{\mathbb{R}^{3}}F_{\beta }\left(
u_{0},v_{0}\right) dx=0.  \label{2-6}
\end{equation}%
Following the argument of \cite[Lemma 3.1]{DM1}, it is not difficult to
verify that $\left( u_{0},v_{0}\right) $ also satisfies the following
Pohozaev type identity:%
\begin{eqnarray}
&&\frac{1}{2}\int_{\mathbb{R}^{3}}(\left\vert \nabla u_{0}\right\vert
^{2}+\left\vert \nabla v_{0}\right\vert ^{2})dx+\frac{1}{2}\int_{\mathbb{R}%
^{3}}(3V\left( x\right) +\langle \nabla V(x),x\rangle
)(u_{0}^{2}+v_{0}^{2})dx  \notag \\
&&+\frac{1}{4}\int_{\mathbb{R}^{3}}(5\rho \left( x\right) +2\langle \nabla
\rho (x),x\rangle )\phi _{\rho ,\left( u_{0},v_{0}\right) }\left(
u_{0}^{2}+v_{0}^{2}\right) dx  \notag \\
&=&\frac{3}{p}\int_{\mathbb{R}^{3}}F_{\beta }\left( u_{0},v_{0}\right) dx.
\label{2-7}
\end{eqnarray}%
Then it follows from conditions $\left( D5\right) -\left( D6\right) $ and (%
\ref{2-6})--(\ref{2-7}) that%
\begin{eqnarray*}
J_{\beta }\left( u_{0},v_{0}\right) &=&\frac{p-2}{2p}\left\Vert \left(
u_{0},v_{0}\right) \right\Vert ^{2}-\frac{4-p}{4p}\int_{\mathbb{R}^{3}}\rho
\left( x\right) \phi _{\rho ,\left( u_{0},v_{0}\right) }\left(
u_{0}^{2}+v_{0}^{2}\right) dx \\
&=&\frac{p-2}{2\left( 6-p\right) }\int_{\mathbb{R}^{3}}\left[ 2V\left(
x\right) +\langle \nabla V(x),x\rangle \right] (u_{0}^{2}+v_{0}^{2})dx \\
&&+\frac{1}{2\left( 6-p\right) }\int_{\mathbb{R}^{3}}\left[ 2\left(
p-3\right) \rho \left( x\right) +\left( p-2\right) \langle \nabla \rho
(x),x\rangle \right] \phi _{\rho ,\left( u_{0},v_{0}\right) }\left(
u_{0}^{2}+v_{0}^{2}\right) dx \\
&\geq &\frac{d_{0}\left( p-2\right) }{2\left( 6-p\right) }\int_{\mathbb{R}%
^{3}}(u_{0}^{2}+v_{0}^{2})dx>0.
\end{eqnarray*}%
This completes the proof.
\end{proof}

Define
\begin{equation*}
\mathbb{A}_{\beta }:=\left\{ \left( u,v\right) \in \mathbf{H}\setminus
\left\{ \left( 0,0\right) \right\} :\left( u,v\right) \text{ is a solution
of system }\left( HF_{\beta }\right) \text{ with }J_{\beta }\left(
u,v\right) <\gamma \left( \beta \right) \right\} ,
\end{equation*}%
where $\gamma \left( \beta \right) >0$ as in (\ref{3-3}). Clearly, $\mathbb{A%
}_{\beta }\subset \mathcal{N}_{\beta }\left[ \frac{3\sqrt{3}\pi \left(
p-2\right) ^{2}\lambda ^{3/2}}{64\sqrt[3]{2}p(4-p)\rho _{\max }^{2}}\right]
. $ Furthermore, we have the following result.

\begin{proposition}
\label{t6}Assume that $3\leq p<4,\beta >0$ and conditions $\left( D1\right)
-\left( D4\right) $ hold. Let $\left( u_{\beta },v_{\beta }\right) $ be a
nontrivial solution of system $(HF_{\beta }).$ Then we have\newline
$\left( i\right) $ If $\frac{1+\sqrt{73}}{3}\leq p<4,$ then $\left( u_{\beta
},v_{\beta }\right) \in \mathcal{N}_{\beta }^{-}.$\newline
$\left( ii\right) $ If $3\leq p<\frac{1+\sqrt{73}}{3}$ and $\left( u_{\beta
},v_{\beta }\right) \in \mathbb{A}_{\beta },$ then there exists $\widehat{%
\beta }_{0}>\frac{p-2}{2}$ such that for every $\beta >\widehat{\beta }_{0},$
there holds $\left( u_{\beta },v_{\beta }\right) \in \mathcal{N}_{\beta
}^{-}.$
\end{proposition}

\begin{proof}
Let $\left( u_{\beta },v_{\beta }\right) $ be a nontrivial solution of
system $(HF_{\beta }).$ By Lemma \ref{L2-3}, we have%
\begin{equation}
\int_{\mathbb{R}^{3}}\rho \left( x\right) \phi _{\rho ,\left(
u_{0},v_{0}\right) }\left( u_{0}^{2}+v_{0}^{2}\right) dx\leq \frac{16\sqrt[3]%
{2}\rho _{\max }^{2}}{3\sqrt{3}\pi }\left( \frac{2(6-p)}{d_{0}\left(
p-2\right) }\theta \right) ^{3/2}\left[ \int_{\mathbb{R}^{3}}(|\nabla
u_{0}|^{2}+|\nabla v_{0}|^{2})dx\right] ^{1/2},  \label{2-8}
\end{equation}%
where $\theta :=J_{\beta }\left( u_{0},v_{0}\right) .$ We now define%
\begin{equation*}
\begin{array}{ll}
z_{1}=\int_{\mathbb{R}^{3}}(|\nabla u_{0}|^{2}+|\nabla v_{0}|^{2})dx, &
z_{2}=\int_{\mathbb{R}^{3}}V\left( x\right) \left(
u_{0}^{2}+v_{0}^{2}\right) dx, \\
z_{3}=\int_{\mathbb{R}^{3}}\langle \nabla V(x),x\rangle \left(
u_{0}^{2}+v_{0}^{2}\right) dx, & z_{4}=\int_{\mathbb{R}^{3}}\rho \left(
x\right) \phi _{\rho ,\left( u_{0},v_{0}\right) }\left(
u_{0}^{2}+v_{0}^{2}\right) dx, \\
z_{5}=\int_{\mathbb{R}^{3}}\langle \nabla \rho (x),x\rangle \phi _{\rho
,\left( u_{0},v_{0}\right) }\left( u_{0}^{2}+v_{0}^{2}\right) dx, &
z_{6}=\int_{\mathbb{R}^{3}}F_{\beta }\left( u_{0},v_{0}\right) dx.%
\end{array}%
\end{equation*}%
Then from (\ref{2-6})--(\ref{2-7}) it follows that%
\begin{equation}
\left\{
\begin{array}{l}
\frac{1}{2}z_{1}+\frac{1}{2}z_{2}+\frac{1}{4}z_{4}-\frac{1}{p}z_{6}=\theta ,
\\
z_{1}+z_{2}+z_{4}-z_{6}=0, \\
\frac{1}{2}z_{1}+\frac{3}{2}z_{2}+\frac{1}{2}z_{3}+\frac{5}{4}z_{4}+\frac{1}{%
2}z_{5}-\frac{3}{p}z_{6}=0, \\
z_{i}>0\text{ for }i=1,2,3,4,5,6.%
\end{array}%
\right.  \label{2-10}
\end{equation}%
Moreover, by Proposition \ref{P3} and (\ref{2-8}), we have%
\begin{equation*}
\theta \geq \frac{d_{0}\left( p-2\right) }{2\left( 6-p\right) }\int_{\mathbb{%
R}^{3}}(u_{0}^{2}+v_{0}^{2})dx>0
\end{equation*}%
and%
\begin{equation}
z_{4}^{2}\leq \left( \frac{16\sqrt[3]{2}\rho _{\max }^{2}}{3\sqrt{3}\pi }%
\right) ^{2}\left( \frac{2(6-p)}{d_{0}\left( p-2\right) }\theta \right)
^{3}z_{1}.  \label{2-12}
\end{equation}%
Note that the general solution of system (\ref{2-10}) is%
\begin{equation}
\left[
\begin{array}{c}
z_{1} \\
z_{2} \\
z_{3} \\
z_{4} \\
z_{5} \\
z_{6}%
\end{array}%
\right] =\frac{\theta }{p-2}\left[
\begin{array}{c}
3(p-2) \\
6-p \\
0 \\
0 \\
0 \\
2p%
\end{array}%
\right] +t\left[
\begin{array}{c}
p-2 \\
-2\left( p-3\right) \\
0 \\
2\left( p-2\right) \\
0 \\
p%
\end{array}%
\right] +s\left[
\begin{array}{c}
1 \\
-1 \\
2 \\
0 \\
0 \\
0%
\end{array}%
\right] +r\left[
\begin{array}{c}
1 \\
-1 \\
0 \\
0 \\
2 \\
0%
\end{array}%
\right] ,  \label{2-3}
\end{equation}%
where $t,s,r\in \mathbb{R}.$ From (\ref{2-3}), we obtain that $z_{i}>0$ $%
(i=1,2,3,4,5,6)$ provided that the parameters $t,s,r>0$ satisfy%
\begin{equation}
2\left( p-3\right) t+s+r<\frac{6-p}{p-2}\theta  \label{2-13}
\end{equation}%
or
\begin{equation}
0<t<\frac{6-p}{2(p-3)\left( p-2\right) }\theta .  \label{2-11}
\end{equation}%
Using (\ref{2-3}) again, we get
\begin{equation}
-\left( p-2\right) \left( z_{1}+z_{2}\right) +\left( 4-p\right)
z_{4}=-2p\theta +(p-2)(4-p)t.  \label{2-18}
\end{equation}%
$\left( i\right) $ $\frac{1+\sqrt{73}}{3}\leq p<4.$ By (\ref{2-11})--(\ref%
{2-18}), one has%
\begin{eqnarray*}
-\left( p-2\right) \left( z_{1}+z_{2}\right) +\left( 4-p\right) z_{4}
&=&-2p\theta +(p-2)(4-p)t \\
&<&-2p\theta +\frac{(4-p)\left( 6-p\right) }{2(p-3)}\theta \\
&\leq &0.
\end{eqnarray*}%
This shows that%
\begin{equation*}
\phi _{\beta ,\left( u_{0},v_{0}\right) }^{\prime \prime }\left( 1\right)
=-\left( p-2\right) \left\Vert \left( u_{0},v_{0}\right) \right\Vert
^{2}+\left( 4-p\right) \int_{\mathbb{R}^{3}}\rho \left( x\right) \phi _{\rho
,\left( u_{0},v_{0}\right) }\left( u_{0}^{2}+v_{0}^{2}\right) dx<0,
\end{equation*}%
leading to $\left( u_{0},v_{0}\right) \in \mathcal{N}_{\beta }^{-}.$\newline
$\left( ii\right) $ $3\leq p<\frac{1+\sqrt{73}}{3}.$ Substituting (\ref{2-3}%
) and (\ref{2-13}) into (\ref{2-12}), we have%
\begin{eqnarray}
&&4(p-2)^{2}t^{2}-\left( 4-p\right) \left( \frac{16\sqrt[3]{2}\rho _{\max
}^{2}}{3\sqrt{3}\pi }\right) ^{2}\left( \frac{2(6-p)}{d_{0}\left( p-2\right)
}\right) ^{3}\theta ^{3}t  \notag \\
&<&\frac{2p}{p-2}\left( \frac{16\sqrt[3]{2}\rho _{\max }^{2}}{3\sqrt{3}\pi }%
\right) ^{2}\left( \frac{2(6-p)}{d_{0}\left( p-2\right) }\right) ^{3}\theta
^{4}.  \label{2-14}
\end{eqnarray}%
Since $t>0$, it follows from (\ref{2-14}) that%
\begin{equation*}
4(p-2)^{2}t^{2}-\left( 4-p\right) A\theta ^{3}t-\frac{2p}{p-2}A\theta ^{4}<0,
\end{equation*}%
where $A:=\left( \frac{16\sqrt[3]{2}\rho _{\max }^{2}}{3\sqrt{3}\pi }\right)
^{2}\left( \frac{2(6-p)}{d_{0}\left( p-2\right) }\right) ^{3}>0$. This
implies that
\begin{equation}
0<t<\frac{\left( 4-p\right) A\theta ^{3}+\sqrt{\left( 4-p\right)
^{2}A^{2}\theta ^{6}+32p(p-2)A\theta ^{4}}}{8(p-2)^{2}}.  \label{2-15}
\end{equation}%
Then it follows from (\ref{2-18}) and (\ref{2-15}) that%
\begin{eqnarray}
&&\frac{-\left( p-2\right) \left( z_{1}+z_{2}\right) +\left( 4-p\right) z_{4}%
}{\theta }  \notag \\
&\leq &-2p+\frac{(4-p)\left( A\theta ^{2}\left( 4-p\right) +\sqrt{%
A^{2}\left( 4-p\right) ^{2}\theta ^{4}+32p(p-2)A\theta ^{2}}\right) }{%
8\left( p-2\right) }.  \label{2-16}
\end{eqnarray}%
Now, we prove that there exists a constant $\widehat{\beta }_{0}>\frac{p-2}{2%
}$ such that for all $\beta >\widehat{\beta }_{0},$%
\begin{equation}
A\theta ^{2}\left( 4-p\right) ^{2}+\sqrt{A^{2}\theta ^{4}\left( 4-p\right)
^{4}+32p(p-2)A\theta ^{2}\left( 4-p\right) ^{2}}<16p\left( p-2\right) .
\label{2-20}
\end{equation}%
Indeed, let%
\begin{equation*}
\Lambda _{\beta }:=A\theta ^{2}\left( 4-p\right) ^{2}.
\end{equation*}%
Clearly,%
\begin{equation*}
\Lambda _{\beta }<A\gamma ^{2}\left( \beta \right) \left( 4-p\right) ^{2}
\end{equation*}%
and
\begin{eqnarray*}
&&A\theta ^{2}\left( 4-p\right) ^{2}+\sqrt{A^{2}\theta ^{4}\left( 4-p\right)
^{4}+32p(p-2)A\theta ^{2}\left( 4-p\right) ^{2}} \\
&=&\Lambda _{\beta }+\sqrt{\Lambda _{\beta }^{2}+32p(p-2)\Lambda _{\beta }}.
\end{eqnarray*}%
We note that if $\Lambda _{\beta }<4p\left( p-2\right) ,$ then there holds%
\begin{equation*}
\Lambda _{\beta }+\sqrt{\Lambda _{\beta }^{2}+32p(p-2)\Lambda _{\beta }}%
<16p\left( p-2\right) .
\end{equation*}%
Next, we only claim that there exists a constant $\widehat{\beta }_{0}>\frac{%
p-2}{2}$ such that
\begin{equation}
A\gamma ^{2}\left( \beta \right) \left( 4-p\right) ^{2}<4p\left( p-2\right)
\text{ for all }\beta >\widehat{\beta }_{0}.  \label{7-4}
\end{equation}%
In fact, since%
\begin{equation*}
\gamma \left( \beta \right) :=\frac{V_{\max }}{\lambda }\left[ \frac{2\left(
p-2\right) }{p}\left( \frac{V_{\max }S_{\lambda ,p}^{p}}{\lambda \left(
1+\beta \right) }\right) ^{2/\left( p-2\right) }+\frac{16\sqrt[3]{2}\rho
_{\max }^{2}}{3\sqrt{3}\pi \lambda ^{3/2}}\left( \frac{2}{4-p}\right) ^{%
\frac{4}{p-2}}\left( \frac{V_{\max }S_{\lambda ,p}^{p}}{\lambda \left(
1+\beta \right) }\right) ^{4/\left( p-2\right) }\right] ,
\end{equation*}%
there exists $\widehat{\beta }_{0}=\widehat{\beta }_{0}\left( \lambda
,V_{\max },\rho _{\max }\right) >\frac{p-2}{2}$ such that (\ref{7-4}) holds,
and so (\ref{2-20}) holds. Hence, by (\ref{2-16})--(\ref{2-20}) one has%
\begin{equation*}
-\left( p-2\right) \left( z_{1}+z_{2}\right) +\left( 4-p\right) z_{4}<0\text{
for all }\beta >\widehat{\beta }_{0}.
\end{equation*}%
This shows that%
\begin{equation*}
\phi _{\beta ,\left( u_{0},v_{0}\right) }^{\prime \prime }\left( 1\right)
=-\left( p-2\right) \left\Vert \left( u_{0},v_{0}\right) \right\Vert
^{2}+\left( 4-p\right) \int_{\mathbb{R}^{3}}\rho \left( x\right) \phi _{\rho
,\left( u_{0},v_{0}\right) }\left( u_{0}^{2}+v_{0}^{2}\right) dx<0,
\end{equation*}%
leading to $\left( u_{0},v_{0}\right) \in \mathcal{N}_{\beta }^{-}.$
Therefore, we have $\mathbb{A}_{\beta }\subset \mathcal{N}_{\beta }^{-}.$
This completes the proof.
\end{proof}

Define
\begin{equation*}
\alpha _{\beta }^{-}:=\inf_{\left( u,v\right) \in \mathcal{N}_{\beta
}^{\left( 1\right) }}J_{\beta }\left( u,v\right) \text{ for }2<p<4.
\end{equation*}%
Clearly, $\alpha _{\beta }^{-}=\inf_{u\in \mathcal{N}_{\beta }^{-}}J_{\beta
}\left( u,v\right) .$ It follows from Lemmas \ref{g5} and \ref{g4} that for
every $\beta >\bar{\beta}_{0},$
\begin{equation}
\frac{p-2}{4p}C_{\beta }^{-1/\left( p-2\right) }<\alpha _{\beta }^{-}<\min
\left\{ \frac{3\sqrt{3}\pi \left( p-2\right) ^{2}\lambda ^{3/2}}{64\sqrt[3]{2%
}p(4-p)\rho _{\max }^{2}},\frac{p-2}{2p}S_{\lambda ,p}^{2p/\left( p-2\right)
}\right\} \text{ for }2<p<4.  \label{7-6}
\end{equation}%
Next, we consider the limiting problem as follows:%
\begin{equation}
\left\{
\begin{array}{ll}
-\Delta u+V_{\infty }u+\rho _{\infty }\phi _{\rho _{\infty },\left(
u,v\right) }u=\left\vert u\right\vert ^{p-2}u+\beta \left\vert v\right\vert
^{\frac{p}{2}}\left\vert u\right\vert ^{\frac{p}{2}-2}u & \text{ in }\mathbb{%
R}^{3}, \\
-\Delta v+V_{\infty }v+\rho _{\infty }\phi _{\rho _{\infty },\left(
u,v\right) }v=\left\vert v\right\vert ^{p-2}v+\beta \left\vert u\right\vert
^{\frac{p}{2}}\left\vert v\right\vert ^{\frac{p}{2}-2}v & \text{ in }\mathbb{%
R}^{3}.%
\end{array}%
\right.
\tag*{$\left(
HF_{V_{\infty
},\rho
_{\infty
},\beta
}^{%
\infty
}%
\right)
$}
\end{equation}%
Using the similar arguments in \cite[ Theorems 1.5 and 4.9]{SW}, there
exists a constant $\beta _{0}^{\infty }>\frac{p-2}{2}$ such that for all $%
\beta >\beta _{0}^{\infty },$ system $\left( HF_{V_{\infty },\rho _{\infty
},\beta }^{\infty }\right) $ has a vectorial positive solution $\left(
u_{\beta }^{\infty },v_{\beta }^{\infty }\right) \in \mathcal{N}_{\beta
,\infty }^{\left( 1\right) }$ satisfying $J_{V_{\infty },\rho _{\infty
},\beta }^{\infty }(u_{\beta }^{\infty },v_{\beta }^{\infty })=\alpha
_{\beta ,\infty }^{-}<\frac{p-2}{2p}S_{\lambda ,p}^{2p/\left( p-2\right) },$
where $\mathcal{N}_{\beta ,\infty }^{\left( 1\right) }=\mathcal{N}_{\beta
}^{\left( 1\right) },$ $J_{V_{\infty },\rho _{\infty },\beta }^{\infty
}=J_{\beta }$ and $\alpha _{\beta ,\infty }^{-}=\alpha _{\beta }^{-}$ with $%
V(x)\equiv V_{\infty }$ and $\rho (x)\equiv \rho _{\infty }.$ Furthermore,
we have the following results.

\begin{lemma}
\label{L5.5}Assume that $2<p<4$ and conditions $\left( D1\right) -\left(
D2\right) $ and $\left( D4\right) $ hold. Then for all $\beta >\beta
_{0}^{\infty },$
\begin{equation*}
\alpha _{\beta }^{-}<\alpha _{\beta ,\infty }^{-}<\frac{p-2}{2p}S_{\lambda
,p}^{2p/\left( p-2\right) }.
\end{equation*}
\end{lemma}

\begin{proof}
Let $\left( u_{\beta }^{\infty },v_{\beta }^{\infty }\right) \in \mathcal{N}%
_{\beta ,\infty }^{\left( 1\right) }$ be a vectorial positive solution of
system $\left( HF_{V_{\infty },\rho _{\infty },\beta }^{\infty }\right) $
satisfying%
\begin{equation*}
\left\Vert \left( u_{\beta }^{\infty },v_{\beta }^{\infty }\right)
\right\Vert _{V_{\infty }}<\left( \frac{3\sqrt{3}\pi \left( p-2\right)
\lambda ^{3/2}}{16\sqrt[3]{2}(4-p)\rho _{\infty }^{2}}\right) ^{1/2}
\label{7-7}
\end{equation*}%
and
\begin{equation*}
J_{V_{\infty },\rho _{\infty },\beta }^{\infty }(u_{\beta }^{\infty
},v_{\beta }^{\infty })=\alpha _{\beta ,\infty }^{-}<\frac{p-2}{2p}%
S_{\lambda ,p}^{2p/\left( p-2\right) }.  \label{7-8}
\end{equation*}%
Then we have
\begin{equation}
\left\Vert \left( u_{\beta }^{\infty },v_{\beta }^{\infty }\right)
\right\Vert <\left\Vert \left( u_{\beta }^{\infty },v_{\beta }^{\infty
}\right) \right\Vert _{V_{\infty }}<\left( \frac{3\sqrt{3}\pi \left(
p-2\right) \lambda ^{3/2}}{16\sqrt[3]{2}(4-p)\rho _{\infty }^{2}}\right)
^{1/2}  \label{7-9}
\end{equation}%
and%
\begin{eqnarray}
&&\left\Vert \left( u_{\beta }^{\infty },v_{\beta }^{\infty }\right)
\right\Vert ^{2}+\int_{\mathbb{R}^{3}}\rho \left( x\right) \phi _{\rho
,\left( u_{\beta }^{\infty },v_{\beta }^{\infty }\right) }\left[ \left(
u_{\beta }^{\infty }\right) ^{2}+\left( v_{\beta }^{\infty }\right) ^{2}%
\right] dx-\int_{\mathbb{R}^{N}}F_{\beta }\left( u_{\beta }^{\infty
},v_{\beta }^{\infty }\right) dx  \notag \\
&<&\left\Vert \left( u_{\beta }^{\infty },v_{\beta }^{\infty }\right)
\right\Vert _{V_{\infty }}^{2}+\int_{\mathbb{R}^{3}}\rho _{\infty }\phi
_{\rho _{\infty },\left( u_{\beta }^{\infty },v_{\beta }^{\infty }\right) }%
\left[ \left( u_{\beta }^{\infty }\right) ^{2}+\left( v_{\beta }^{\infty
}\right) ^{2}\right] dx-\int_{\mathbb{R}^{N}}F_{\beta }\left( u_{\beta
}^{\infty },v_{\beta }^{\infty }\right) dx.  \label{7-11}
\end{eqnarray}%
Moreover, we also have%
\begin{equation}
J_{\beta }(u_{\beta }^{\infty },v_{\beta }^{\infty })<J_{V_{\infty },\rho
_{\infty },\beta }^{\infty }(u_{\beta }^{\infty },v_{\beta }^{\infty
})<\alpha _{\beta ,\infty }^{-}<\frac{p-2}{2p}S_{\lambda ,p}^{2p/\left(
p-2\right) }  \label{7-12}
\end{equation}%
and%
\begin{equation*}
J_{\beta }(tu_{\beta }^{\infty },tv_{\beta }^{\infty })<J_{V_{\infty },\rho
_{\infty },\beta }^{\infty }(tu_{\beta }^{\infty },tv_{\beta }^{\infty })%
\text{ for all }t>0.  \label{7-13}
\end{equation*}%
Since $\phi _{\beta ,\left( u_{\beta }^{\infty },v_{\beta }^{\infty }\right)
}\left( t\right) =J_{\beta }(tu_{\beta }^{\infty },tv_{\beta }^{\infty })$
is increasing for $t>0$ small enough, it follows from (\ref{7-11})--(\ref%
{7-12}) that%
\begin{equation*}
\phi _{\beta ,\left( u_{\beta }^{\infty },v_{\beta }^{\infty }\right)
}^{\prime }\left( 1\right) <\frac{d}{dt}|_{t=1}J_{V_{\infty },\rho _{\infty
},\beta }^{\infty }(tu_{\beta }^{\infty },tv_{\beta }^{\infty })=0,
\end{equation*}%
which implies that there exists $0<t_{0,\infty }<1$ such that $\left(
t_{0,\infty }u_{\beta }^{\infty },t_{0,\infty }v_{\beta }^{\infty }\right)
\in \mathcal{N}_{\beta ,\infty }$ and%
\begin{eqnarray*}
J_{\beta }\left( t_{0,\infty }u_{\beta }^{\infty },t_{0,\infty }v_{\beta
}^{\infty }\right) &<&J_{V_{\infty },\rho _{\infty },\beta }^{\infty }\left(
t_{0,\infty }u_{\beta }^{\infty },t_{0,\infty }v_{\beta }^{\infty }\right)
\notag \\
&<&J_{V_{\infty },\rho _{\infty },\beta }^{\infty }\left( u_{\beta }^{\infty
},v_{\beta }^{\infty }\right) <\alpha _{\beta ,\infty }^{-}<\frac{p-2}{2p}%
S_{\lambda ,p}^{2p/\left( p-2\right) },  \label{7-14}
\end{eqnarray*}%
where $\mathcal{N}_{\beta ,\infty }=\mathcal{N}_{\beta }$ with $V(x)\equiv
V_{\infty }$ and $\rho (x)\equiv \rho _{\infty }.$ Moreover, by (\ref{7-9}),
one has%
\begin{equation*}
\left\Vert \left( t_{0,\infty }u_{\beta }^{\infty },t_{0,\infty }v_{\beta
}^{\infty }\right) \right\Vert <\left\Vert \left( u_{\beta }^{\infty
},v_{\beta }^{\infty }\right) \right\Vert <\left( \frac{3\sqrt{3}\pi \left(
p-2\right) \lambda ^{3/2}}{16\sqrt[3]{2}(4-p)\rho _{\infty }^{2}}\right)
^{1/2}.
\end{equation*}%
Similar to the argument of (\ref{3-6}), we have $\left( t_{0,\infty
}u_{\beta }^{\infty },t_{0,\infty }v_{\beta }^{\infty }\right) \in \mathcal{N%
}_{\beta }^{\left( 1\right) }.$ Hence, we conclude that%
\begin{equation*}
\alpha _{\beta }^{-}=\inf_{\left( u,v\right) \in \mathcal{N}_{\beta
}^{\left( 1\right) }}J_{\beta }\left( u,v\right) \leq J_{\beta }\left(
t_{0,\infty }u_{\beta }^{\infty },t_{0,\infty }v_{\beta }^{\infty }\right)
<\alpha _{\beta ,\infty }^{-}<\frac{p-2}{2p}S_{\lambda ,p}^{2p/\left(
p-2\right) }.
\end{equation*}%
The proof is complete.
\end{proof}

\textbf{We are now ready to prove Theorem \ref{t2}: }By Lemmas \ref{g7}, \ref%
{g4} and the Ekeland variational principle, for each $\beta >\beta
_{0}:=\max \left\{ \bar{\beta}_{0},\beta _{0}^{\infty }\right\} $ there
exists a bounded minimizing sequence $\left\{ \left( u_{n},v_{n}\right)
\right\} \subset \mathcal{N}_{\beta }^{\left( 1\right) }$ such that
\begin{equation*}
J_{\beta }\left( u_{n},v_{n}\right) =\alpha _{\beta }^{-}+o\left( 1\right)
\text{ and }J_{\beta }^{\prime }\left( u_{n},v_{n}\right) =o\left( 1\right)
\text{ in }\mathbf{H}^{-1}.
\end{equation*}%
Moreover, by Lemma \ref{L5.5}, we have%
\begin{equation}
\alpha _{\beta }^{-}<\alpha _{\beta ,\infty }^{-}<\frac{p-2}{2p}S_{\lambda
,p}^{2p/\left( p-2\right) }.  \label{7-15}
\end{equation}%
Since $\left\{ \left( u_{n},v_{n}\right) \right\} $ is bounded, there exists
a convergent subsequence, denoted as $\left\{ \left( u_{n},v_{n}\right)
\right\} $ for convenience, and $\left( u_{\beta },v_{\beta }\right) \in
\mathbf{H}$ such that as $n\rightarrow \infty $,%
\begin{equation}
\begin{array}{l}
\left( u_{n},v_{n}\right) \rightharpoonup \left( u_{\beta },v_{\beta
}\right) \text{ weakly in }\mathbf{H}, \\
\left( u_{n},v_{n}\right) \rightarrow \left( u_{\beta },v_{\beta }\right)
\text{ strongly in }L_{loc}^{p}(\mathbb{R}^{3})\times L_{loc}^{p}(\mathbb{R}%
^{3}), \\
\left( u_{n},v_{n}\right) \rightarrow \left( u_{\beta },v_{\beta }\right)
\text{ a.e. in }\mathbb{R}^{3}.%
\end{array}
\label{7-16}
\end{equation}%
This indicated that $\left( u_{\beta },v_{\beta }\right) $ is a solution of
system $\left( HF_{\beta }\right) .$ Next, we show that%
\begin{equation*}
\left( u_{n},v_{n}\right) \rightarrow \left( u_{\beta },v_{\beta }\right)
\text{ in }\mathbf{H}.
\end{equation*}%
To this end, we suppose the contrary. Then it has
\begin{equation}
\left\Vert \left( u_{\beta },v_{\beta }\right) \right\Vert
^{2}<\liminf_{n\rightarrow \infty }\left\Vert \left( u_{n},v_{n}\right)
\right\Vert ^{2},  \label{15-4}
\end{equation}%
which implies that
\begin{equation}
\left\Vert \left( u_{\beta },v_{\beta }\right) \right\Vert <\left( \frac{3%
\sqrt{3}\pi \left( p-2\right) \lambda ^{3/2}}{16\sqrt[3]{2}(4-p)\rho
_{\infty }^{2}}\right) ^{1/2},  \label{15-1}
\end{equation}%
since $\left\{ \left( u_{n},v_{n}\right) \right\} \subset \mathcal{N}_{\beta
}^{\left( 1\right) }.$ Thus, $\left( u_{\beta },v_{\beta }\right) \in
\mathcal{N}_{\beta }^{\left( 1\right) }\cup \left\{ \left( 0,0\right)
\right\} .$ Let $\left( w_{n},z_{n}\right) =\left( u_{n}-u_{\beta
},v_{n}-v_{\beta }\right) .$ Then by (\ref{7-16})--(\ref{15-1}), there
exists $c_{0}>0$ such that%
\begin{equation}
\left( w_{n},z_{n}\right) \rightharpoonup \left( 0,0\right) \text{ weakly in
}\mathbf{H}  \label{7-17}
\end{equation}%
and%
\begin{equation*}
c_{0}\leq \left\Vert \left( w_{n},z_{n}\right) \right\Vert ^{2}=\left\Vert
\left( u_{n},v_{n}\right) \right\Vert ^{2}-\left\Vert \left( u_{\beta
},v_{\beta }\right) \right\Vert ^{2}+o(1),
\end{equation*}%
which implies that%
\begin{equation}
\left\Vert \left( w_{n},z_{n}\right) \right\Vert <\left( \frac{3\sqrt{3}\pi
\left( p-2\right) \lambda ^{3/2}}{16\sqrt[3]{2}(4-p)\rho _{\infty }^{2}}%
\right) ^{1/2}\text{ for }n\text{ sufficiently large.}  \label{15-5}
\end{equation}%
On the other hand, it follows from the Brezis-Lieb Lemma \cite{BLi} and (\ref%
{7-17}) that%
\begin{equation*}
\left\Vert \left( w_{n},z_{n}\right) \right\Vert ^{2}=\left\Vert \left(
w_{n},z_{n}\right) \right\Vert _{V_{\infty }}^{2}+o\left( 1\right) ,
\end{equation*}%
\begin{equation*}
\int_{\mathbb{R}^{N}}F_{\beta }\left( u_{n},v_{n}\right) dx=\int_{\mathbb{R}%
^{N}}F_{\beta }\left( w_{n},z_{n}\right) dx+\int_{\mathbb{R}^{N}}F_{\beta
}\left( u_{\lambda ,\beta }^{\left( 1\right) },v_{\lambda ,\beta }^{\left(
1\right) }\right) dx+o(1)
\end{equation*}%
and%
\begin{eqnarray*}
\int_{\mathbb{R}^{3}}\rho \left( x\right) \phi _{\rho ,\left(
u_{n},v_{n}\right) }\left( u_{n}^{2}+v_{n}^{2}\right) dx &=&\int_{\mathbb{R}%
^{3}}\rho \left( x\right) \phi _{\rho ,\left( w_{n},z_{n}\right) }\left(
w_{n}^{2}+z_{n}^{2}\right) dx \\
&&+\int_{\mathbb{R}^{3}}\rho \left( x\right) \phi _{\rho ,\left( u_{\beta
},v_{\beta }\right) }\left( u_{\beta }^{2}+v_{\beta }^{2}\right) dx+o(1) \\
&=&\int_{\mathbb{R}^{3}}\rho _{\infty }\phi _{\rho _{\infty },\left(
w_{n},z_{n}\right) }\left( w_{n}^{2}+z_{n}^{2}\right) dx \\
&&+\int_{\mathbb{R}^{3}}\rho \left( x\right) \phi _{\rho ,\left( u_{\beta
},v_{\beta }\right) }\left( u_{\beta }^{2}+v_{\beta }^{2}\right) dx+o(1).
\end{eqnarray*}%
These imply that%
\begin{equation}
\left\Vert \left( w_{n},z_{n}\right) \right\Vert _{V_{\infty }}^{2}+\int_{%
\mathbb{R}^{3}}\rho _{\infty }\phi _{\rho _{\infty },\left(
w_{n},z_{n}\right) }\left( w_{n}^{2}+z_{n}^{2}\right) dx-\int_{\mathbb{R}%
^{3}}F_{\beta }\left( w_{n},z_{n}\right) dx=o\left( 1\right)  \label{15-6}
\end{equation}%
and%
\begin{equation}
J_{\beta }\left( u_{n},v_{n}\right) =J_{V_{\infty },\rho _{\infty },\beta
}^{\infty }\left( w_{n},z_{n}\right) +J_{\beta }\left( u_{\beta },v_{\beta
}\right) +o(1).  \label{15-7}
\end{equation}%
Thus, it follows from (\ref{15-5})--(\ref{15-7}) that $\alpha _{\beta
}^{-}\geq \alpha _{\beta ,\infty }^{-},$ which contradicts to (\ref{7-15}).
Therefore, we conclude that $\left( u_{n},v_{n}\right) \rightarrow \left(
u_{\beta },v_{\beta }\right) $ strongly in $\mathbf{H},$ which implies that $%
\left( u_{\beta },v_{\beta }\right) \in \mathcal{N}_{\beta }^{\left(
1\right) }$ and $J_{\beta }\left( u_{\beta },v_{\beta }\right) =\alpha
_{\beta }^{-},$ and so is $\left( \left\vert u_{\beta }\right\vert
,\left\vert v_{\beta }\right\vert \right) .$ According to Lemma \ref{g7}, we
may assume that $\left( u_{\beta },v_{\beta }\right) $ is a nonnegative
nontrivial critical point of $J_{\beta }$. Since $\alpha _{\beta }^{-}<\frac{%
p-2}{2p}S_{\lambda ,p}^{2p/\left( p-2\right) }$ by (\ref{7-6}), it follows
from Lemma \ref{l5} that $u_{\beta }\neq 0$ and $u_{\beta }\neq 0.$ The
proof is complete.

\textbf{We are now ready to prove Theorem \ref{t3}: }The proof directly
follows from Theorem \ref{t2} and Proposition \ref{t6}.

\section{Nonradial ground state solutions}

\begin{lemma}
\label{L5.2}Assume that $2<p<3$ and conditions $(D1)-(D2)$ and $\left(
D7\right) -\left( D8\right) $ hold. Then for every%
\begin{equation*}
\Lambda \left( \lambda _{0},k_{0}\right) <\beta <\frac{p}{2^{\frac{6-p}{2}}}%
\left( \frac{V_{\infty }}{3-p}\right) ^{3-p}\left( \frac{\rho _{\infty }}{p-2%
}\right) ^{p-2}-1,
\end{equation*}%
the following statements are true. \newline
$(i)$ There exists $K>0$ such that $\inf\limits_{\left( u,v\right) \in
\mathbf{H}_{r}}J_{\beta ,\varepsilon }(u,v)\geq -K$ for all $\varepsilon >0.$%
\newline
$(ii)$ $\inf\limits_{\left( u,v\right) \in \mathbf{H}_{r}}J_{\beta
,\varepsilon }(u,v)<0$ for $\varepsilon >0$ sufficiently small.
\end{lemma}

\begin{proof}
$(i)$ Similar to the argument in \cite[Theorem 4.3]{R1} (or \cite[Lemma 2.5]%
{SW}), by conditions $(D1)-\left( D2\right) ,$ there exists $K>0$ such that
\begin{eqnarray*}
J_{\varepsilon ,\beta }(u,v) &\geq &\frac{1}{2}\Vert \left( u,v\right) \Vert
_{\lambda }^{2}+\frac{\rho _{\min }^{2}}{4}\int_{\mathbb{R}^{3}}\phi
_{u,v}\left( u^{2}+v^{2}\right) dx-\frac{1}{p}\int_{\mathbb{R}^{3}}F_{\beta
}\left( u,v\right) dx \\
&>&-K\text{ for all }\left( u,v\right) \in \mathbf{H}_{r},
\end{eqnarray*}%
which implies that $\inf\limits_{\left( u,v\right) \in \mathbf{H}%
_{r}}J_{\varepsilon ,\beta }(u,v)\geq -K.$

$(ii)$ Since $\left( u_{0},v_{0}\right) \in \mathbf{H}_{r}\setminus \left\{
\left( 0,0\right) \right\} $ is a minimizer of the minimization problem (\ref%
{2-17}) obtained in Proposition \ref{p1} for $\theta =\lambda _{0}$ and $%
k=k_{0},$ for all $\beta >\Lambda \left( \lambda _{0},k_{0}\right) ,$ it
follows from condition $\left( D8\right) $ and Remark \ref{r1.2} that
\begin{eqnarray*}
J_{\varepsilon ,\beta }(u_{0},v_{0}) &=&\frac{1}{2}\int_{\mathbb{R}%
^{3}}\left( |\nabla u_{0}|^{2}+V\left( \varepsilon x\right)
u_{0}^{2}+|\nabla v_{0}|^{2}+V\left( \varepsilon x\right) v_{0}^{2}\right) dx
\\
&&+\frac{1}{4}\int_{\mathbb{R}^{3}}\rho \left( \varepsilon x\right) \phi
_{\rho ,\left( u_{0},v_{0}\right) }\left( u_{0}^{2}+v_{0}^{2}\right) dx-%
\frac{1}{p}\int_{\mathbb{R}^{3}}F_{\beta }\left( u_{0},v_{0}\right) dx \\
&<&\frac{1}{2}\Vert \left( u_{0},v_{0}\right) \Vert _{\lambda _{0}}^{2}+%
\frac{k_{0}}{4}\int_{\mathbb{R}^{3}}\phi _{k_{0},\left( u_{0},v_{0}\right)
}\left( u_{0}^{2}+v_{0}^{2}\right) dx-\frac{1}{p}\int_{\mathbb{R}%
^{3}}F_{\beta }\left( u_{0},v_{0}\right) dx \\
&<&0\text{ for }\varepsilon >0\text{ sufficiently small,}
\end{eqnarray*}%
which implies that $\inf\limits_{\left( u,v\right) \in \mathbf{H}%
_{r}}J_{\beta ,\varepsilon }(u,v)<0$ for $\varepsilon >0$ sufficiently
small. The proof is complete.
\end{proof}

By adopting the multibump technique developed by Ruiz \cite{R2} (also see
\cite[Theorem 4.7]{SW}), we have the following lemma.

\begin{lemma}
\label{L5.1}Assume that $2<p<3$ and conditions $(D1)-\left( D2\right) $ and $%
(D8)$ hold. Then for every%
\begin{equation*}
\Lambda \left( \lambda _{0},k_{0}\right) <\beta <\frac{p}{2^{\frac{6-p}{2}}}%
\left( \frac{V_{\infty }}{3-p}\right) ^{3-p}\left( \frac{\rho _{\infty }}{p-2%
}\right) ^{p-2}-1,
\end{equation*}%
we have
\begin{equation*}
\alpha _{\varepsilon ,\beta }:=\inf_{(u,v)\in \mathbf{H}}J_{\varepsilon
,\beta }(u,v)\rightarrow -\infty \text{ as }\varepsilon \rightarrow 0^{+}.
\end{equation*}
\end{lemma}

\begin{proof}
By conditions $(D1)-(D2)$ and $\left( D8\right) $, there exist a point $%
x_{0}\in \mathbb{R}^{3}$ and two positive constants $\lambda _{0},k_{0}$
such that%
\begin{equation*}
x_{0}\in \mathcal{D}:=\left\{ x\in \mathbb{R}^{3}:V_{\min }<V(x)<\lambda _{0}%
\text{ and }\rho _{\min }<\rho (x)<k_{0}\right\} .
\end{equation*}%
Since $\mathcal{D}$ is an open set, without loss of generality, we may
assume that
\begin{equation*}
B^{3}(x_{0},1)\subset \mathcal{D},
\end{equation*}%
where $B^{3}(x_{0},1)$ is a ball centered at $x_{0}$ with radius $1$ in $%
\mathbb{R}^{3}$. This implies that
\begin{equation}
B^{3}\left( \frac{x_{0}}{\varepsilon },\frac{1}{\varepsilon }\right) \subset
\mathcal{D}_{\varepsilon }:=\left\{ x\in \mathbb{R}^{3}:V_{\min
}<V(\varepsilon x)<\lambda _{0}\text{ and }\rho _{\min }<\rho (\varepsilon
x)<k_{0}\right\}  \label{5.0}
\end{equation}%
for all $\varepsilon >0.$ For any $R>0$, we define a cut--off function $\psi
_{R}\in C^{1}(\mathbb{R}^{2},[0,1])$ by
\begin{equation*}
\psi _{R}(x)=\left\{
\begin{array}{ll}
1, & \text{ if }\left\vert x\right\vert <\frac{R}{2}, \\
0, & \text{ if }\left\vert x\right\vert >R,%
\end{array}%
\right.
\end{equation*}%
and $\left\vert \nabla \psi _{R}\right\vert \leq 1$ in $\mathbb{R}^{2}$. Let
\begin{equation*}
\left( u_{R}(x),v_{R}(x)\right) =\left( \psi _{R}(x)u_{0}\left( x\right)
,\psi _{R}(x)v_{0}\left( x\right) \right) \text{ for any }x\in \mathbb{R}%
^{3},
\end{equation*}%
where $\left( u_{0},v_{0}\right) \in \mathbf{H}_{r}\setminus \left\{ \left(
0,0\right) \right\} $ is a minimizer of the minimization problem (\ref{2-17}%
) obtained in Proposition \ref{p1} for $\theta =\lambda _{0}$ and $k=k_{0}.$
Then $J_{\beta }^{\infty }(u_{0},v_{0})<0$ for all $\beta >\Lambda \left(
\lambda _{0},k_{0}\right) .$ It is obvious that as $n\rightarrow \infty ,$
we have%
\begin{equation}
\int_{\mathbb{R}^{3}}F_{\beta }\left( u_{R},v_{R}\right) dx\rightarrow \int_{%
\mathbb{R}^{2}}F_{\beta }(u_{0},v_{0})dx,  \label{5.1}
\end{equation}%
\begin{equation}
\Vert \left( u_{R},v_{R}\right) \Vert _{\mathbf{H}}^{2}\rightarrow \Vert
(u_{0},v_{0})\Vert _{\mathbf{H}}^{2}  \label{5.2}
\end{equation}%
and
\begin{equation}
\int_{\mathbb{R}^{3}}\phi _{\left( u_{R},v_{R}\right) }\left(
u_{R}^{2}+v_{R}^{2}\right) dx\rightarrow \int_{\mathbb{R}^{3}}\phi
_{(u_{0},v_{0})}\left( u_{0}^{2}+v_{0}^{2}\right) dx,  \label{5.3}
\end{equation}%
where $\int_{\mathbb{R}^{3}}\phi _{(u,v)}\left( u^{2}+v^{2}\right) dx=\int_{%
\mathbb{R}^{3}}\rho \left( x\right) \phi _{\rho ,(u,v)}\left(
u^{2}+v^{2}\right) dx$ with $\rho (x)\equiv 1$ on $\mathbb{R}^{3}$. Since $%
J_{\beta }^{\infty }\in C^{1}(\mathbf{H},\mathbb{R})$ and $J_{\beta
}^{\infty }(u_{0},v_{0})<0$ for all $\beta >\Lambda \left( \lambda
_{0},k_{0}\right) $, it follows from (\ref{5.1})--(\ref{5.3}) that there
exists $R_{0}>0$ such that
\begin{equation}
J_{\beta }^{\infty }(u_{R_{0}},v_{R_{0}})<0\text{ for all }\beta >\Lambda
\left( \lambda _{0},k_{0}\right) .  \label{5.4}
\end{equation}%
Let $0<\varepsilon _{N}<\frac{1}{N^{4}+R_{0}}$ and
\begin{equation*}
u_{R_{0},N}^{(i)}(x)=u_{R_{0}}\left( x+iN^{3}e-\frac{x_{0}}{\varepsilon _{N}}%
\right) \text{ and }v_{R_{0},N}^{(i)}(x)=v_{R_{0}}\left( x+iN^{3}e-\frac{%
x_{0}}{\varepsilon _{N}}\right) ,
\end{equation*}%
for $e\in S^{1}$ and $i=1,2,\ldots ,N$, where $N^{3}>2R_{0}$. Then we have
\begin{equation*}
\text{\textrm{supp}}u_{R_{0},N}^{(i)}(x)\subset B^{3}\left( \frac{x_{0}}{%
\varepsilon _{N}},\frac{1}{\varepsilon _{N}}\right) \text{ and \textrm{supp}}%
v_{R_{0},N}^{(i)}(x)\subset B^{3}\left( \frac{x_{0}}{\varepsilon _{N}},\frac{%
1}{\varepsilon _{N}}\right) \text{ for all }i=1,2,\ldots ,N.
\end{equation*}%
Obviously, $\varepsilon _{N}\rightarrow 0$ as $N\rightarrow \infty $. Then
by (\ref{5.0}), we deduce that
\begin{eqnarray}
&&\int_{\mathbb{R}^{3}}\left[ |\nabla u_{R_{0},N}^{(i)}|^{2}+V\left(
\varepsilon _{N}x\right) \left( u_{R_{0},N}^{(i)}\right) ^{2}+|\nabla
v_{R_{0},N}^{(i)}|^{2}+V\left( \varepsilon _{N}x\right) \left(
v_{R_{0},N}^{(i)}\right) ^{2}\right] dx  \notag \\
&<&\int_{\mathbb{R}^{3}}\left[ |\nabla u_{R_{0},N}^{(i)}|^{2}+\lambda
_{0}\left( u_{R_{0},N}^{(i)}\right) ^{2}+|\nabla
v_{R_{0},N}^{(i)}|^{2}+\lambda _{0}\left( v_{R_{0},N}^{(i)}\right) ^{2}%
\right] dx  \notag \\
&=&\int_{\mathbb{R}^{3}}\left( |\nabla u_{R_{0}}|^{2}+\lambda
_{0}u_{R_{0}}^{2}+|\nabla v_{R_{0}}|^{2}+\lambda _{0}v_{R_{0}}^{2}\right) dx,
\label{5.5}
\end{eqnarray}%
\begin{equation}
\int_{\mathbb{R}^{3}}F_{\beta }\left(
u_{R_{0},N}^{(i)},v_{R_{0},N}^{(i)}\right) dx=\int_{\mathbb{R}^{3}}F_{\beta
}(u_{R_{0}},v_{R_{0}})dx  \label{5.6}
\end{equation}%
and%
\begin{equation}
\int_{\mathbb{R}^{3}}\rho \left( \varepsilon x\right) \phi _{\rho
_{\varepsilon },\left( u_{R_{0},N}^{(i)},v_{R_{0},N}^{(i)}\right) }\left[
\left( u_{R_{0},N}^{(i)}\right) ^{2}+\left( v_{R_{0},N}^{(i)}\right) ^{2}%
\right] dx<k_{0}^{2}\int_{\mathbb{R}^{3}}\phi _{\left(
u_{R_{0}},v_{R_{0}}\right) }\left( u_{R_{0}}^{2}+v_{R_{0}}^{2}\right) dx
\label{5.7}
\end{equation}%
for all $e\in \mathbb{S}^{2}$ and $i=1,2,\ldots ,N.$ Let
\begin{equation*}
w_{R_{0},N}\left( x\right) =\sum_{i=1}^{N}u_{R_{0},N}^{\left( i\right) }%
\text{ and }z_{R_{0},N}\left( x\right) =\sum_{i=1}^{N}v_{R_{0},N}^{\left(
i\right) }.
\end{equation*}%
Observe that $w_{R_{0},N}$ is a sum of translation of $u_{R_{0}}.$ When $%
N^{3}\geq N_{0}^{3}>2R_{0}$, the summands have disjoint support. In such a
case, by (\ref{5.5})--(\ref{5.7}) we have%
\begin{eqnarray}
&&\int_{\mathbb{R}^{3}}\left( |\nabla w_{R_{0},N}|^{2}+V\left( \varepsilon
_{N}x\right) w_{R_{0},N}^{2}+|\nabla z_{R_{0},N}^{2}|^{2}+V\left(
\varepsilon _{N}x\right) z_{R_{0},N}^{2}\right) dx  \notag \\
&=&\sum_{i=1}^{N}\int_{\mathbb{R}^{3}}\left[ |\nabla
u_{R_{0},N}^{(i)}|^{2}+V\left( \varepsilon _{N}x\right) \left(
u_{R_{0},N}^{(i)}\right) ^{2}+|\nabla v_{R_{0},N}^{(i)}|^{2}+V\left(
\varepsilon _{N}\right) \left( v_{R_{0},N}^{(i)}\right) ^{2}\right] dx
\notag \\
&<&\sum_{i=1}^{N}\int_{\mathbb{R}^{3}}\left[ |\nabla
u_{R_{0},N}^{(i)}|^{2}+\lambda _{0}\left( u_{R_{0},N}^{(i)}\right)
^{2}+|\nabla v_{R_{0},N}^{(i)}|^{2}+\lambda _{0}\left(
v_{R_{0},N}^{(i)}\right) ^{2}\right] dx  \notag \\
&=&N\int_{\mathbb{R}^{3}}\left( |\nabla u_{R_{0}}|^{2}+\lambda
_{0}u_{R_{0}}^{2}+|\nabla v_{R_{0}}|^{2}+\lambda _{0}v_{R_{0}}^{2}\right) dx
\label{14}
\end{eqnarray}%
and%
\begin{equation}
\int_{\mathbb{R}^{3}}F_{\beta }\left( w_{R_{0},N},z_{R_{0},N}\right)
dx=\sum_{i=1}^{N}\int_{\mathbb{R}^{3}}F_{\beta }\left(
u_{R_{0},N}^{(i)},v_{R_{0},N}^{(i)}\right) dx=N\int_{\mathbb{R}^{3}}F_{\beta
}(u_{R_{0}},v_{R_{0}})dx  \label{15}
\end{equation}%
and%
\begin{eqnarray*}
&&\int_{\mathbb{R}^{3}}\rho \left( \varepsilon x\right) \phi _{\rho
_{\varepsilon },\left( w_{R_{0},N},z_{R_{0},N}\right) }\left(
w_{R_{0},N}^{2}+z_{R_{0},N}^{2}\right) dx \\
&=&\int_{\mathbb{R}^{3}}\int_{\mathbb{R}^{3}}\frac{\rho \left( \varepsilon
x\right) \rho \left( \varepsilon y\right) w_{R_{0},N}^{2}\left( x\right)
w_{R_{0},N}^{2}\left( y\right) }{\left\vert x-y\right\vert }dxdy+\int_{%
\mathbb{R}^{3}}\int_{\mathbb{R}^{3}}\frac{\rho \left( \varepsilon x\right)
\rho \left( \varepsilon y\right) z_{R_{0},N}^{2}\left( x\right)
z_{R_{0},N}^{2}\left( y\right) }{\left\vert x-y\right\vert }dxdy \\
&&+2\int_{\mathbb{R}^{3}}\int_{\mathbb{R}^{3}}\frac{\rho \left( \varepsilon
x\right) \rho \left( \varepsilon y\right) w_{R_{0},N}^{2}\left( x\right)
z_{R_{0},N}^{2}\left( y\right) }{\left\vert x-y\right\vert }dxdy \\
&=&\sum_{i=1}^{N}\int_{\mathbb{R}^{3}}\int_{\mathbb{R}^{3}}\frac{\rho \left(
\varepsilon x\right) \rho \left( \varepsilon y\right) \left(
u_{R_{0},N}^{\left( i\right) }\right) ^{2}\left( x\right) \left(
u_{R_{0},N}^{\left( i\right) }\right) ^{2}\left( y\right) }{\left\vert
x-y\right\vert }dxdy \\
&&+\sum_{i=1}^{N}\int_{\mathbb{R}^{3}}\int_{\mathbb{R}^{3}}\frac{\rho \left(
\varepsilon x\right) \rho \left( \varepsilon y\right) \left(
v_{R_{0},N}^{\left( i\right) }\right) ^{2}\left( x\right) \left(
v_{R_{0},N}^{\left( i\right) }\right) ^{2}\left( y\right) }{\left\vert
x-y\right\vert }dxdy \\
&&+2\sum_{i\neq j}^{N}\int_{\mathbb{R}^{3}}\int_{\mathbb{R}^{3}}\frac{\rho
\left( \varepsilon x\right) \rho \left( \varepsilon y\right) \left(
u_{R_{0},N}^{\left( i\right) }\right) ^{2}\left( x\right) \left(
v_{R_{0},N}^{\left( j\right) }\right) ^{2}\left( y\right) }{\left\vert
x-y\right\vert }dxdy \\
&<&k_{0}^{2}N\int_{\mathbb{R}^{3}}\phi _{\left( u_{R_{0}},v_{R_{0}}\right)
}\left( u_{R_{0}}^{2}+v_{R_{0}}^{2}\right) dx+2\sum_{i\neq j}^{N}\int_{%
\mathbb{R}^{3}}\int_{\mathbb{R}^{3}}\frac{\rho \left( \varepsilon x\right)
\rho \left( \varepsilon y\right) \left( u_{R_{0},N}^{\left( i\right)
}\right) ^{2}\left( x\right) \left( v_{R_{0},N}^{\left( j\right) }\right)
^{2}\left( y\right) }{\left\vert x-y\right\vert }dxdy.
\end{eqnarray*}%
A straightforward calculation shows that
\begin{eqnarray*}
&&\sum_{i\neq j}^{N}\int_{\mathbb{R}^{3}}\int_{\mathbb{R}^{3}}\frac{\rho
\left( \varepsilon x\right) \rho \left( \varepsilon y\right) \left(
u_{R_{0},N}^{\left( i\right) }\right) ^{2}\left( x\right) \left(
u_{R_{0},N}^{\left( j\right) }\right) ^{2}\left( y\right) }{\left\vert
x-y\right\vert }dxdy \\
&\leq &\rho _{\max }^{2}\sum_{i\neq j}^{N}\int_{\mathbb{R}^{3}}\int_{\mathbb{%
R}^{3}}\frac{\left( u_{R_{0},N}^{\left( i\right) }\right) ^{2}\left(
x\right) \left( u_{R_{0},N}^{\left( j\right) }\right) ^{2}\left( y\right) }{%
4\pi \left\vert x-y\right\vert }dxdy \\
&\leq &\frac{\rho _{\max }^{2}(N^{2}-N)}{N^{3}-2R_{0}}\left( \int_{\mathbb{R}%
^{3}}u_{0}^{2}\left( x\right) dx\right) ^{2}\left( \int_{\mathbb{R}%
^{3}}v_{0}^{2}\left( x\right) dx\right) ^{2},
\end{eqnarray*}%
which implies that%
\begin{equation}
\sum_{i\neq j}^{N}\int_{\mathbb{R}^{3}}\int_{\mathbb{R}^{3}}\frac{\rho
\left( \varepsilon x\right) \rho \left( \varepsilon y\right) \left(
u_{R_{0},N}^{\left( i\right) }\right) ^{2}\left( x\right) \left(
u_{R_{0},N}^{\left( j\right) }\right) ^{2}\left( y\right) }{\left\vert
x-y\right\vert }dxdy\rightarrow 0\text{ as }N\rightarrow \infty .  \label{17}
\end{equation}%
Therefore, it follows from (\ref{5.4}) and (\ref{14})--(\ref{17}) that%
\begin{equation*}
J_{\varepsilon ,\beta }\left( w_{R_{0},N},z_{R_{0},N}\right) <NJ_{\beta
}^{\infty }\left( u_{R_{0}},v_{R_{0}}\right) +C_{0}\ \text{for some }C_{0}>0
\end{equation*}%
and
\begin{equation*}
J_{\varepsilon ,\beta }\left( w_{R_{0},N},z_{R_{0},N}\right) \rightarrow
-\infty \text{ as }N\rightarrow \infty ,
\end{equation*}%
which implies that $\alpha _{\varepsilon ,\beta }:=\inf_{(u,v)\in \mathbf{H}%
}J_{\varepsilon ,\beta }(u,v)\rightarrow -\infty $ as $\varepsilon
\rightarrow 0^{+}.$ This completes the proof.
\end{proof}

Define
\begin{equation*}
\delta _{\varepsilon ,\beta }:=\inf\limits_{\left( u,v\right) \in \mathbf{H}%
_{r}}J_{\varepsilon ,\beta }(u,v).
\end{equation*}%
Then by Lemmas \ref{L5.2} and \ref{L5.1}, we have the following result.

\begin{lemma}
\label{L5.3}Assume that $2<p<3$ and conditions $(D1)-(D2)$ and $\left(
D7\right) -\left( D8\right) $ hold. Then for every%
\begin{equation*}
\Lambda \left( \lambda _{0},k_{0}\right) <\beta <\frac{p}{2^{\frac{6-p}{2}}}%
\left( \frac{V_{\infty }}{3-p}\right) ^{3-p}\left( \frac{\rho _{\infty }}{p-2%
}\right) ^{p-2}-1,
\end{equation*}%
there exists $K>0$ such that
\begin{equation*}
\alpha _{\varepsilon ,\beta }<-K\leq \delta _{\varepsilon ,\beta }\text{ for
}\varepsilon >0\text{ sufficiently small.}
\end{equation*}
\end{lemma}

\textbf{We are now ready to prove Theorem \ref{th1.4}.} By conditions $%
(D1)-(D2)$ and $(D8)$, there exist a point $x_{0}\in \mathbb{R}^{3}$ and two
positive constants $\lambda _{0},k_{0}$ such that%
\begin{equation*}
x_{0}\in \mathcal{D}:=\left\{ x\in \mathbb{R}^{3}:V_{\min }<V(x)<\lambda _{0}%
\text{ and }\rho _{\min }<\rho (x)<k_{0}\right\} .
\end{equation*}%
Let $\left( u_{0},v_{0}\right) \in \mathbf{H}_{r}\setminus \left\{ \left(
0,0\right) \right\} $ is a minimizer of the minimization problem (\ref{2-17}%
) obtained in Proposition \ref{p1} for $\theta =\lambda _{0}$ and $k=k_{0}$.
Define $u_{\varepsilon }(x)=u_{0}(x-\frac{x_{0}}{\varepsilon })$ and $%
v_{\varepsilon }(x)=v_{0}(x-\frac{x_{0}}{\varepsilon })$, where $x_{0}\in
\mathbb{R}^{3}$ as in condition $\left( D8\right) .$ Then by Remark \ref%
{r1.2}, we know that when $V(x)$ and $\rho \left( x\right) $ are replaced by
$V(\varepsilon x)$ and $\rho \left( \varepsilon x\right) ,$ respectively,
condition $(D3)$ still holds for $\varepsilon >0$ sufficiently small.
Therefore, by Theorem \ref{t1}, system $(HF_{\varepsilon ,\beta })$ admits a
vectorial ground state solution $\left( u_{\varepsilon ,\beta }^{\left(
1\right) },v_{\varepsilon ,\beta }^{\left( 1\right) }\right) \in \mathbf{H}$
satisfying $J_{\varepsilon ,\beta }\left( u_{\varepsilon ,\beta }^{\left(
1\right) },v_{\varepsilon ,\beta }^{\left( 1\right) }\right) <J_{\beta
}^{\infty }(u_{0},v_{0})<0.$ Moreover, it follows from Lemma \ref{L5.3} that%
\begin{equation*}
J_{\varepsilon ,\beta }\left( u_{\varepsilon ,\beta }^{\left( 1\right)
},v_{\varepsilon ,\beta }^{\left( 1\right) }\right) =\alpha _{\varepsilon
,\beta }<-K\leq \delta _{\varepsilon ,\beta }\text{ for }\varepsilon >0\text{
sufficiently small,}
\end{equation*}%
which implies that $\left( u_{\varepsilon ,\beta }^{\left( 1\right)
},v_{\varepsilon ,\beta }^{\left( 1\right) }\right) $ is a nonradial
vectorial ground state solution of system $(HF_{\varepsilon ,\beta })$. The
proof is complete.

\section{Appendix}

\begin{proposition}
\label{AT-1}Assume that $2<p<3$ and $\theta ,k>0.$ Then the following
statements are true.\newline
$\left( i\right) $ $\Lambda \left( \theta ,k\right) \geq \frac{p}{2}\left(
\frac{\lambda }{3-p}\right) ^{3-p}\left( \frac{k}{p-2}\right) ^{p-2}-1;$%
\newline
$\left( ii\right) $ $\Lambda \left( \theta ,k\right) $ is achieved, i.e.
there exists $\left( u_{0},v_{0}\right) \in \mathbf{H}_{r}\setminus \left\{
\left( 0,0\right) \right\} $ such that
\begin{equation*}
\Lambda \left( \theta ,k\right) =\frac{\frac{1}{2}\left\Vert \left(
u_{0},v_{0}\right) \right\Vert _{\theta }^{2}+\frac{k}{4}\int_{\mathbb{R}%
^{3}}\phi _{k,\left( u_{0},v_{0}\right) }\left( u_{0}^{2}+v_{0}^{2}\right)
dx-\frac{1}{p}\int_{\mathbb{R}^{3}}(\left\vert u_{0}\right\vert
^{p}+\left\vert v_{0}\right\vert ^{p})dx}{\frac{2}{p}\int_{\mathbb{R}%
^{3}}\left\vert u_{0}\right\vert ^{\frac{^{p}}{2}}\left\vert
v_{0}\right\vert ^{\frac{p}{2}}dx}>0.
\end{equation*}
\end{proposition}

\begin{proof}
$\left( i\right) $ Similar to (\ref{3-4})--(\ref{3-5}), we have
\begin{equation}
\frac{k}{2}\int_{\mathbb{R}^{3}}(\left\vert u\right\vert
^{3}+v^{2}\left\vert u\right\vert )dx\leq \frac{1}{2}\int_{\mathbb{R}%
^{3}}\left\vert \nabla u\right\vert ^{2}dx+\frac{k}{8}\int_{\mathbb{R}%
^{3}}\phi _{k,\left( u,v\right) }\left( u^{2}+v^{2}\right) dx  \label{A-2}
\end{equation}%
and%
\begin{equation}
\frac{k}{2}\int_{\mathbb{R}^{3}}(u^{2}\left\vert v\right\vert +\left\vert
v\right\vert ^{3})dx\leq \frac{1}{2}\int_{\mathbb{R}^{3}}\left\vert \nabla
v\right\vert ^{2}dx+\frac{k}{8}\int_{\mathbb{R}^{3}}\phi _{k,\left(
u,v\right) }\left( u^{2}+v^{2}\right) dx  \label{A-3}
\end{equation}%
for all $\left( u,v\right) \in \mathbf{H}_{r}\setminus \left\{ \left(
0,0\right) \right\} .$ Then it follows from (\ref{A-2})--(\ref{A-3}) that%
\begin{eqnarray*}
&&\frac{\frac{1}{2}\left\Vert \left( u,v\right) \right\Vert _{\theta }^{2}+%
\frac{k}{4}\int_{\mathbb{R}^{3}}\phi _{k,\left( u,v\right) }\left(
u^{2}+v^{2}\right) dx-\frac{1}{p}\int_{\mathbb{R}^{3}}(\left\vert
u\right\vert ^{p}+\left\vert v\right\vert ^{p})dx}{\frac{2}{p}\int_{\mathbb{R%
}^{3}}\left\vert u\right\vert ^{\frac{^{p}}{2}}\left\vert v\right\vert ^{%
\frac{p}{2}}dx} \\
&\geq &\frac{\frac{1}{2}\int_{\mathbb{R}^{3}}\theta \left(
u^{2}+v^{2}\right) dx+\frac{k}{2}\int_{\mathbb{R}^{3}}(\left\vert
u\right\vert ^{3}+\left\vert v\right\vert ^{3})dx-\frac{1}{p}\int_{\mathbb{R}%
^{3}}(\left\vert u\right\vert ^{p}+\left\vert v\right\vert ^{p})dx}{\frac{2}{%
p}\int_{\mathbb{R}^{3}}\left\vert u\right\vert ^{\frac{^{p}}{2}}\left\vert
v\right\vert ^{\frac{p}{2}}dx} \\
&\geq &\frac{\left( \frac{d_{\theta ,k}}{2p}-\frac{1}{p}\right) \int_{%
\mathbb{R}^{3}}(\left\vert u\right\vert ^{p}+\left\vert v\right\vert ^{p})dx%
}{\frac{1}{p}\int_{\mathbb{R}^{3}}(\left\vert u\right\vert ^{p}+\left\vert
v\right\vert ^{p})dx}=\frac{d_{\theta ,k}}{2}-1,
\end{eqnarray*}%
where $d_{\theta ,k}:=p\left( \frac{\theta }{3-p}\right) ^{3-p}\left( \frac{k%
}{p-2}\right) ^{p-2}.$ This shows that%
\begin{eqnarray*}
\Lambda \left( \theta ,k\right)  &:&=\inf_{u\in \mathbf{H}_{r}\setminus
\left\{ \left( 0,0\right) \right\} }\frac{\frac{1}{2}\left\Vert \left(
u,v\right) \right\Vert _{\theta }^{2}+\frac{k}{4}\int_{\mathbb{R}^{3}}\phi
_{k,\left( u,v\right) }\left( u^{2}+v^{2}\right) dx-\frac{1}{p}\int_{\mathbb{%
R}^{3}}(\left\vert u\right\vert ^{p}+\left\vert v\right\vert ^{p})dx}{\frac{2%
}{p}\int_{\mathbb{R}^{3}}\left\vert u\right\vert ^{\frac{^{p}}{2}}\left\vert
v\right\vert ^{\frac{p}{2}}dx} \\
&\geq &\frac{p}{2}\left( \frac{\theta }{3-p}\right) ^{3-p}\left( \frac{k}{p-2%
}\right) ^{p-2}-1.
\end{eqnarray*}%
$\left( ii\right) $ Let $\left\{ \left( u_{n},v_{n}\right) \right\} \subset
\mathbf{H}_{r}\setminus \{\left( 0,0\right) \}$ be a minimum sequence of (%
\ref{2-17}). First of all, we claim that $\{\left( u_{n},v_{n}\right) \}$ is
bounded in $H_{r}$. Suppose on the contrary. Then $\left\Vert
(u_{n},v_{n})\right\Vert _{\theta }\rightarrow \infty $ as $n\rightarrow
\infty $. Since $\Lambda \left( \lambda ,k\right) >-\infty $ and%
\begin{equation*}
\frac{\frac{1}{2}\left\Vert \left( u,v\right) \right\Vert _{\theta }^{2}+%
\frac{k}{4}\int_{\mathbb{R}^{3}}\phi _{k,\left( u,v\right) }\left(
u^{2}+v^{2}\right) dx-\frac{1}{p}\int_{\mathbb{R}^{3}}(\left\vert
u\right\vert ^{p}+\left\vert v\right\vert ^{p})dx}{\frac{2}{p}\int_{\mathbb{R%
}^{3}}\left\vert u\right\vert ^{\frac{^{p}}{2}}\left\vert v\right\vert ^{%
\frac{p}{2}}dx}=\Lambda \left( \lambda ,k\right) +o\left( 1\right) ,
\end{equation*}%
there exists $C_{1}>\max \left\{ -1,\Lambda \left( \theta ,k\right) \right\}
$ such that%
\begin{eqnarray}
\widetilde{J}\left( u_{n},v_{n}\right) := &&\frac{1}{2}\left\Vert \left(
u_{n},v_{n}\right) \right\Vert _{\theta }^{2}+\frac{k}{4}\int_{\mathbb{R}%
^{3}}\phi _{k,\left( u_{n},v_{n}\right) }\left( u_{n}^{2}+v_{n}^{2}\right)
dx-\frac{1}{p}\int_{\mathbb{R}^{3}}(\left\vert u_{n}\right\vert
^{p}+\left\vert v_{n}\right\vert ^{p})dx  \notag \\
&&-\frac{2C_{1}}{p}\int_{\mathbb{R}^{3}}\left\vert u_{n}\right\vert ^{\frac{%
^{p}}{2}}\left\vert v_{n}\right\vert ^{\frac{p}{2}}dx<0  \label{A-4}
\end{eqnarray}%
for $n$ sufficiently large. Similar to (\ref{3-4})--(\ref{3-5}), we also have%
\begin{equation}
\frac{k}{\sqrt{8}}\int_{\mathbb{R}^{3}}(\left\vert u\right\vert
^{3}+v^{2}\left\vert u\right\vert )dx\leq \frac{1}{4}\int_{\mathbb{R}%
^{3}}\left\vert \nabla u\right\vert ^{2}dx+\frac{k}{8}\int_{\mathbb{R}%
^{3}}\phi _{k,\left( u,v\right) }\left( u^{2}+v^{2}\right) dx  \label{A-5}
\end{equation}%
and%
\begin{equation}
\frac{k}{\sqrt{8}}\int_{\mathbb{R}^{3}}(u^{2}\left\vert v\right\vert
+\left\vert v\right\vert ^{3})dx\leq \frac{1}{4}\int_{\mathbb{R}%
^{3}}\left\vert \nabla v\right\vert ^{2}dx+\frac{k}{8}\int_{\mathbb{R}%
^{3}}\phi _{k,\left( u,v\right) }\left( u^{2}+v^{2}\right) dx  \label{A-6}
\end{equation}%
for all $\left( u,v\right) \in \mathbf{H}_{r}.$ Then it follows from (\ref%
{A-4})--(\ref{A-6}) that%
\begin{equation*}
\widetilde{J}\left( u_{n},v_{n}\right) \geq \frac{1}{4}\left\Vert \left(
u_{n},v_{n}\right) \right\Vert _{\theta }^{2}+\frac{k}{8}\int_{\mathbb{R}%
^{3}}\phi _{k,\left( u_{n},v_{n}\right) }\left( u_{n}^{2}+v_{n}^{2}\right)
dx+\int_{\mathbb{R}^{3}}(f_{\theta ,k}\left( u_{n}\right) +f_{\theta
,k}\left( v_{n}\right) )dx,
\end{equation*}%
where $f_{\theta ,k}\left( s\right) :=\frac{\theta }{4}s^{2}+\frac{k}{\sqrt{8%
}}s^{3}-\frac{1+C_{1}}{p}s^{p}$ for $s>0.$ It is clear that $f_{\theta ,k}$
is positive for $s\rightarrow 0^{+}$ or $s\rightarrow \infty ,$ since $2<p<3$
and $\theta ,k>0.$ Define%
\begin{equation*}
m_{\theta ,k}:=\inf_{s>0}f_{\theta ,k}(s).
\end{equation*}%
If $m_{\theta ,k}\geq 0,$ then by (\ref{A-4}) we have
\begin{equation*}
0\geq \widetilde{J}\left( u_{n},v_{n}\right) \geq \frac{1}{4}\left\Vert
\left( u_{n},v_{n}\right) \right\Vert _{\theta }^{2}+\frac{k}{8}\int_{%
\mathbb{R}^{3}}\phi _{k,\left( u_{n},v_{n}\right) }(u_{n}^{2}+v_{n}^{2})dx>0,
\end{equation*}%
which is a contradiction. We now assume that $m_{\theta ,k}<0.$ Then the set
$\left\{ s>0:f_{\theta ,k}\left( s\right) <0\right\} $ is an open interval $%
\left( s_{1},s_{2}\right) $ with $s_{1}>0.$ Note that constants $%
s_{1},s_{2},m_{\theta ,k}$ depend on $p,\theta ,k$ and $C_{1}$. Thus, there
holds
\begin{eqnarray}
\widetilde{J}\left( u_{n},v_{n}\right)  &\geq &\frac{1}{4}\left\Vert \left(
u_{n},v_{n}\right) \right\Vert _{\theta }^{2}+\frac{k}{8}\int_{\mathbb{R}%
^{3}}\phi _{k,\left( u_{n},v_{n}\right) }(u_{n}^{2}+v_{n}^{2})dx+\int_{%
\mathbb{R}^{3}}(f_{\theta ,k}\left( u_{n}\right) +f_{\theta ,k}\left(
v_{n}\right) )dx  \notag \\
&\geq &\frac{1}{4}\left\Vert \left( u_{n},v_{n}\right) \right\Vert _{\theta
}^{2}+\frac{k}{8}\int_{\mathbb{R}^{3}}\phi _{k,\left( u_{n},v_{n}\right)
}(u_{n}^{2}+v_{n}^{2})dx+\int_{D_{n}^{\left( 1\right) }}f_{\beta }\left(
u_{n}\right) dx+\int_{D_{n}^{\left( 2\right) }}f_{\beta }\left( v_{n}\right)
dx  \notag \\
&\geq &\frac{1}{4}\left\Vert \left( u_{n},v_{n}\right) \right\Vert _{\theta
}^{2}+\frac{k}{8}\int_{\mathbb{R}^{3}}\phi _{k,\left( u_{n},v_{n}\right)
}(u_{n}^{2}+v_{n}^{2})dx-\left\vert m_{\theta ,k}\right\vert \left(
\left\vert D_{n}^{\left( 1\right) }\right\vert +\left\vert D_{n}^{\left(
2\right) }\right\vert \right) ,  \label{A-7}
\end{eqnarray}%
where the sets $D_{n}^{\left( 1\right) }:=\left\{ x\in \mathbb{R}%
^{3}:u_{n}\left( x\right) \in \left( s_{1},s_{2}\right) \right\} $ and $%
D_{n}^{\left( 2\right) }:=\left\{ x\in \mathbb{R}^{3}:v_{n}\left( x\right)
\in \left( s_{1},s_{2}\right) \right\} .$ It follows from (\ref{A-4}) and (%
\ref{A-7}) that
\begin{equation}
\left\vert m_{\theta ,k}\right\vert \left( \left\vert D_{n}^{\left( 1\right)
}\right\vert +\left\vert D_{n}^{\left( 2\right) }\right\vert \right) >\frac{1%
}{4}\left\Vert \left( u_{n},v_{n}\right) \right\Vert _{\theta }^{2},
\label{A-12}
\end{equation}%
which implies that $\left\vert D_{n}^{\left( 1\right) }\right\vert
+\left\vert D_{n}^{\left( 2\right) }\right\vert \rightarrow \infty $ as $%
n\rightarrow \infty ,$ since $\left\Vert (u_{n},v_{n})\right\Vert _{\theta
}\rightarrow \infty $ as $n\rightarrow \infty .$ Moreover, since $%
D_{n}^{\left( 1\right) }$ and $D_{n}^{\left( 2\right) }$ are spherically
symmetric, we define $\rho _{n}^{\left( i\right) }:=\sup \left\{ \left\vert
x\right\vert :x\in D_{n}^{\left( i\right) }\right\} $ for $i=1,2.$ Then we
can take $x^{\left( i\right) }\in \mathbb{R}^{3}(i=1,2)$ such that $%
\left\vert x^{\left( i\right) }\right\vert =\rho _{n}^{\left( i\right) }.$
Clearly, $u_{n}\left( x^{\left( 1\right) }\right) =v_{n}\left( x^{\left(
2\right) }\right) =s_{1}>0.$ Recall the following Strauss's inequality by
Strauss \cite{S}:%
\begin{equation}
\left\vert z\left( x\right) \right\vert \leq c_{0}\left\vert x\right\vert
^{-1}\left\Vert z\right\Vert _{H^{1}}\text{ for all }z\in H_{r}^{1}(\mathbb{R%
}^{3})  \label{A-13}
\end{equation}%
for some $c_{0}>0.$ Thus, by (\ref{A-12})--(\ref{A-13}), we have%
\begin{equation*}
0<s_{1}=u_{n}\left( x^{\left( 1\right) }\right) <c_{0}\left( \rho
_{n}^{\left( 1\right) }\right) ^{-1}\left\Vert u_{n}\right\Vert _{H^{1}}\leq
2c_{0}\left\vert m_{\theta ,k}\right\vert ^{1/2}\left( \rho _{n}^{\left(
1\right) }\right) ^{-1}\left( \left\vert D_{n}^{\left( 1\right) }\right\vert
+\left\vert D_{n}^{\left( 2\right) }\right\vert \right) ^{1/2}
\end{equation*}%
and%
\begin{equation*}
0<s_{1}=v_{n}\left( x^{\left( 2\right) }\right) <c_{0}\left( \rho
_{n}^{\left( 2\right) }\right) ^{-1}\left\Vert v_{n}\right\Vert _{H^{1}}\leq
2c_{0}\left\vert m_{\theta ,k}\right\vert ^{1/2}\left( \rho _{n}^{\left(
2\right) }\right) ^{-1}\left( \left\vert D_{n}^{\left( 1\right) }\right\vert
+\left\vert D_{n}^{\left( 2\right) }\right\vert \right) ^{1/2}.
\end{equation*}%
These imply that%
\begin{equation}
c_{i}\rho _{n}^{\left( i\right) }\leq \left( \left\vert D_{n}^{\left(
1\right) }\right\vert +\left\vert D_{n}^{\left( 2\right) }\right\vert
\right) ^{1/2}\text{ for some }c_{i}>0\text{ and }i=1,2.  \label{A-14}
\end{equation}%
On the other hand, since $\widetilde{J}\left( u_{n},v_{n}\right) \leq 0,$ we
have
\begin{eqnarray*}
&&\frac{8}{k^{2}}\left\vert m_{\beta }\right\vert \left( \left\vert
D_{n}^{\left( 1\right) }\right\vert +\left\vert D_{n}^{\left( 2\right)
}\right\vert \right) \geq \int_{\mathbb{R}^{3}}\phi _{(u_{n},v_{n})}\left(
u_{n}^{2}+v_{n}^{2}\right) dx \\
&=&\int_{\mathbb{R}^{3}}\int_{\mathbb{R}^{3}}\frac{u_{n}^{2}(x)u_{n}^{2}(y)}{%
|x-y|}dxdy+\int_{\mathbb{R}^{3}}\int_{\mathbb{R}^{3}}\frac{%
v_{n}^{2}(x)v_{n}^{2}(y)}{|x-y|}dxdy+2\int_{\mathbb{R}^{3}}\int_{\mathbb{R}%
^{3}}\frac{u_{n}^{2}(x)v_{n}^{2}(y)}{|x-y|}dxdy \\
&\geq &\int_{D_{n}^{\left( 1\right) }}\int_{D_{n}^{\left( 1\right) }}\frac{%
u_{n}^{2}(x)u_{n}^{2}(y)}{|x-y|}dxdy+\int_{D_{n}^{\left( 2\right)
}}\int_{D_{n}^{\left( 2\right) }}\frac{v_{n}^{2}(x)v_{n}^{2}(y)}{|x-y|}%
dxdy+2\int_{D_{n}^{\left( 2\right) }}v_{n}^{2}(y)\left( \int_{D_{n}^{\left(
1\right) }}\frac{u_{n}^{2}(x)}{|x-y|}dx\right) dy \\
&\geq &s_{1}^{4}\left( \frac{\left\vert D_{n}^{\left( 1\right) }\right\vert
^{2}}{2\rho _{n}^{\left( 1\right) }}+\frac{\left\vert D_{n}^{\left( 2\right)
}\right\vert ^{2}}{2\rho _{n}^{\left( 2\right) }}\right)
+2\int_{D_{n}^{\left( 2\right) }}v_{n}^{2}(y)\left( \int_{D_{n}^{\left(
1\right) }}\frac{u_{n}^{2}(x)}{|x|+\left\vert y\right\vert }dx\right) dy \\
&\geq &s_{1}^{4}\left( \frac{\left\vert D_{n}^{\left( 1\right) }\right\vert
^{2}}{2\rho _{n}^{\left( 1\right) }}+\frac{\left\vert D_{n}^{\left( 2\right)
}\right\vert ^{2}}{2\rho _{n}^{\left( 2\right) }}\right) +\frac{%
2s_{1}^{4}\left\vert D_{n}^{\left( 1\right) }\right\vert \left\vert
D_{n}^{\left( 2\right) }\right\vert }{\rho _{n}^{\left( 1\right) }+\rho
_{n}^{\left( 2\right) }} \\
&\geq &s_{1}^{4}\left( \frac{\left\vert D_{n}^{\left( 1\right) }\right\vert
^{2}}{2\rho _{n}^{\left( 1\right) }}+\frac{\left\vert D_{n}^{\left( 2\right)
}\right\vert ^{2}}{2\rho _{n}^{\left( 2\right) }}+\frac{2\left\vert
D_{n}^{\left( 1\right) }\right\vert \left\vert D_{n}^{\left( 2\right)
}\right\vert }{\rho _{n}^{\left( 1\right) }+\rho _{n}^{\left( 2\right) }}%
\right) ,
\end{eqnarray*}%
and together with (\ref{A-14}), we further have%
\begin{eqnarray*}
\frac{8}{k^{2}s_{1}^{4}}\left\vert m_{\theta ,k}\right\vert \left(
\left\vert D_{n}^{\left( 1\right) }\right\vert +\left\vert D_{n}^{\left(
2\right) }\right\vert \right)  &\geq &\frac{c_{1}\left\vert D_{n}^{\left(
1\right) }\right\vert ^{2}}{2\left( \left\vert D_{n}^{\left( 1\right)
}\right\vert +\left\vert D_{n}^{\left( 2\right) }\right\vert \right) ^{1/2}}+%
\frac{c_{2}\left\vert D_{n}^{\left( 2\right) }\right\vert ^{2}}{2\left(
\left\vert D_{n}^{\left( 1\right) }\right\vert +\left\vert D_{n}^{\left(
2\right) }\right\vert \right) ^{1/2}} \\
&&+\frac{2\left\vert D_{n}^{\left( 1\right) }\right\vert \left\vert
D_{n}^{\left( 2\right) }\right\vert }{(c_{1}^{-1}+c_{2}^{-1})\left(
\left\vert D_{n}^{\left( 1\right) }\right\vert +\left\vert D_{n}^{\left(
2\right) }\right\vert \right) ^{1/2}} \\
&\geq &\min \left\{ \frac{c_{1}}{2},\frac{c_{2}}{2}%
,(c_{1}^{-1}+c_{2}^{-1})^{-1}\right\} \left( \left\vert D_{n}^{\left(
1\right) }\right\vert +\left\vert D_{n}^{\left( 2\right) }\right\vert
\right) ^{3/2},
\end{eqnarray*}%
which implies that for all $n,$
\begin{equation*}
\left\vert D_{n}^{\left( 1\right) }\right\vert +\left\vert D_{n}^{\left(
2\right) }\right\vert \leq M\text{ for some }M>0.
\end{equation*}%
This contradicts with $\left\vert D_{n}^{\left( 1\right) }\right\vert
+\left\vert D_{n}^{\left( 2\right) }\right\vert \rightarrow \infty $ as $%
n\rightarrow \infty .$ Hence, we conclude that $\left\{ \left(
u_{n},v_{n}\right) \right\} $ is bounded in $\mathbf{H}_{r}.$

Assume that $\left( u_{n},v_{n}\right) \rightharpoonup \left(
u_{0},v_{0}\right) $ in $\mathbf{H}_{r}.$ Next, we prove that $\left(
u_{n},v_{n}\right) \rightarrow \left( u_{0},v_{0}\right) $ strongly in $%
\mathbf{H}_{r}.$ Suppose on contrary. Then there holds%
\begin{equation*}
\left\Vert \left( u_{0},v_{0}\right) \right\Vert _{\theta }^{2}<\liminf
\left\Vert \left( u_{n},v_{n}\right) \right\Vert _{\theta }^{2},
\end{equation*}%
Since $\mathbf{H}_{r}\hookrightarrow L^{r}(\mathbb{R}^{3})\times L^{r}(%
\mathbb{R}^{3})$ is compact for $2<r<6$ (see \cite{S}), we have%
\begin{equation*}
\int_{\mathbb{R}^{3}}(\left\vert u_{n}\right\vert ^{p}+\left\vert
v_{n}\right\vert ^{p})dx\rightarrow \int_{\mathbb{R}^{3}}(\left\vert
u_{0}\right\vert ^{p}+\left\vert v_{0}\right\vert ^{p})dx
\end{equation*}%
and%
\begin{equation*}
\int_{\mathbb{R}^{3}}\left\vert u_{n}\right\vert ^{\frac{^{p}}{2}}\left\vert
v_{n}\right\vert ^{\frac{p}{2}}dx\rightarrow \int_{\mathbb{R}^{3}}\left\vert
u_{0}\right\vert ^{\frac{^{p}}{2}}\left\vert v_{0}\right\vert ^{\frac{p}{2}%
}dx.
\end{equation*}%
Moreover, it follows from Ruiz \cite[Lemma 2.1]{R1} that
\begin{equation*}
\int_{\mathbb{R}^{3}}\phi _{(u_{n},v_{n})}\left( u_{n}^{2}+v_{n}^{2}\right)
dx\rightarrow \int_{\mathbb{R}^{3}}\phi _{(u_{0},v_{0})}\left(
u_{0}^{2}+v_{0}^{2}\right) dx.
\end{equation*}%
These imply that%
\begin{equation*}
\frac{\frac{1}{2}\left\Vert \left( u_{0},v_{0}\right) \right\Vert _{\theta
}^{2}+\frac{k}{4}\int_{\mathbb{R}^{3}}\phi _{k,\left( u_{0},v_{0}\right)
}\left( u_{0}^{2}+v_{0}^{2}\right) dx-\frac{1}{p}\int_{\mathbb{R}%
^{3}}(\left\vert u_{0}\right\vert ^{p}+\left\vert v_{0}\right\vert ^{p})dx}{%
\frac{2}{p}\int_{\mathbb{R}^{3}}\left\vert u_{0}\right\vert ^{\frac{^{p}}{2}%
}\left\vert v_{0}\right\vert ^{\frac{p}{2}}dx}<\Lambda \left( \theta
,k\right) ,
\end{equation*}%
which is a contradiction. Hence, we conclude that $\left( u_{n},v_{n}\right)
\rightarrow \left( u_{0},v_{0}\right) $ strongly in $\mathbf{H}_{r}$ and $%
\left( u_{0},v_{0}\right) \in \mathbf{H}_{r}\setminus \left\{ \left(
0,0\right) \right\} .$ Therefore, $\Lambda \left( \theta ,k\right) $ is
achieved. This completes the proof.
\end{proof}

\section*{Acknowledgments}

J. Sun was supported by the National Natural Science Foundation of China
(Grant No. 12371174) and Shandong Provincial Natural Science Foundation
(Grant No. ZR2020JQ01). T.F. Wu was supported by the National Science and
Technology Council, Taiwan (Grant No. 112-2115-M-390-001-MY3).

\section*{Data availability}
Data sharing not applicable to this article as no datasets were generated or analysed during the current study.

\section*{Competing Interests}
On behalf of all authors, the corresponding author Tsung-fang Wu states that there is no conflict of interest.

\end{document}